\documentclass[11pt, a4paper]{amsart}

\usepackage{amscd,verbatim}
\usepackage{amssymb}

\usepackage{bm}

\usepackage{amssymb, amsmath}
\usepackage[all]{xy}
\usepackage[inline]{enumitem}
\usepackage{hyphenat}
\usepackage[colorlinks,linkcolor=blue,citecolor=blue,urlcolor=red]{hyperref}
\usepackage[dvipsnames]{xcolor}
\usepackage{mathtools}
\usepackage{tikz-cd}
\usetikzlibrary{cd,decorations.pathreplacing}
\usepackage[labelformat=empty]{caption}
\usepackage{xfrac}
 \usepackage[ left=3cm, right=3cm, top = 2.8cm, bottom = 2.8cm ]{geometry}
\usepackage{stmaryrd}
\usepackage{bbm}

\makeatletter
\def\@tocline#1#2#3#4#5#6#7{\relax
  \ifnum #1>\c@tocdepth 
  \else
    \par \addpenalty\@secpenalty\addvspace{#2}%
    \begingroup \hyphenpenalty\@M
    \@ifempty{#4}{%
      \@tempdima\csname r@tocindent\number#1\endcsname\relax
    }{%
      \@tempdima#4\relax
    }%
    \parindent\z@ \leftskip#3\relax \advance\leftskip\@tempdima\relax
    \rightskip\@pnumwidth plus4em \parfillskip-\@pnumwidth
    #5\leavevmode\hskip-\@tempdima
      \ifcase #1
      \or\or \hskip 2em \or \hskip 2em \else \hskip 3em \fi%
      #6\nobreak\relax
    \dotfill\hbox to\@pnumwidth{\@tocpagenum{#7}}\par
    \nobreak
    \endgroup
  \fi}
\makeatother

\newcommand{\A}{\mathbf{A}}

\newcommand{\C}{\mathbf{C}}

\newcommand{\G}{\mathbf{G}}

\renewcommand{\P}{\mathbf{P}}
\newcommand{\Q}{\mathbb{Q}}

\newcommand{\Z}{\mathbb{Z}}
\newcommand{\F}{\mathbb{F}}

\newcommand{\sN}{\mathcal{N}}
\newcommand{\sM}{{\mathscr M}}
\newcommand{\sO}{\mathcal{O}}
\newcommand{\sP}{\mathcal{P}}

\newcommand{\sU}{\mathcal{U}}
\newcommand{\sV}{\mathcal{V}}

\newcommand{\sX}{\mathcal{X}}
\newcommand{\sY}{\mathcal{Y}}

\newcommand{\fm}{\mathfrak{m}}

\newcommand{\fp}{\mathfrak{p}}
\newcommand{\fq}{\mathfrak{q}}
\newcommand{\fr}{\mathfrak{r}}

\newcommand{\Cone}{\operatorname{Cone}}

\newcommand{\ul}[1]{{\underline{#1}}}

\newcommand{\Set}{{\operatorname{\mathbf{Set}}}}
\newcommand{\Cpx}{{\operatorname{\mathbf{Cpx}}}}

\newcommand{\Hom}{\operatorname{Hom}}

\newcommand{\Ker}{\operatorname{Ker}}

\newcommand{\Tr}{\operatorname{Tr}}

\newcommand{\Spec}{\operatorname{Spec}}
\newcommand{\Spa}{\operatorname{Spa}}

\newcommand{\Sch}{\operatorname{\mathbf{Sch}}}
\newcommand{\Shv}{\operatorname{\mathbf{Shv}}}
\newcommand{\Ab}{\operatorname{\mathbf{Ab}}}

\newcommand{\pro}[1]{\text{\rm pro}_{#1}\text{\rm--}}

\newcommand{\dlog}{{\operatorname{dlog}}}

\newcommand{\eff}{{\operatorname{eff}}}

\newcommand{\red}{{\operatorname{red}}}

\newcommand{\Zar}{{\operatorname{Zar}}}
\newcommand{\Nis}{{\operatorname{Nis}}}
\newcommand{\et}{{\operatorname{\acute{e}t}}}

\newcommand{\inj}{\hookrightarrow}

\newcommand{\surj}{\rightarrow\!\!\!\!\!\rightarrow}

\newcommand{\Res}{\operatorname{Res}}

\newcommand{\id}{{\operatorname{Id}}}

\newcommand{\Sym}{{\operatorname{Sym}}}

\newcommand{\CH}{{\operatorname{CH}}}

\newcommand{\Frac}{{\operatorname{Frac}}}

\renewcommand{\lim}{\operatornamewithlimits{\varprojlim}}
\newcommand{\colim}{\operatornamewithlimits{\varinjlim}}

\newcommand{\ol}{\overline}

\renewcommand{\phi}{\varphi}
\renewcommand{\epsilon}{\varepsilon}

\newcommand{\CI}{\operatorname{\mathbf{CI}}}

\newcommand{\Bl}{{\mathbf{Bl}}}

\newcommand{\M}{\mathbf{M}}

\def\rmapo#1{\overset{#1}{\longrightarrow}}

\def\tV{\widetilde{V}}
\def\tW{\widetilde{W}}

\newcounter{spec}
{\end{list}}%

\newtheorem{lemma}{Lemma}[section]
\newtheorem{thm}[lemma]{Theorem}

\newtheorem{theorem}{Theorem}
\newtheorem{prop}[lemma]{Proposition}

\newtheorem{corollary}[lemma]{Corollary}
\newtheorem{cor-intro}{Corollary}

\newtheorem{claim}[lemma]{Claim}
\theoremstyle{definition}
\newtheorem{defn}[lemma]{Definition}

\newtheorem{definition}[lemma]{Definition}
\newtheorem{para}[lemma]{}
\theoremstyle{remark}

\newtheorem{rmk}[lemma]{Remark}
\newtheorem{remark}[lemma]{Remark}

\newtheorem{example}[lemma]{Example}

\numberwithin{equation}{section}

\numberwithin{equation}{lemma}

\colorlet{LightRubineRed}{RubineRed!70!}

\usepackage[color = blue!20, bordercolor = black, textsize = tiny]{todonotes}

\usepackage{relsize} 
\usepackage[bbgreekl]{mathbbol} 
\usepackage{amsfonts} 
\DeclareSymbolFontAlphabet{\mathbb}{AMSb} 
\DeclareSymbolFontAlphabet{\mathbbl}{bbold} 

\DeclareSymbolFontAlphabet{\mathbbl}{bbold}

\usepackage{scalerel}
\usepackage[inline]{enumitem}
\def\Ab{\mathbf{Ab}}
\def\lSm{\mathbf{lSm}}
\def\SmlSm{\mathbf{SmlSm}}
\def\Sm{\mathbf{Sm}}
\def\sM{\mathcal{M}}
\def\RSC{\mathbf{RSC}}

\def\qfor{ \text{ for } }

\def\tX{\tilde{X}}

\def\ty{\tilde{y}}
\def\tU{\tilde{U}}
\def\tV{\tilde{V}}
\def\tW{\tilde{W}}
\def\tT{\tilde{T}}
\def\tA{\tilde{A}}
\def\tB{\tilde{B}}
\def\tC{\tilde{C}}
\def\tD{\tilde{D}}

\def\tc{\tilde{c}}
\def\tf{\tilde{f}}

\def\tg{\tilde{g}}
\def\th{\tilde{h}}

\def\tY{\tilde{Y}}

\def\tu{\tilde{u}}
\def\tV{\tilde{V}}
\def\cU{\mathcal{U}}
\def\cW{\mathcal{W}}
\def\cV{\mathcal{V}}
\def\cT{\mathcal{T}}
\def\cY{\mathcal{Y}}
\def\cQ{\mathcal{Q}}
\def\Spa{\mathrm{Spa}}
\def\W{\mathbb{W}}
\def\tP{\tilde{P}}

\def\catprojlim#1{\underset{#1}{``\varprojlim"}}

\def\qaq{\;\text{ and }\;}
\def\rmapo#1{\overset{#1}{\longrightarrow}}

\def\SchSS{\Sch_{(S,\tS)}}
\def\SchXX{\Sch_{(X,\tX)}}

\newcounter{elno}

\newcommand\sbullet[1][.6]{\mathbin{\ThisStyle{\vcenter{\hbox{%
  \scalebox{#1}{$\SavedStyle\bullet$}}}}}%
}
\begin{document}

\def\THH{\operatorname{THH}}
\def\TC{\operatorname{TC}}
\def\TCmin{\operatorname{TC}^-}
\def\TP{\operatorname{TP}}
\def\HH{\operatorname{HH}}
\def\HC{\operatorname{HC}}
\def\HCmin{\operatorname{HC}^-}
\def\HP{\operatorname{HP}}

\def\Fil{\operatorname{Fil}}
\def\Gr{\operatorname{Gr}}
\def\gr{\operatorname{gr}}

\def\QSyn{\operatorname{QSyn}}
\def\QRSPerfd{\operatorname{QRSPerfd}}
\def\lQSyn{\operatorname{lQSyn}}
\def\lQRSPerfd{\operatorname{lQRSPerfd}}

\def\syn{\mathrm{syn}}
\def\Fsyn{\mathrm{Fsyn}}
\def\Fet{\mathrm{F\acute{e}t}}

\def\LogRec{\operatorname{\mathbf{LogRec}}}
\def\Ch{\operatorname{\mathrm{Ch}}}

\def\ltr{\mathrm{ltr}}

\def\kX{\mathfrak{X}}
\def\kY{\mathfrak{Y}}

\def\otCIsp{\otimes_{\CI}^{sp}}
\def\otCINissp{\otimes_{\CI}^{\Nis,sp}}

\def\tL{\tilde{L}}
\def\tX{\tilde{X}}
\def\tY{\tilde{Y}}
\def\tF{\widetilde{F}}
\def\tG{\widetilde{G}}
\def\tR{\widetilde{R}}
\def\tS{\widetilde{S}}

\def\Sh{\operatorname{\mathbf{Shv}}}
\def\PSh{\operatorname{\mathbf{PSh}}}
\def\Shltr{\operatorname{\mathbf{Shv}_{dNis}^{ltr}}}
\def\Shlog{\operatorname{\mathbf{Shv}_{dNis}^{log}}}
\def\Shvlog{\operatorname{\mathbf{Shv}^{log}}}
\def\Sm{\operatorname{\mathrm{Sm}}}
\def\SmlSm{\operatorname{\mathrm{SmlSm}}}
\def\lSm{\operatorname{\mathrm{lSm}}}
\def\FlSm{\operatorname{\mathrm{FlSm}}}
\def\FlQSm{\operatorname{\mathrm{FlQSm}}}
\def\FlQSyn{\operatorname{\mathrm{FlQSyn}}}
\def\lCor{\operatorname{\mathrm{lCor}}}
\def\SmlCor{\operatorname{\mathrm{SmlCor}}}
\def\PShltr{\operatorname{\mathbf{PSh}^{ltr}}}
\def\PShlog{\operatorname{\mathbf{PSh}^{log}}}
\def\logCI{\mathbf{logCI}} 

\def\Mod{\operatorname{Mod}}
\newcommand{\DM}[1][]{\operatorname{\mathcal{DM}_{#1}}}
\newcommand{\DMeff}[1][]{\operatorname{\mathcal{DM}^{\eff}_{#1}}}
\newcommand{\DA}[1][]{\operatorname{\mathcal{DA}_{#1}}}
\newcommand{\DAeff}[1][]{\operatorname{\mathcal{DA}^{\eff}_{#1}}}
\newcommand{\FDA}[1][]{\operatorname{\mathcal{DA}_{#1}}}
\newcommand{\FDAeff}[1][]{\operatorname{\mathcal{DA}^{\eff}_{#1}}}
\newcommand{\SH}[1][]{\operatorname{\mathcal{SH}_{#1}}}
\newcommand{\logSH}[1][]{\operatorname{\mathbf{log}\mathcal{SH}_{#1}}}
\newcommand{\logDA}[1][]{\operatorname{\mathbf{log}\mathcal{DA}_{#1}}}
\newcommand{\logDAeff}[1][]{\operatorname{\mathbf{log}\mathcal{DA}^{\eff}_{#1}}}
\newcommand{\logDM}[1][]{\operatorname{\mathbf{log}\mathcal{DM}_{#1}}}
\newcommand{\logDMeff}[1][]{\operatorname{\mathbf{log}\mathcal{DM}^{\eff}_{#1}}}
\newcommand{\logFDA}[1][]{\operatorname{\mathbf{log}\mathcal{FDA}_{#1}}}
\newcommand{\logFDAeff}[1][]{\operatorname{\mathbf{log}\mathcal{FDA}^{\eff}_{#1}}}
\newcommand{\WOmega}{\operatorname{\mathcal{W}\Omega}}
\def\Log{\operatorname{\mathcal{L}\textit{og}}}
\def\Rsc{\operatorname{\mathcal{R}\textit{sc}}}
\def\Pro{\mathrm{Pro}\textrm{-}}
\def\pro{\mathrm{pro}\textrm{-}}
\def\dg{\mathrm{dg}}
\def\plim{\mathrm{``lim"}}
\def\ker{\mathrm{ker}}
\def\coker{\mathrm{coker}}
\def\PrL{\mathcal{P}\mathrm{r^L}}
\def\PrLo{\mathcal{P}\mathrm{r^{L,\otimes}}}
\def\Spt{\mathcal{S}\mathrm{pt}}
\def\PSpt{\mathrm{Pre}\mathcal{S}\mathrm{pt}}
\def\Fun{\mathrm{Fun}}
\def\Sym{\mathrm{Sym}}
\def\CAlg{\mathrm{CAlg}}
\def\Poly{\mathrm{Poly}}
\def\Cat{\mathrm{Cat}}

\def\Alb{\operatorname{Alb}}
\def\bAlb{\mathbf{Alb}}
\def\Gal{\operatorname{Gal}}

\def\hofib{\mathrm{hofib}}
\def\fib{\mathrm{fib}}
\def\triv{\mathrm{triv}}
\def\ABl{\mathcal{A}\textit{Bl}}
\def\divsm#1{{#1_\mathrm{div}^{\mathrm{Sm}}}}

\def\cA{\mathcal{A}}
\def\cB{\mathcal{B}}
\def\cC{\mathcal{C}}
\def\cD{\mathcal{D}}
\def\cE{\mathcal{E}}
\def\cF{\mathcal{F}}
\def\cG{\mathcal{G}}
\def\cH{\mathcal{H}}
\def\cI{\mathcal{I}}
\def\cJ{\mathcal{J}}
\def\cK{\mathcal{K}}
\def\cL{\mathcal{L}}
\def\cM{\mathcal{M}}
\def\cN{\mathcal{N}}
\def\cO{\mathcal{O}}
\def\cP{\mathcal{P}}
\def\cQ{\mathcal{Q}}
\def\cR{\mathcal{R}}
\def\cS{\mathcal{S}}
\def\cT{\mathcal{T}}
\def\cU{\mathcal{U}}
\def\cV{\mathcal{V}}
\def\cW{\mathcal{W}}
\def\cX{\mathcal{X}}
\def\cY{\mathcal{Y}}
\def\cZ{\mathcal{Z}}

\def\tcA{\widetilde{\mathcal{A}}}
\def\tcB{\widetilde{\mathcal{B}}}
\def\tcC{\widetilde{\mathcal{C}}}
\def\tcD{\widetilde{\mathcal{D}}}
\def\tcE{\widetilde{\mathcal{E}}}
\def\tcF{\widetilde{\mathcal{F}}}

\def\one{\mathbbm{1}}

\def\XP{X \backslash \sP}
\def\M0a{{}^t\cM_0^a}
\newcommand{\Ind}{{\operatorname{Ind}}}

\def\Xkbar{\overline{X}_{\overline{k}}}
\def\dx{{\rm d}x}

\newcommand{\dNis}{{\operatorname{dNis}}}
\newcommand{\loget}{{\operatorname{l\acute{e}t}}}
\newcommand{\ket}{{\operatorname{k\acute{e}t}}}
\newcommand{\vet}{{\operatorname{v\acute{e}t}}}
\newcommand{\ABNis}{{\operatorname{AB-Nis}}}
\newcommand{\sNis}{{\operatorname{sNis}}}
\newcommand{\sZar}{{\operatorname{sZar}}}
\newcommand{\set}{{\operatorname{s\acute{e}t}}}
\newcommand{\cofib}{\mathrm{Cofib}}

\newcommand{\Gmlog}{\G_m^{\log}}
\newcommand{\Gmlogred}{\overline{\G_m^{\log}}}

\newcommand{\varcolim}{\mathop{\mathrm{colim}}}
\newcommand{\varlim}{\mathop{\mathrm{lim}}}
\newcommand{\tensor}{\otimes}

\newcommand{\eq}[2]{\begin{equation}\label{#1}#2 \end{equation}}
\newcommand{\eqalign}[2]{\begin{equation}\label{#1}\begin{aligned}#2 \end{aligned}\end{equation}}

\def\varplim#1{\text{``}\varlim_{#1}\text{''}}
\def\det{\mathrm{d\acute{e}t}}

\renewcommand{\int}{\mathrm{int}}
\def\federem#1{\begin{color}{teal}{#1}\end{color}}
\def\tomrem#1{\begin{color}{purple}{#1}\end{color}}

\author{Alberto Merici, Kay R\"ulling \and Shuji Saito}
\address{Institut f\"ur Mathematik,  Universit\"at Heidelberg. INF 205, 69115 Heidelberg, Germany}
\email[A. Merici]{merici@mathi.uni-heidelberg.de}

\address{School of Mathematics and Natural Sciences, University of Wuppertal, Germany}
\email[K. R\"ulling]{ruelling@uni-wuppertal.de}

\address{Graduate School of Mathematical Sciences, the University of Tokyo, 3-8-1 Komaba Meguro-ku
Tokyo 153-8914, Japan}
\email[S. Saito]{sshuji.goo@gmail.com}

\thanks{Funded by the Deutsche Forschungsgemeinschaft (DFG, German Research 
Foundation)
TRR 326 \textit{Geometry and Arithmetic of Uniformized Structures}, 
project number 444845124./}
\title{A construction of tame sheaves and tame de Rham--Witt cohomology}

\begin{abstract} 
In this article, we consider an algebraic version of the tame site of a pair $(X,\tX)$. With this definition, we provide a general machinery to construct a tame sheaf from the data of an \'etale sheaf on $X$ and a family of local tame sections. We apply this construction to the big de Rham--Witt sheaves with tame sections defined by log poles and, over a field, to reciprocity sheaves, and deduce some consequences. As an application, we compare tame syntomic cohomology with the Nygaard filtration on the tame de Rham--Witt complex.
\end{abstract}
\maketitle
\tableofcontents

\section*{Introduction}

Let $k$ be a field and $\Sm_k$ be the category of smooth schemes separated and of finite type over $k$. 
If {$k$ has characteristic zero}, a well behaved cohomology theory on $\Sm_k$ is given by
\[
X\mapsto R\Gamma(\ol{X},\Omega^\bullet_{\ol{X}}(\log D)), \tag{1}\label{eq;Omegalog}\]
where $\ol{X}$ is a smooth compactification of $X$ with $D=\ol{X}-X$ a divisor with simple normal {crossings} and 
$\Omega^\bullet_{\ol{X}}(\log D)$ is the associated {de Rham complex with logarithmic poles along $D$.} 
{It is independent of the choice of $\ol{X}$, by  \cite{DeligneHodgeII} }. 
An important feature of \eqref{eq;Omegalog} is that it satisfies \'etale descent: For example, this follows from the comparison with singular cohomology assuming $k=\C$ by the Lefschetz principle. 

{If the characteristic of $k$ is $p>0$, then one can} construct an integral $p$-adic cohomology theory on $\Sm_k$  following the above construction using the de Rham--Witt complex {with log poles instead}, see e.g. \cite{ertl2021integral}. 
{But it turns out that it does not satisfies  \'etale descent.}
{Indeed,} assuming the existence of good compactifications, the cohomology theory
\[
X\mapsto R\lim_m R\Gamma(\ol{X},W_m\Omega^\bullet_{\ol{X}}(\log D)) \tag{2}\label{eq;WOmegalog}\]
compares with rigid cohomology while the cohomology groups are  finitely generated $W(k)$-modules. Therefore,  
it cannot have \'etale descent, by \cite{AbeCrew}. 
A counterexample is given by an  Artin--Schreier covering so this discrepancy is due to  wild ramification.
On the other hand, in \cite{mericicrys}, it has been shown that (under the assumption of resolutions of singularities) the cohomology theory \eqref{eq;WOmegalog} satisfies the tame descent in the sense of \cite{HS2020}, using tools from motivic homotopy theory of logarithmic schemes. This suggests that in order to construct integral models of $p$-adic {cohomology theories} 
of (not necessarily proper) schemes, one needs to take into account this weaker descent instead of \'etale descent.
\medbreak

The main goal of this paper is to 
introduce a variant of the tame topology of \cite{HS2020} and establish its foundational results, which enable us to construct various cohomology theories which satisfy  tame descent.
For a quasi-compact open immersion $S\hookrightarrow \tS$\footnote{For example, $S=\tS=\Spec(k)$ or $(S,\tS)=(\Spec(\Frac(R)),\Spec(R))$ for a discrete valuation ring $R$.}, we consider a tame site $\Sch_{(S,\tS),t}$ whose underlying category is a category of quasi-compact open immersions $X\hookrightarrow \tX$ over $S\hookrightarrow \tS$ with $X$ of finite type over $S$ and $\tX\to \tS$ satisfying a certain finiteness condition. For $(X,\tX)\in \Sch_{(S,\tS),t}$, we also consider the tame site $(X,\tX)_t$ {on} the category of the objects $(U,\tU)$ over $(X,\tX)$ with $U\to X$ \'etale endowed with the induced topology of $\Sch_{(S,\tS),t}$.
A useful property of tame cohomology is the independence of compactifications:
For two objects $(X,\tX),\; (X,\tX')$ of $\Sch_{(S,\tS),t}$ such that $\tX,\tX'$ are both  Nagata compactifications of $X$ over $\tS$, there is a natural equivalence 
\[ R\Gamma((X,\tX)_t,F) \simeq R\Gamma((X,\tX')_t,F),\]
for any sheaf $F$ on $\Sch_{(S,\tS),t}$, {see Lemma \ref{lem:blow-up-inv}. Thus}  the formation
$X\to R\Gamma((X,\tX)_t,F)$ with a choice of a compactification $\tX$ gives a well-defined cohomology theory on the category of schemes separated and of finite type over $S$ relative to the embedding $S\hookrightarrow \tS$, which is functorial for morphisms of $S$-schemes.
Another useful fact is that we can provide a general machinery to construct sheaves on $(X,\tX)_t$ providing several examples of tame sheaves and cohomology theories satisfying the tame descent. 
The construction which we call the $\beta$-construction, takes as an input an \'etale sheaf $F$ on $X_{\et}$ equipped with a local datum $\beta$ of ``tame sections'' on every valuation ring of a residue field of $X$ centered on $\tX$, and then gives as an output a tame sheaf $F_{\beta}$ on $(X,\tX)_t$ (see Proposition \ref{prop:constr-tame-sheaves}).
In the circumstances where resolution of singularities holds, it turns out that the cohomology theories \eqref{eq;Omegalog} and \eqref{eq;WOmegalog} are obtained as special instances of the $\beta$-construction, giving another proof that these cohomology theories satisfy tame descent (see Theorem \ref{thm:tameDRWcoh} and \ref{thm:tamecohRSC}).

The construction of the tame site $(X,\tX)_t$ is similar in spirit to that of the site of \cite{Huebner2020} constructed on the adic space $\Spa(X,\tX)$. A difference here is that since the objects of our site are schemes (rather than adic spaces), it is much easier to construct sheaves on $(X,\tX)_t$. We were informed by D'Addezio that he had defined related sites in the setup of marked schemes in \cite{DAddezio}.
\medbreak

Now we state the main result on a computation of tame cohomology on $(X,\tX)_t$.

\begin{theorem}\label{thm:tame-vs-etale-intro}
    (see Theorem \ref{thm:tame-vs-etale}):
	Let $F$ be a tame sheaf of abelian groups on $(X,\tX)_t$ such that the following condition is satisfied:
	\begin{enumerate}[label=(p)]
		\item
		for every $(U,\tU)\to (X,\tX)$ in $(X,\tX)_t$, the stalk of the Zariski sheaf on $\tU$ given by $\tW\subset \tU \mapsto F(\tW\cap U, \tW)$
		at any point $x\in \tU$ is a $\Z_{(p_x)}$-module, where 
		{$p_x$ is the exponential characteristic of $\kappa(x)$.}
	\end{enumerate}
	Then, we have a natural equivalence
\[
R\Gamma((X,\tX)_t, F)\simeq \colim_{\cW\to (X,\tX)}R\Gamma(\tW_{\et}, j^{\cW}_*F), 
\]
where the colimit runs along all admissible blow-ups $\cW=(W,\tW)\to (X,\tX)$ and $j^{\cW}_*F$ is the \'etale sheaf on $\tW_{\et}$ given by $\tV/\tW \mapsto F(\tV\times_{\tW} W,\tV)$.
\end{theorem}

An analogous result for torsion sheaves on adic spaces is obtained by combining \cite[Proposition 8.5]{Huebner2020} and \cite[Corollary 5.7]{Huebner2025} by a different method.
The above theorem plays a fundamental role in our subsequent article, where we construct a functorial  integral structure on cohomology of the structure sheaf over a non archimedian field.
\medbreak

As an application, over a perfect field $k$ of positive characteristics $p$, we define the \emph{tame de Rham--Witt complex} $W_{\sbullet}\Omega^{\bullet, t}$, a complex of tame sheaves obtained from the de Rham--Witt complex using the $\beta$-construction mentioned above, and get the complexes
\[\Z/p^{\sbullet}(r)^t= \Cone(F_r-\iota_r: 
\sN^r W_{\sbullet}\Omega^{\bullet, t}\to W_{\sbullet}\Omega^{\bullet,t})[-1],\]
where $\sN^r W_{\sbullet}\Omega^{\bullet, t}$ is a Nygaard filtration on $W_{\sbullet}\Omega^{\bullet, t}$ defined following \cite{LangerZink} (see Definition \ref{defn:tame-syntomic}).
Then, the $r$th \emph{tame syntomic complex} is defined as  
\[\Z_p(r)^t=R\lim \Z/p^{\sbullet}(r)^t.\]
We then show that the tame syntomic complex recovers the tame motivic cohomology considered in \cite{lueders-tame-twists} and \cite{mericicrys}:
\begin{theorem}[Theorem \ref{thm:syn-vs-log}]
Assume $\tX$ is separated and of finite type over $k$ and $X\subset \tX$ is a smooth open subscheme.
We have  an isomorphism of pro-complexes on $(X,\tX)_t$ 
\[\Z/p^{\sbullet}(r)^t= W_{\sbullet}\Omega^{r,t}_{\log}[-r],\]
where  $W_n\Omega^{r,t}_{\log}$ is the sheaf on 
$(X,\tX)_t$ whose section over an object $(V,\tV)$ of $(X,\tX)_t$ is given by  
\[W_n\Omega^{r,t}_{\log}(V,\tV):= \Gamma(V,W_n\Omega^r_{X,\log}).\]
\end{theorem}

Using the above result we obtain a cycle map to tame de Rham--Witt cohomology.

\begin{theorem}[Theorem \ref{thm:tameDRWcoh} and Corollary \ref{cor:tame-DRW-cycle-map}]\label{thm:intro3}
Assume $\tX$ is  smooth and  proper over  $k$ and $X\subset \tX$  is a dense open subscheme. 
There is a functorial  cycle map ${\rm cyc}^t$ to tame de Rham--Witt cohomology which factors via tame syntomic cohomology and 
is compatible with the usual crystalline cycle map ${\rm cyc}^{\rm crys}$ in the sense that the following diagram commutes
\[\xymatrix{
\CH^r(\tX)\ar[r]^-{{\rm cyc}^{\rm crys}}\ar[d] & H^r_{\rm crys}(\tX/W(k))\ar[d]\\
\CH^r(X)\ar[r]^-{{\rm cyc}^t} & H^r((X,\tX)_t, W\Omega^{\bullet,t}).
}\]

Moreover, if resolution of singularities in a strong sense hold in dimension $d=\dim \tX$ (this is the case for $d\le 2$),
and if $D=\tX\setminus X$ is a divisor with simple normal crossings, then there is a canonical equivalence
\[R\Gamma((X,\tX)_t, W\Omega^{\bullet,t})= R\Gamma_{\rm log-crys}((\tX,D)/W(k)),\]
where the right hand side denotes log-crystalline cohomology.
\end{theorem}

In fact using the $\beta$-construction and \cite{Hesselholt-dRW}
we define the tame big de Rham--Witt complex $\W_T\Omega^{\bullet,t}$ on $(X,\tX)_t$,
where $X\inj \tX$ is any quasi-compact open immersion between qcqs schemes and $T$ is any truncation set. 
One of the main structure results is Proposition \ref{prop:WOmegat-via-val} which together with 
Theorem \ref{thm:tame-vs-etale-intro} yields the proof of the second part of Theorem \ref{thm:intro3}. 
We remark that with the same arguments and \cite{DeligneHodgeII} one deduces in case $X$ is smooth and separated 
over $\C$ with some compactification $\tX$ the equivalence
\[R\Gamma((X,\tX), \Omega^{\bullet, t})\simeq R\Gamma_{\text{Betti}}(X(\C), \C).\]

The following is an algebraic version of the comparison of the tame cohomology with tame adic cohomology,
in \cite[14.8]{HS2020} (there for $\F_p$-schemes). 
\begin{theorem}[Proposition \ref{prop-def;XStHS} and Theorem \ref{thm:HS-comparison}]
    Let $S$ be an affine scheme and $X$ an $S$-scheme with a quasi-compact open 
    immersion $X\hookrightarrow \tX$ into  a  proper $S$-scheme $\tX$. Then there are adjoint functors\[
\begin{tikzcd}
\Sh((X,\tX)_t)\ar[r,"u_*"'] &\Sh((X/S)_t),\ar[l, shift right = 1, "u^*"']
\end{tikzcd}
\]
with $u^*$ exact, and for all $G\in \Sh((X/S)_t)$ we have an equivalence
\[
    R\Gamma((X/S)_t,G) \simeq R\Gamma((X,\tX)_t,u^*G).
    \]
\end{theorem}

As we can easily construct sheaves on $(X,\tX)_t$ one might wonder how the cohomology of
a sheaf $F$ on $(X,\tX)_t$ compares to the cohomology of $u_*F$ on $(X/S)_t$. 
These do not coincide in general and  in fact can be rather different.
For example, if $S=k$ is an algebraically closed field of positive characteristic $p$, $X$ is a smooth separated $k$-scheme of 
dimension $d\le 3$, $X\inj \tX$ is a smooth compactification, and $\sO^t$ is the tame structure sheaf on $(X,\tX)_t$, then 
\[H^i((X,\tX)_t, \sO^t)=H^i(\tX, \sO_{\tX})\quad \text{and}\quad H^i((X/k)_t, u_*\sO^t)= H^i(\tX_{\et}, \Z/p)\otimes_{\F_p} k.\]
These two cohomology groups only coincide if the Frobenius acts bijectively on the  cohomology of the structure sheaf.
If $\tX$ is a supersingular elliptic curve and $X$ is any non-empty open, then the left hand side 
is a one dimensional $k$-vector space whereas the right hand side vanishes,
see Remark \ref{rmk:tame-vs-tame}. 

In fact a similar phenomena also appears in characteristic zero.
Assume ${\rm ch}(k)=0$, $\tX$ is smooth and proper over $k$, and {$D=\tX\setminus X$} has simple normal crossings. 
Then we have   by Theorem \ref{thm:intro3} and Lemma \ref{lem:tameRSCA1}
\[
H^j((X,\tX)_t, \Omega^{q,t})= H^j(\tX, \Omega^q_{\tX}(\log {D})),\]
whereas
\[ H^j((X/k)_t,  u_*\Omega^{q,t})= H^j(X, h^{0}_{\A^1}(\Omega^q)),
 \]
where $h^{0}_{\A^1}(\Omega^q)$ is the maximal $\A^1$-invariant subsheaf of $\Omega^q$.
Here we use that $h^{0}_{\A^1}(\Omega^q)$ is in fact an \'etale sheaf on $X$, by \cite[Proposition 5.24]{VoevPST}.
Note however that in positive characteristic it is not an \'etale sheaf, see Remark \ref{rmk:h0-tame-vs-et} for details.

\medbreak
Here is an outline of the paper:
\begin{enumerate}
	\item In Section \ref{sec:1}, we give some preliminary results on not necessarily Noetherian compactifications that will be used later. This uses the notion of \emph{ift} morphisms due to Temkin (see \cite{Temkin-insep-loc-uni}).
	\item Sections \ref{sec:2}--\ref{sec:7} are devoted to the definition and properties of the tame and $\vet$ site. These definitions and computations are very much inspired by the definition of the \'etale and tame site of a Huber pair (see \cite{Huebner2020} and \cite{Huebner2021}), but with some technical difficulties in the proofs to be addressed, therefore everything needed to be rewritten:
	\begin{enumerate}
	\item In Section \ref{sec:2} and \ref{sec:3}, we define the tame and $\vet$ topos and prove some immediate results. 
	\item In Section \ref{sec:4} and \ref{sec:5} we show the existence of acyclic schemes and compute the local schemes for the tame and $\vet$ topology. The proof of Proposition \ref{prop:tamely-acyclic} is  
    {rather long and follows the proof of a  similar computation in \cite{KelMiya}}.
	\item In Section \ref{sec:6}, we use the existence of such acyclic objects to show that \v Cech cohomology agrees with sheaf cohomology. {This relies on the classical strategy by Artin \cite{artin} and } is very close to \cite[Theorem 11.1 and Corollary 11.7]{HS2020}.
	\item {In Section \ref{sec:7} we prove Theorem \ref{thm:tame-vs-etale-intro}.} 
    \end{enumerate}
	\item In Sections \ref{sec:8}--\ref{sec:11}, we give the construction of tame sheaves and provide examples:
	\begin{enumerate}
		\item In Section \ref{sec:8} we construct the general machinery.
		\item In Section \ref{sec:9} we prove {basic structural results} on the {tame} big de Rham--Witt complex.
		\item In Section \ref{sec:10}, we restrict to the case of positive characteristic and compute the cohomology of the tame $p$-typical de Rham--Witt sheaves, and we deduce some results on tame syntomic cohomology.
		\item In Section \ref{sec:11}, we consider reciprocity sheaves over a perfect field $k$.
	\end{enumerate} 
	\item Finally, in Section \ref{sec:12} we {prove} a comparison with the tame site of \cite{HS2020}, and we remark that the cohomology theories constructed in the previous sections have tame descent even though they do not come from 
    {cohomology theories} of sheaves over the site $(X/S)_t$. This implies that passing from the site $(X,\tX)_t$ to the site $(X/S)_t$, a lot of information is lost.
    \end{enumerate}

\subsection*{Future work} In subsequent papers, we hope to use the construction and properties of tame sheaves developed here to construct integral models of cohomology theories in positive and mixed characteristics. In particular, for $X$ smooth over a non-archimedean field $K$, the constructions above give $\cO_K$-models of coherent and de Rham cohomology. According to the various result on failure of existence of integral structures on crystalline cohomology (see e.g. \cite{AbeCrew}), the tame descent is remarkable, as \'etale descent indeed cannot be fulfilled.

\subsection*{Acknowledgments}
We thank Katharina Hübner and Amine Koubaa for useful discussions and Hugo Zock for pointing out a mistake in an earlier version.
We also thank Vincent Cossart for answering a question on resolution of singularities in positive characteristic. We thank Marco D'Addezio for pointing out the reference \cite{DAddezio}.

\section{Ift morphisms}\label{sec:1}

We recall a notion inspired by \cite{Huber}, and hinted in \cite[Remark 2.2.4 (iii)]{Temkin-insep-loc-uni}. 

\begin{defn}
\begin{enumerate}
    \item Let $R$ be a ring. A map $R\to A$ is \emph{ift}  (integral over finite type) if there exists a factorization $R\to A'\to A$ with $R\to A'$ of finite type and $A'\to A$ integral. 
    \item A morphism $f\colon X\to S$ is \emph{ift at $x\in X$} if there exists an affine open neighborhood $\Spec(A)=U\subseteq X$ of $x$ and an affine open $\Spec(R)=V\subseteq S$ with $V\subseteq f(U)$ such that the induced map $R\to A$ is ift.
    \item A morphism $f\colon X\to S$ is \emph{locally ift} if it is ift at all $x\in X$ and it is \emph{ift} if it is locally ift and quasi-compact.
\end{enumerate}
\end{defn}

\begin{rmk}
As remarked in the proof of \cite[Proposition 2.2.5]{Temkin-insep-loc-uni}, for $R\to R'$ integral and $R'\to A$ of finite type, the composition $R\to A$ is ift. Hence, being (locally) ift is stable under compositions and if $X\to Y$ is a morphism of 
quasi-compact quasi-separated schemes over a base $S$ with $X$ locally ift over $S$, then $X\to Y$ is locally ift. Moreover, it is clear that the property is stable under base change and, as observed in \cite[Remark 2.2.4]{Temkin-insep-loc-uni}, it is a local property on the source, 
\end{rmk}
In our context, we will consider ift morphisms that are of finite presentation on a fixed open. These morphisms will then be \emph{nft} in the sense of \cite[Definition 2.2.2 and Lemma 2.2.7]{Temkin-insep-loc-uni}, but as we will need to keep track of the fixed open, we will not use this notion. We prove the following lemma, which generalizes \cite[Proposition 2.3.8]{Temkin-insep-loc-uni}:
\begin{lemma}\label{lem:ift-fp-limits} 
Let $\phi:\tR\to R$ and $\psi: \tA\to A$ be maps of rings and 
$(f,\tf):(R,\tR)\to (A,\tA)$ be a pair of maps of rings compatible with $\phi$ and $\psi$ such that 
$f$ is of finite presentation and $\tf$ is ift.
Let $\{(B_i,\tB_i)\}_{i\in I}$ be a filtered system of pairs of rings such that $\tB_i\subseteq B_i$ are integrally closed and $(g_i,\tg_i)_{i\in I} \colon (R,\tR)\to \{(B_i,\tB_i)\}_{i\in I}$ be a system of pairs of maps of rings.
Let $B=\colim B_i$ and $\tB = \colim \tB_i$ and $(g,\tg)=\colim(g_i,\tg_i):(R,\tR) \to (B,\tB)$. Then, for all $(h,\th)\colon (A,\tA)\to (B,\tB)$ compatible with $(f,\tf)$ and $(g,\tg)$, there is $i\in I$ and $(h_i,\th_i)$ fitting into the following commutative squares of pairs of rings:\[
    \begin{tikzcd}
        (R,\tR)\ar[r,"{(f,\tf)}"]\ar[d,"{(g_i,\tg_i)}"']&(A,\tA)\ar[d,"{(h,\th)}"]\arrow{dl}[rotate=30,swap,xshift=3ex]{(h_i,\th_i)}\\
        (B_i,\tB_i)\ar[r]&(B,\tB).
    \end{tikzcd}
    \]
        If $\tg_i$ are ift, then so is $\th_i$.
\end{lemma}
\begin{proof}
    As $f\colon R\to A$ is of finite presentation, $h\colon A\to B$ factors through $h_i\colon A\to B_i$ for some $i$ by \cite[\href{https://stacks.math.columbia.edu/tag/00QO}{Tag 00QO}]{stacks-project}. We need to show that the composition $\tA\to A\xrightarrow{h_i} B_i$ factors via 
    $\tB_i\inj B_i$ for $i$ large enough.
    By definition, $\tf\colon \tR\to \tA$ factors as a composition $\tR\to \tR[t_1\ldots t_n]\twoheadrightarrow E\xrightarrow{\phi} \tA$ with $\phi$ integral. As $\tR\to \tR[t_1,\ldots,t_n]$ is of finite presentation, for $i$ large enough, there is a map $\nu\colon \tR[t_1,\ldots t_n]\to \tB_i$ 
   fitting in the following commutative diagram
\[\xymatrix{
\tilde{R}[t_1,\ldots, t_n]\ar[r]\ar[rd]_-\nu &\tilde{A} \ar[r]& A \ar[d]^-{h_i} \\
 &\tilde{B}_{i} \ar[r]& B_i}\]
 Since the composite $\tilde{R}[t_1,\ldots, t_n]\to A$ factors through $E$, the kernel of $\tR[x_1,\ldots x_n]\twoheadrightarrow E$ maps to zero in $B_i$ via this composition. 
 Since $\tilde{B}_{i} \to B_{i}$ is injective, 
 $\nu$ induces a map $E\to \tB_i$. As $\tA$ is integral over $E$, the image of $\tA$ in $B_{i}$ is integral over $\tB_i$. 
       Since $\tB_i$ is integrally closed in $B_i$, we conclude that the map $(h_i)_{|\tA}\colon \tA\to B_i$ factors via  $\th_i\colon \tA\to \tilde{B}_{i}$. 
       If $\tg_i$ is ift, so is $\th_i$ by \cite[Proposition 2.2.5]{Temkin-insep-loc-uni}
\end{proof}
\begin{remark}\label{rmk:ift-limit-fp-etale}
    In the situation of Lemma \ref{lem:ift-fp-limits}, if $g_i$ is of finite presentation, so is $h_i$ by \cite[\href{https://stacks.math.columbia.edu/tag/00F4}{Tag 00F4}]{stacks-project}.
    If both $f$ and $g_i$ are \'etale, then $h_i$ is \'etale by 
    \cite[\href{https://stacks.math.columbia.edu/tag/02GW}{Tag 02GW}]{stacks-project}.
\end{remark}
\begin{lemma}\label{lem;A1}
    Let $f\colon X\to S$ be a morphism of schemes. Assume that
    \begin{enumerate}
        \item f is ift
        \item $X$ is qcqs
        \item $S$ is quasi-separated
    \end{enumerate}
    Then there exists $f'\colon X'\to S$ of finite presentation and $X\to X'$ integral that factor $f$. If $X\to S$ is separated, then $X'\to S$ is separated.
    \begin{proof}
        This is a version of \cite[\href{https://stacks.math.columbia.edu/tag/01ZG}{Tag 01ZG}]{stacks-project} and we follow its proof.
        By \cite[\href{https://stacks.math.columbia.edu/tag/01ZA}{Tag 01ZA}]{stacks-project} we can write 
        $X=\lim_{i\in I} X_i$ with $I$ a directed set, $X_i$ of finite presentation over $\Spec(\Z)$, and with affine transition maps $X_{i'}\to X_i$. Then, take an affine open covering $X=V_1\cup\ldots \cup V_n$ such that each $V_j$ maps to an affine open $U_j\subseteq S$ and the corresponding map of rings is ift, \emph{i.e.} it factors as $\cO(U_j)\to R\to \cO(V_j)$ with $\cO(U_j)\to R$ of finite type and $R\to \cO(V_j)$ integral. By \cite[\href{https://stacks.math.columbia.edu/tag/01Z4}{Tag 01Z4}]{stacks-project},
        and \cite[\href{https://stacks.math.columbia.edu/tag/01Z6}{Tag 01Z6}]{stacks-project},  we can write $V_j = \lim_{i\in I} V_{ij}$ for $1\leq j\leq n$, such that $V_{ij}$ are of finite presentation over $\Spec(\Z)$, we have
        $V_j=V_{ij}\times_{X_i} X$, and $X_i = V_{i1}\cup\cdots\cup V_{in}$ is an affine open covering for each $i\in I$. For each $j$, choose finitely many $h_{ja}\in R$ 
        that generate $R$ as a $\cO(U_j)$-algebra and consider their image $k_{ja}$ in  $\cO(V_j)$. 
        Choose  $i_0$ such that the $k_{ja}$ come from $\cO(V_{i_0j})$, for all $j$ and $a$. 
 In particular, this implies that the map $V_j\to V_{i_0j}\times_{\Spec(\Z)}U_j=:W_{j}$ is integral. Moreover,  $X':=\cup_{j=1}^n W_{j}$ 
 is an open subscheme of $X_{i_0,S}=X_{i_0}\times_{\Spec \Z} S$, which contains the image of $X\to X_{i_0, S}$.
 Thus $f$ factors as $X\to X'\to S$ with  $X\to X'$  integral. 
 As $X_{i_0}$  and $S$ are quasi-separated \footnote{Note that morphisms 
         of finite presentation are quasi-separated by definition.} so is $X_{i_0, S}$.
         Recall that if $T$ is quasi-separated and $T'\subseteq T$ is an affine open, then the open immersion $T'\to T$ is quasicompact by \cite[\href{https://stacks.math.columbia.edu/tag/01K4}{Tag 01K4} and \href{https://stacks.math.columbia.edu/tag/01KO}{Tag 01KO}]{stacks-project}. 
         Thus  $W_j\to X_{i_0,S}$ is a quasi-compact open immersion and hence so is $X'\to X_{i_0, S}$. It follows that  $X'\to X_{i_0,S}$ is of finite presentation, and therefore also $X'\to S$.

    \end{proof}
\end{lemma}

\begin{lemma}\label{lem:pro-zmt}
    Let $f\colon X\to Y$ be an integral morphism and $j\colon Y\to \ol{Y}$ a quasi-compact open immersion. 
    Then there exists a cartesian diagram with $g$ an integral morphism
    \[\xymatrix{ 
     X\ar[r]\ar[d]_f & \ol{X}\ar[d]^g\\
     Y\ar[r]^j   & \ol{Y}.    }    
    \]
\end{lemma}

\begin{proof}
We may take $\ol{X}$ to be the integral closure of $\ol{Y}$ in $X$. More precisely,
by assumption the composition $j\circ f$ is quasi-compact and quasi separated and hence $(j\circ f)_*\cO_X$ is a quasi-coherent
$\cO_{\ol{Y}}$-module. Let $\cA$ be the subsheaf which over an open $V\subset \ol{Y}$ is given by 
\[\cA(V)=\{s\in \cO_X((j\circ f)^{-1}V)\mid s \text{ is integral over } \cO_{\ol{Y}}(V)\}.\]
Then $\cA$ is a quasi-coherent $\cO_{\ol{Y}}$-algebra and 
$j\circ f$ factors as $X\xrightarrow{u}\ol{X}:=\Spec \cA\xrightarrow{g} \ol{Y}$ with $g$  integral. 
As $f$ is integral we have $g^{-1}(Y)=X$ by construction.
\end{proof}

The following result will be used in Section \ref{Compvet}. 
An analogous result for proper morphisms instead of ift morphisms is \cite[Lemma 4.14.(2)]{Kel-S23}.

\begin{thm}\label{thm:ift-sep-uc-limit-modifications}
Let $f:X\to Y$ be a morphism between quasi-compact quasi-separated schemes, which is ift, separated, and universally closed,
and let $U\subset Y$ be a quasi-compact open such that the base change 
$X_U= X\times_Y U\to U$ is an isomorphism. Then $X=\lim_i X_i$ is a directed limit with 
finite $Y$-morphisms $\tau: X_j\to X_i$ ($j\ge i$) as transition maps  inducing injections $\tau^*:\cO_{X_i}\inj \tau_*\cO_{X_j}$,
such that each  $X_i\to Y$ is proper and an isomorphism over $U$.
\end{thm}
\begin{proof}
By Lemma \ref{lem;A1} we can factor $f$ as $X\xrightarrow{h} X'\xrightarrow{g} Y$ with $h$ integral 
and $g$ separated and of finite type. Let $Z\subset X'$ denote the scheme-theoretic image of $h$.
As $f$ is universally closed \cite[\href{https://stacks.math.columbia.edu/tag/0AH6}{Tag 0AH6}]{stacks-project}
yields that $f$ factors as $X\xrightarrow{h_1} Z\xrightarrow{g_1} Y$ with $g_1$ proper and $h_1$ integral and surjective.
Now consider $\cA=h_{1*}\cO_X$. It is an integral quasi-coherent $\cO_Z$-algebra and can hence be written
as a directed colimit $\cA= \colim \cA_i$ of its finite quasi-coherent $\cO_Z$-subalgebras. 
Thus $X= \lim X_i$ is a directed limit with $X_i=\Spec \cA_i$ finite over $Z$ and hence also finite transition maps
$\tau: X_j\to X_i$ which by construction are induced by the inclusions  $\cO_{X_i}\inj \tau_*\cO_{X_j}$ inside $\cA$. 
Clearly, the composition $X_i\to Z\to Y$ is proper. It remains to show that the projection 
$(X_{i})_U\to U$ is an isomorphism, for all $i$.

The question is local on $U$ and we can therefore assume $U=\Spec R$ is affine. 
Hence  $X_U$ is affine as well and the projection induces an isomorphism 
\[R\xrightarrow{\simeq} A:=\cO_{X}(X_U).\] 
By \cite[\href{https://stacks.math.columbia.edu/tag/01Z6}{Tag 01Z6}]{stacks-project}, the scheme $(X_i)_U$ is affine for $i$ big enough, so by \cite[Corollary A.2]{ConradNagata}, $(X_i)_U$ is affine for all $i$, therefore the map $X_U\to (X_i)_U$ is induced by an integral ring extension $A_i \inj A$.
Moreover, the section $\sigma_i: U \xrightarrow{\simeq} X_U\to (X_{i})_U$ of the projection corresponds to a surjection 
of rings $A_i\surj R$ which factors via the isomorphism $A\xrightarrow{\simeq} R$ and is therefore injective as well. 
 This shows that $\sigma_i$ is an isomorphism inverse to the projection.
\end{proof}
\section{Tame topos of pairs}\label{sec:2}

We fix some notations and recall definitions from, e.g., \cite[6.2]{Gabber-Ramero}, \cite[2]{HS2020}.

\medskip

Let $(K,v)$ be a valuation field with valuation ring $\cO_v$ and maximal ideal $\fm_v$.
Fix an embedding into $(\bar{K}, \bar{v})$, where $\bar{K}$ is a separable closure of $K$ and 
$\bar{v}$ is an extension of $v$ to $\bar{K}$. We denote by $(K_v^{sh}, v^{sh})$ the strict henselization 
of $(K,v)$ (inside $(\bar{K}, \bar{v})$).
A  finite separable extension $(L,w)/(K,v)$ of valuation fields is called {\em unramified} (resp. {\em tame}), if $K_v^{sh}=L_w^{sh}$ (resp.
$([L_w^{sh}:K_v^{sh}],p)=1$, where $p$ is the exponential characteristic of the residue field of $K$).
The {\em tame closure} $(K^t, v^t)$ of $(K,v)$ is the union of all finite tame Galois extensions of $(K^{sh}, v^{sh})$.
The field $K^t$ is also the fixed field of $\bar{K}$ under the tame ramification group 
\[R_{\bar{v}/v}:=\{\sigma \in \mathrm{Gal}(\bar{K}/K)\mid \sigma(\cO_{\bar{v}})\subset \cO_{\bar{v}} \text{ and }
\sigma(x)/x-1\in \fm_{\bar{v}} \textrm{ for all }x\in \bar{K}^\times\}.\]

We record the following well-known lemma  for later reference. 
\begin{lemma}\label{lem:tame-Gal-closure}
\begin{enumerate}[label=(\arabic*)]
\item\label{lem:tame-Gal-clousre1} Let $(L,w)/(K,v)$  be a finite separable extension of valuation fields.
Let $N/K$ be a Galois hull of  $L/K$ and let $\tilde{w}$ be an extension of $w$ to $N$.
Then $(L,w)/(K,v)$ is tame if and only if $(N,\tilde{w})/(K,v)$ is tame. 

In particular $(L,w)/(K,v)$ is tame if and only if $(L,w)$ is a subextension of $(K^t,v^t)/(K,v)$.
\item\label{lem:tame-Gal-clousre2} 
Let $(L,w)/(K,v)$ be a tame extension and let $(K',v')/(K,v)$ be any algebraic extension of valuation fields.
Let $L\cdot K'$ be the composition field in an algebraic closure of $K$ and let $w'$ be a valuation extending $v'$.
Then $(L\cdot K', w')/(K', v')$ is tame.
\end{enumerate}
\end{lemma}
\begin{proof}
\ref{lem:tame-Gal-clousre1}. Note that $N^{sh}_{\tilde{w}}$ is a  Galois hull  of $L^{sh}_w/K^{sh}_v$.   
Therefore we may assume $K,L,N$ are strictly henselian valuation fields of characteristic $p>0$. 
Thus if $(N,\tilde{w})/(K,v)$ is tame then $[N:K]=[N:L]\cdot [L:K]$ is prime to $p$ and hence $(L,w)/(K,v)$ is tame as well.
Now assume $(L,w)/(K,v)$ is tame. Denote by $G_K\supset G_L\supset G_N$ the absolute Galois groups with respect to 
a fixed separable  closure $\bar{K}$ of $K$, and  by $P$ the  pro-$p$-Sylow subgroup of $G_K$, which is a normal subgroup.
The indices satisfy the following equality (of supernatural numbers)
\[[G_K:G_L]\cdot [G_L:P\cap G_L]= [G_K:P]\cdot [P:P\cap G_L].\]
As $P$ is a normal subgroup of $G_K$, the intersection $P\cap G_L$ is a normal subgroup of $G_L$
and we have an inclusion of profinite groups $G_L/G_L\cap P\inj G_K/P$. 
Hence $[G_K:P]$ and $[G_L:P\cap G_L]$ are prime to $p$.
By assumption $[G_K:G_L]=[L:K]$ is prime to $p$ as well. Thus $[P:P\cap G_L]=1$, i.e., $P=P\cap G_L$.
The Galois hull of $L/K$ is the composition field (inside $\bar{K}$) of all the $\sigma(L)$, 
where $\sigma$ runs through all the embeddings  $L\inj \bar{K}$. Extending these $\sigma$'s to $K$-automorphisms of $\bar{K}$, 
we find  $G_{\sigma(L)}=\sigma G_L\sigma^{-1}$. Hence $G_N=\cap_{\sigma} \sigma G_L\sigma^{-1}$. 
As $P$ is a normal subgroup of $G_K$ it follows that $P$ is contained in $G_N$ as well. 
Thus  $[G_K:P]=[G_K:G_N]\cdot [G_N:P]$ is prime to $p$ and hence so is $[N:K]=[G_K:G_N]$.

\ref{lem:tame-Gal-clousre2} follows from \ref{lem:tame-Gal-clousre1} and the fact that ${K'}^t= K'\cdot K^t$, 
see \cite[6.2.18]{Gabber-Ramero}.
\end{proof}

Recall that for any morphism of schemes $X\to \tX$, $\Spa(X,\tX)$ is the topological space whose underlying set is the set of triples $(x,v,\epsilon)$ such that $x\in X$ is a point, $v$ is a valuation on $k(x)$ and $\epsilon\colon \Spec(\cO_v)\to \tX$ is a map compatible with $\Spec(k(x))\to X$. 
\begin{defn}\label{def;XSt}
Let $S\to \tS$ be a morphism of schemes. 
\begin{itemize}
\item[(1)]
Let $\SchSS$ be the category of pairs $(X,\tX)$, where $X\to \tX$ is a map of qcqs $\tS$-schemes 
such that $X\to \tS$ factors through $S$.
Morphisms $(X,\tX)\to (Y,\tY)$ are pairs of morphisms $f: X\to Y$ and $\tf:\tX\to \tY$ satisfying the obvious compatibility.
\item[(2)]
A morphism $(f,\tf):(X,\tX)\to (Y,\tY)$ in $\SchSS$ is a \emph{modification} if $f$ is an isomorphism and $\tf$ is proper, and \emph{integral birational} if $f$ is an isomorphism and $\tf$ is integral. We say that $(f,\tf)$ is a \emph{quasi-modification} if $f$ is an isomorphism and $\tf$ is ift, separated and universally closed.
\item[(3)] A morphism $(f,\tf):(V,\tV)\to (U,\tU)$ in $\SchSS$ is \emph{strict \'etale} if $\tf$ is \'etale, $V=\tV\times_{\tU} U$ and $f= \tf\times_{\tU} \id$.
\item[(4)]
A morphism $(f,\tf):(V,\tV)\to (U,\tU)$ in $\SchSS$ is \emph{tame over $(x,v,\epsilon_v)\in \Spa(U,\tU)$} if $f$ is \'etale over a neighborhood of $x$ and there is $(y,w,\epsilon_w)\in \Spa(V,\tV)$ such that $f(y)=x$, $w_{|k(x)}=v$, and $w/v$ is tamely ramified and the following diagram commutes:\[
\begin{tikzcd}
\Spec(\cO_w)\ar[r,"\epsilon_w"]\ar[d,"(f)_{|y,w\geq 0}"]&\tV\ar[d,"\tf"]\\
\Spec(\cO_v)\ar[r,"\epsilon_v"]&\tU.
\end{tikzcd}
\]
A morphism $(f,\tf):(V,\tV)\to (U,\tU)$ is a \emph{tame covering} if it is tame over all points of $\Spa(U,\tU)$.
\end{itemize}
\end{defn}
\begin{rmk}\label{rmk1;modification}
    \begin{enumerate}
        \item By Raynaud--Gruson \cite[\href{https://stacks.math.columbia.edu/tag/081T}{Tag 081T}]{stacks-project}, every modification $(f,\tf):(V,\tV)\to (U,\tU)$ is dominated by an \emph{admissible blow-up}, i.e. a blow-up $\pi_{\cA}: \Bl_{\cA}(\tU) \to \tU$ in a  quasi-coherent ideal $\cA\subset \cO_{\tX}$ of finite type such that the support of $\cO_{\tX}/\cA$ is contained in $\tU\setminus U$.
        \item Note that by the valuative criterion, if $\tf\colon \tU'\to \tU$ is separated and universally closed and $U\to \tU'$ is any map, we have a bijection $\Spa(U,\tU)\cong \Spa(U,\tU')$.
        \item By Theorem \ref{thm:ift-sep-uc-limit-modifications}, every quasi-modification $(f,\tf)\colon (X,\tX)\to (Y,\tY)$ is a cofiltered limit of modifications $(f_i,\tf_i)\colon (X_i,\tX_i)\to (Y,\tY)$ such that for every $i$ the maps $(X,\tX)\to (X_i,\tX_i)$ are integral birational.
    \end{enumerate}
    
\end{rmk}

The two sites below are very much inspired by the definition of the \'etale and tame site of a Huber pair (see \cite{Huebner2021} and \cite{HS2020}).

\begin{defn}\label{def;tame site}
Let $X\hookrightarrow \tX$ be a quasi-compact open immersion of qcqs schemes. 
 Let $(X,\tX)_\tau$ be the full subcategory of $\SchXX$
 consisting of objects $(f,\tf)\colon (U,\tU)\to (X,\tX)$
 with $f$ \'etale and $\tf$ ift, and $U\hookrightarrow \tU$  a quasi-compact open immersion. This category has fiber products given by\[
    (V_1,\tilde{V_1})\times_{(U,\tilde{U})} (V_2,\tilde{V_2}) = (V_1\times_U V_2, \tilde{V_1}\times_{\tilde{U}}\tilde{V_2}), \]
    and terminal object $(X,\tX)$, so for $I$ an indexing diagram and $\{(U_i,\tU_i)\}_{i\in I}\subseteq (X,\tX)_\tau$, the limit (if it exists) is computed as $(\lim_i U_i,\lim \tU_i)$.
    On this category, we will consider the following three topologies:
    \begin{enumerate}
    \item The \emph{strict \'etale} topology which is generated by strict \'etale coverings
 \item The \emph{v-\'etale} topology which is generated by strict \'etale coverings and quasi-mo\-di\-fi\-ca\-tions.
 \item The \emph{tame} topology which is generated by tame coverings, where a family $\{(f_i,\tf_i): (V_i,\tV_i) \to (U,\tU)\}_i$ of maps in $(X,\tX)_\tau$ is a tame covering if for every $(x,v,\epsilon_v)\in \Spa(U,\tU)$,
there is $i\in I$ such that $(V_i,\tV_i) \to (U,\tU)$ is {tame} over $(x,v,\epsilon_v)$.
\end{enumerate}
We let $(X,\tX)_{\set}$, $(X,\tX)_{\vet}$ and $(X,\tX)_{t}$ denote the strict \'etale, the v-\'etale and the tame site on $(X,\tX)_\tau$, respectively.
We have morphisms of sites:
\eq{morphism-tameet}{
(X,\tX)_t \rmapo{\nu} (X,\tX)_{\vet}\rmapo{\mu} (X,\tX)_{\set}}
corresponding to the inclusion functors.

\end{defn}  
In \cite{DAddezio} D'Addezio defines a v-\'etale topology for marked schemes. This definition is similar in spirit but 
cannot be compared directly with the one above. We warn the reader that D'Addezio's definition is close to the definition of the tame site above, except that in his definition
the tameness condition is dropped, i.e., his v-\'etale site allows covers which are ramified with ramification controlled by the markings.
In the current article we will only consider the  v-\'etale site defined in Definition \ref{def;tame site} above.

\begin{rmk}\label{rmk-def;XSt}
\begin{enumerate}[label=(\arabic*)]
\item\label{rmk-def;XSt1}
The tame topology is finitary, i.e., any covering can be refined by a covering of the form 
$g:(V,\tV)\to (U,\tU)$ for $(V,\tV)$ and $(U,\tU)$ in $(X,\tX)_t$. 
This follows from the fact that $\Spa(X,\tX)$ 
is a quasi-compact topological space by \cite[Lemma 4.3]{HS2020} with the topology defined there.
\item\label{rmk-def;XSt2} Let $(U,\tU)\in (X,\tX)_t$ and let $\ol{U}$ be the closure of $U$ in $\tU$. Then $(U,\ol{U})\to (U,\tU)$ is a quasi-modification, hence it is a v-\'etale covering and a tame covering.
\item\label{rmk-def;XSt3}
Note that for any $(U,\tU)\in (X,\tX)_t$, there exists  a quasi-coherent ideal sheaf $\cI\subset \sO_{\tU}$
{\em of  finite type} such that the support of $\sO_{\tU}/\cI$ is equal to $\tU\setminus U$. 
This uses that $U\inj\tU$ is a quasi-compact open immersion
and that $\tU$ is qcqs.
In particular, blowing up $\tU$  in such an ideal we obtain
a modification $(U,\bar{U})\to (U,\tU)$ such that the complement $\bar{U}\setminus U$ is the support of an effective Cartier divisor.
\end{enumerate}
\end{rmk}

\begin{example}\label{ex;lem:Gat}
Consider the presheaf $\sO^t$ on $\SchSS$ given by 
$\sO^t(U,\tU)=\sO(\tU^{\rm int})$ for $(U,\tU)\in \SchSS$,
where $\tU^{\rm int}$ denotes the integral closure of $\tU$ in $U$.
We will see in Lemma \ref{lem:Gat} that this defines a tame sheaf.
\end{example}

\begin{lemma}\label{lem:blow-up-inv}
Let $(X,\tX')\to (X,\tX)$ be a quasi-modification. 
Then, for any sheaf $F\in \Sh((X,\tX)_{\vet})$ and $(U,\tU)\in (X,\tX)_\tau$, we have an isomorphism
\[F(U,\tU)\cong F(U,\tU\times_{\tX} \tX').\]
In particular, the restriction functors along the inclusion $(X,\tX')_\tau\to (X,\tX)_\tau$
\[{\Sh}((X,\tX)_{\vet})\to {\Sh}((X,\tX')_{\vet}) \qaq 
{\Sh}((X,\tX)_t)\to {\Sh}((X,\tX')_t)\]
are equivalences of topoi so that for a presheaf $F$ of abelian groups, we have equivalences
\[R\Gamma((X,\tX)_{\vet},a_{\vet}F)\cong R\Gamma((X,\tX')_{\vet},a_{\vet}F)\]
and
\[R\Gamma((X,\tX)_{t},a_{t}F)\cong R\Gamma((X,\tX')_{t},a_{t}F),\]
where $a_t$ and $a_{\vet}$ denote the respective sheafification functor.
\end{lemma}
\begin{proof}
We prove it for $\vet$. The proof for $t$ is the same.
Let $F$ be a sheaf of sets on $(X,\tX)_{\vet}$. For $(U,\tU)\in (X,\tX)_\tau$ and $\tU'=\tU\times_{\tX} \tX'$,
the map induced by projection $(U,\tU')\to (U,\tU)$ and the map induced by the diagonal 
$\delta: (U, \tU')\to (U,\tU'\times_{\tU} \tU')$ are modifications and thus are coverings in $(X,\tX)_{\vet}$. Hence,  
\[ \delta^*: F(U,\tU'\times_{\tU} \tU') \to F(U,\tU')\]
is injective and we find
\[\begin{aligned}
F(U,\tU) &\cong {\rm eq}\left(F(U,\tU')
\overset{pr_1^*}{\underset{pr_2^*}{\rightrightarrows}} F(U,\tU'\times_{\tU} \tU')\right)\\
&\cong {\rm eq}\left(F(U,\tU')\overset{id}{\underset{id}{\rightrightarrows}} F(U,\tU')\right)= F(U,\tU').\\
\end{aligned}\]
Hence, the functor 
\[{\Sh}((X,\tX')_{\vet})\to {\Sh}((X,\tX)_{\vet}),\quad 
G\mapsto \left((U,\tU)\mapsto G(U, \tU\times_{\tX} \tX')\right)\]
gives a quasi-inverse of the restiction functor ${\Sh}((X,\tX)_{\vet})\to {\Sh}((X,\tX')_{\vet})$.
\end{proof}

\begin{defn}\label{def;int site}
Let $(X,\tX)_{{\rm affine},\tau}$ be the subcategory of $(X,\tX)_{\tau}$ whose objects are pairs $(U,\tilde U)$, with $U$ and $\tilde U$ affine,
equipped with the $\vet$ and tame topologies by restriction.

Let $(X,\tX)_{{\rm int},\tau}$ be the full subcategory of $(X,\tX)_{{\rm affine},\tau}$ whose objects are pairs $(U,\tU)$ 
with $\cO(\tU)\to \cO(U)$ injective and integrally closed, equipped with v-\'etale and tame topologies by restriction.
\end{defn}

\begin{lemma}\label{lem:reduction-affine}
    The inclusion functors $(X,\tX)_{{\rm affine},\vet}\to (X,\tX)_{\vet}$ and  
    $(X,\tX)_{{\rm affine},t}\to (X,\tX)_{t}$ induce equivalences on the associated topoi.
    \begin{proof} It is straightforward to check the properties (1)-(5) of \cite[\href{https://stacks.math.columbia.edu/tag/03A0}{Tag 03A0}]{stacks-project}.
    \end{proof}
\end{lemma}

\medbreak

\section{Computation of the v-\'etale topology}\label{Compvet}\label{sec:3}

In this section we study some properties of the v-\'etale topology.

 \begin{lemma}\label{lem0:refine-vet}
      For every composition $\cV\xrightarrow{\phi_1}  \cU\xrightarrow{\phi_2} \cY$ in $\SchSS$ with $\phi_2$ strict \'etale and $\phi_1$ a quasi-modification, there is  a quasi-modification $\cT\to \cY$ such that $\cU\times_\cY \cT\to \cY$ factors through a quasi-modification $\cU\times_\cY \cT\to\cV$.
        \end{lemma}
\begin{proof}
   First, we show that the result holds for $\phi_1$ a modification. This is classical (see \cite[Proposition 12.27]{MVW}): 
   let $\cU=(U,\tU)$, $\cV=(V,\tV)$, and $\cY=(Y,\tY)$. As observed in 
        Remark \ref{rmk-def;XSt}\ref{rmk-def;XSt2}, we can assume that $Y$ dense in $\tY$.
        By Raynaud--Gruson \cite[\href{https://stacks.math.columbia.edu/tag/081R}{Tag 081R}]{stacks-project}, there exists a 
        $Y$-admissible blow-up $\tT\to \tY$ such that the strict transform $\tV'$ of $\tV$ over $\tT$ is flat of finite presentation over $\tT$. 
        Note that the map $\tV'\to \tT$ factors as
        $\tV'\xrightarrow{\alpha} \tU\times_{\tY}\tT \xrightarrow{\beta} \tT$, where $\beta$ is \'etale and $\alpha$ is proper inducing an isomorphism over the dense open  $U=\tU\times_{\tY}\tT\times_{\tY} Y$. Moreover, $\alpha$ is flat by 
        \cite[Lemma 4.15]{Kel-S23} and thus an isomorphism by \cite[Lemma 4.16]{Kel-S23}.
        Hence, we get a morphism $\tU\times_{\tY}\tT \to \tV$, which is proper and an isomorphism over $U$. 
        
        Let now $\phi_1$ be a quasi-modification. By Remark \ref{rmk1;modification} (3), there exists $\cV_0=(U,\tV_0)$ such that $\phi_1$ factors as a composition of an integral birational morphism $\cV\to \cV_0$ and a modification $\cV_0\to \cU$.
By the previous paragraph, there exists a modification $\cT'=(Y,\tT')\to \cY$ such that for $\tU':=\tU\times_{\tY} \tT'$, the map $\cU'=(U,\tU')\to \cU$ factors through $\cV_0$, hence we have a commutative diagram\[
\begin{tikzcd}
    \cV\times_{\cV_0}\cU'\ar[r,"\phi_1'"]\ar[d,"q_0"]&\cU'\ar[rr,"\phi_2'"]\ar[d]&&\cT'\ar[d,"\psi_0"]\\
    \cV\ar[r]&\cV_0\ar[r]&\cU\ar[r]&\cY
\end{tikzcd}\]
with $\phi_1'$ integral birational, $\phi_2'$ strict \'etale and $q_0$ and $\psi_0$ modifications.
Let $\tT$ be the integral closure of $\tT'$ in $Y$ and put 
$\tW:=\tU'\times_{\tT'}\tT\cong \tU\times_{\tY} \tT$.
Then, 
$\tW$ is the integral closure of $\tU'$ in $U$ by \cite[\href{https://stacks.math.columbia.edu/tag/03GE}{Tag 03GE}]{stacks-project} 
since $\tU'\to \tT'$ is \'etale and $U=\tU'\times_{\tT'} Y$. As $\phi_1'$ is integral birational, the $\tU'$-scheme $\tW$ is isomorphic to
the integral closure of $\tV\times_{\tV_0} \tU'$ in $U$. 
Hence $\tW\to\tU'$ factors via $\tW=(\tV\times_{\tV_0} \tU')^{\rm int}\to \tV\times_{\tV_0} \tU'$.
Therefore we can take $\cT=(Y,\tT)$ and this completes the proof of the lemma. 
\end{proof}

\begin{lemma}\label{lem:refine-vet}
      For $F\in \Shv((X,\tX)_{{\set}})$ and $\cU\in (X,\tX)_{\tau}$, we have
    \begin{equation}\label{eq;lem:refine-vet}
    a_{\vet}(F)(\cU) = \colim_{\cV\to \cU} F(\tV) = \colim_{\cW\to \cU} F(\tW)
    \end{equation}
    where the first colimit runs along all quasi-modifications $\cV\to \cU$ with $\cV=(V,\tV)$ and $\tV$ integrally closed in $V$,
    and the second colimit runs along all admissible blow-ups $\cW\to \cU$.
    \begin{proof}
    Let $\alpha F$ be the presheaf defined by \[
    \cU\in (X,\tX)_{\tau}\mapsto \colim_{\cU'\to \cU} F(\tU'),
    \]
    where the colimit runs through all quasi-modifications. By taking the integral closure, 
    every quasi-modification is dominated by a quasi-modification $\cV=(V,\tV)\to \cU$ with $\tV$ integrally closed in $V$, 
    and by Theorem \ref{thm:ift-sep-uc-limit-modifications}, every quasi-modification is a cofiltered limit of modifications, and by Raynaud--Gruson \cite[\href{https://stacks.math.columbia.edu/tag/081T}{Tag 081T}]{stacks-project}, any modification is dominated by an admissible blow-up. Therefore the two colimits in \eqref{eq;lem:refine-vet} are isomorphic to $\alpha F(\cU)$.
    First, we claim that $\alpha F\in \Shv((X,\tX)_{\vet})$. Indeed, by construction $\alpha F$ sends quasi-modifications to isomorphisms, so it has descent for those coverings. It remains to prove that $\alpha F$ has descent for every strict \'etale covering $\{\cU_i\to \cU\}_{i\in I}$. By a standard reduction, we may assume $I=\{1,2\}$.
Using $F\in \Shv((X,\tX)_{\set})$, for any quasi-modification $\cV\to \cU$, we have 
\[ F(\cV) = F(\cV\times_\cU\cU_1) \times_{F(\cV\times_{\cU} \cU_{12})} F(\cV\times_\cU\cU_2) , \]
where $\cU_{12} = \cU_1 \times_\cU \cU_2$.
Taking the colimit over $\cV$ and using Lemma \ref{lem0:refine-vet} and the fact that filtered colimits commute with fiber products, we get 
\[ \alpha F(\cU) = \alpha F(\cU_1) \times_{\alpha F(\cU_{12})} \alpha F(\cU_2).\]
   \medbreak
   
   Thus, we get a functor 
   $\alpha: \Shv((X,\tX)_{\set})\to \Shv((X,\tX)_{\vet})$.
   It suffices to show that it is a left adjoint of the inclusion $i: \Shv((X,\tX)_{\vet})\to \Shv((X,\tX)_{\set})$.
   By construction we have a natural transformation ${\rm id}\to i\alpha$ and by Lemma \ref{lem:blow-up-inv} also
   a natural isomorphism $\alpha i\xrightarrow{\simeq} {\rm id}$. 
   The statement thus follows from \cite[IV, \S1, Theorem 2(v)]{MacLane}.
\end{proof}
\end{lemma}

\begin{lemma}\label{lem2:refine-vet}
Let $I$ be a sheaf of abelian groups on $(X,\tX)_\set$.
If $I$ is flabby, that is
$H^i(\cV_{\set},F)=0$, for any $i>0$ and $\cV\in (X,\tX)_\tau$, then $a_{\vet}I$ is flabby as a $\vet$ sheaf.
\end{lemma}
\begin{proof}
By \cite[III, Proposition 2.12]{MilneEtCoh}, $I$ is flabby as a $\set$ sheaf if and only if  
for any $\cU\in (X,\tX)_\tau$ and a strict \'etale covering $\cU'\to \cU$,
the \v Cech complex\[
0\to I(\cU)\to I(\cU')\to I(\cU'\times_{\cU} \cU')\to \cdots  
\]
is exact. This implies that for any modification $\cV\to \cU$, the \v Cech complex\[
0\to I(\cV)\to I(\cV\times_{\cU} \cU')\to I(\cV\times_{\cU} \cU'\times_{\cU} \cU')\to \cdots  
\]
is also exact. Noting that filtered colimits are exact, Lemma \ref{lem:refine-vet} together with Lemma \ref{lem0:refine-vet}
implies
\[
0\to a_{\vet}I(\cU)\to a_{\vet}I(\cU')\to a_{\vet}I(\cU'\times_{\cU} \cU')\to \cdots  
\]
is exact. Noting $a_{\vet}I(\cU)\simeq a_{\vet}I(\cV)$ for any modification $\cV\to \cU$,
this implies that $a_{\vet}I$ is flabby on $(X,\tX)_{\vet}$ again by \cite[III, Proposition 2.12]{MilneEtCoh}. 
\end{proof}

\begin{defn}\label{defn:lambda}
    For $\cU=(U,\tU)\in (X,\tX)_{\vet}$, let $\lambda^{\cU}\colon (U,\tU)_{\vet}\to \tU_{\et}$ be the morphism of sites defined by the functor 
    \[ \tU_{\et} \to (U,\tU)_{\vet}, \quad \tV\mapsto (\tV\times_{\tU} U,\tV).\]
     It is clear by construction 
     \[H^q((U,\tU)_{\set},F)=H^q(\tU_{\et},\lambda^{\cU}_*F).\]
\end{defn}
We can now compare the $\vet$ cohomology with the \'etale cohomologies of the sheaves $\lambda_*^{\cU}F$. This result should be compared with \cite[Corollary 5.7]{Huebner2025}, as well as \cite[Theorem 5.1.2]{BPO} and \cite[Theorem 2]{MotModulusI}, where analogous results were obtained.
\begin{lemma}\label{lem:etaleCoh}
For a sheaf of abelian groups  $F$ on $(X,\tX)_{\vet}$ and $\cU=(U,\tU)$, there is a natural isomorphism
\[ H^i(\cU_{\vet},F)=\varinjlim_{\cV\to \cU} H^i(\tV_{\et},\lambda^{\cV}_* F) = \varinjlim_{\cW\to \cU} H^i(\tW_{\et},\lambda^{\cW}_* F),\]
where the first colimit runs along all quasi-modifications $\cV\to \cU$ with $\cV=(V,\tV)$ and $\tV$ 
integrally closed in $V$, and the second colimit runs along all admissible blow-ups $\cW\to \cU$.
\end{lemma}
\begin{proof}
This is very similar to \cite[Theorem 5.1.2]{BPO}. 
Let  $F_{|(X,\tX)_{\set}}\to I^\bullet$ be a flabby resolution of the restriction of $F$ to $(X,\tX)_{\set}$.  
By Lemma \ref{lem2:refine-vet}, this gives a flabby resolution  $F\to a_{\vet}I^{\bullet}$ on $(X,\tX)_{\vet}$. Therefore, by Lemma \ref{lem:refine-vet}\[
H^i(\cU_{\vet},F)= \pi_{-i} (a_{\vet}I^\bullet(\cU)) = \pi_{-i} (\colim_{\cV\to\cU}I^\bullet(\cV)) = \colim_{\cV\to\cU}H^i(\cV_{\set},F) = \colim_{\cV\to\cU} H^i(\tV_{\et},\lambda^{\cV}_*F).
\]
\end{proof}

\section{Acyclic objects}\label{sec:4}

In this section, we show the existence of acyclic pro-covers, Lemma \ref{lem:existence-acyclic}, which will be used in 
the comparison of the tame cohomology with \v Cech cohomology (see Lemma \ref{thm:main-cech}).

\begin{defn}
We let $\widetilde{(X,\tX)_{\tau}}$ be the full subcategory of $\SchXX$ consisting of affine pairs $(U,\tU)$ such that there exists a cofiltered system $\{(U_i,\tU_i)\}_{i\in I}$ in $(X,\tX)_{{\rm affine},\tau}$ such that $U=\lim U_i$ and $\tU=\lim \tU_i$ as schemes. Notice that in this case the map $U\to X$ is no longer in general \'etale (but rather, pro-\'etale), the 
map $U\to \tU$ is no longer in general a quasi-compact open immersion and the map $\tU\to \tX$ is no longer ift.

We also consider the full subcategory $\widetilde{(X,\tX)_{{\rm int}, \tau}}$ of $\widetilde{(X,\tX)_{\tau}}$ whose objects are limits of $(U_i,\tU_i)\in (X,\tX)_{{\rm int},\tau}$ (cf. Definition \ref{def;int site}).
\end{defn} 

\begin{rmk}\label{rmk;int}
For $\cY=(\Spec(A),\Spec(\tA))\in \widetilde{(X,\tX)_{{\rm int},\tau}}$, we observe that:
\begin{enumerate}[label=(\arabic*)]
\item\label{rmk;int2}
$\tA\to A$ is injective and integrally closed, as filtered colimits are exact and the integral closure commutes with filtered colimits.
\item\label{rmk;int3}
If $A=A_1\times A_2$ is a product of rings and $\tA_i$ are the integral closures of $\tA$ in $A_i$ for $i=1,2$, 
we have $\cY=\cY_1\sqcup\cY_2$ with $\cY_i=(\Spec(A_i),\Spec(\tA_i))$. 
\end{enumerate}
\end{rmk}

\begin{defn}\label{def;acyclic}
    We say that $(T,\tT)\in \widetilde{(X,\tX)_{\tau}}$ is \emph{v-\'etale (resp. tamely) acyclic} 
    if for every v-\'etale (resp. tame) covering $(V,\tV)\to (U,\tU)$ in $(X,\tX)_\tau$ and a map $\psi: (T,\tT)\to (U,\tU)$ in $\widetilde{(X,\tX)_{\tau}}$, there exists a map $(T,\tT)\to (V,\tV)$ which lifts $\psi$.
   \end{defn}

\begin{lemma}\label{lem:existence-acyclic}
    For every $\cY\in (X,\tX)_{\tau}$, there is $\cW\in \widetilde{(X,\tX)_\tau}$ v-\'etale (resp. tamely) acyclic such that $\cW \to \cY$ is a cofiltered limit of v-\'etale (resp. tame) coverings 
    $\cW_i\to \cY$ in $(X,\tX)_{{\rm affine},\tau}$ with $\cW_i\in (X,\tX)_{{\rm int},\tau}$.
    \end{lemma}
    \begin{proof}
    We prove the lemma only for the tame topology. The proof for the v-\'etale topology is the same. 
        We can suppose $\cY:=(\Spec(A),\Spec(\tA))\in (X,\tX)_{\rm{affine},t}$.
        We use the same strategy of \cite[Lemma 2.2.7]{BS-proet} (see \cite[Proposition 7.12]{HS2020}). Let $I$ be the set of isomorphism classes of coverings $\cU\to\cY$ in $(X,\tX)_{\rm{affine},t}$. For each $i \in I$, pick a representative $(\Spec(B_i),\Spec(\tB_i))\to \cY$ and set
        \begin{equation}\label{eq1;lem:existence-acyclic}
           A_1 := \colim_{J\subset I \;\mathrm{finite}} \bigotimes_{j\in J}B_j\qquad \tA_1 := \colim_{J\subset I \;\mathrm{finite}} \bigotimes_{j\in J} \tB_j,
        \end{equation}
        where the tensor products are over $A$ and $\tA$, respectively.
      By construction, we can write 
      \[\cY_1:=(\Spec(A_1),\Spec(\tA_1))=\varprojlim_{\lambda_1\in \Lambda_1} \cY_{\lambda_1} \;\text{ with } \cY_{\lambda_1}=(\Spec(A_{\lambda_1}),\Spec(\tA_{\lambda_1}))\]
      as a cofiltered limit of coverings $\cY_{\lambda_1}\to \cY$ in $(X,\tX)_{\rm{affine},t}$ such that for every covering $\cU\to \cY$ in $(X,\tX)_t$, the map $\cY_1 \to\cY$ factors through $\cU$.
      For each $\lambda_1\in \Lambda_1$, let $I_{\lambda_1}$ be the set of isomorphism classes of coverings $\cU\to\cY_{\lambda_1}$ in $(X,\tX)_{\rm{affine},t}$ and apply the same construction as \eqref{eq1;lem:existence-acyclic} to $(\cY_{\lambda_1},I_{\lambda_1})$ instead of $(\cY,I)$ to get 
      $(A_{2,\lambda_1}, \tA_{2,\lambda_1})$ instead of $(A_1,\tA_1)$ and put $\cY_{2,\lambda_1}=(\Spec(A_{2,\lambda_1}), \Spec(\tA_{2,\lambda_1}))$.
      Then, for every covering $\cU\to \cY_{\lambda_1}$, the map $\cY_{2,\lambda_1}\to \cY_{\lambda_1}$ factors through $\cU$. Put
      \begin{equation}\label{eq2;lem:existence-acyclic}
           A_2 :=\colim_{J\subset \Lambda_1 \;\mathrm{finite}} \bigotimes_{\lambda_1\in J}(A_{2,\lambda_1}\otimes_{A_{\lambda_1}} A_1), \qquad  \tA_2 :=\colim_{J\subset \Lambda_1 \;\mathrm{finite}} \bigotimes_{\lambda_1\in J} (\tA_{2,\lambda_1}\otimes_{\tA_{\lambda_1}} \tA_1),
        \end{equation}
    where the tensor products are over $A_1$ and $\tA_1$, respectively.  
    Noting $\cY_{2,\lambda_1}\to \cY_{\lambda_1}$ and $\cY_1\to \cY$ are cofiltered limits of coverings,
    we can write 
      \[\cY_2:=(\Spec(A_2),\Spec(\tA_2))=
        \varprojlim_{\lambda_2\in \Lambda_2} \cY_{\lambda_2}\]
       as a cofiltered limit 
      of coverings $\cY_{\lambda_2}\to \cY$ in $(X,\tX)_{\rm{affine},t}$ such that for every $\lambda_1\in \Lambda_1$ and covering $\cU\to \cY_{\lambda_1}$ in $(X,\tX)_t$, the map 
      $\cY_2 \to\cY_1\to \cY_{\lambda_1}$ factors through $\cU$.
       Iterating the construction, we get a sequence in $\widetilde{(X,\tX)_\tau}$
       \[ \cdots\to \cY_3\to \cY_2\to \cY_1\to \cY\,\text{ with }\;
       \cY_n=\varprojlim_{\lambda_n\in \Lambda_n}  \cY_{\lambda_n}\]
       such that for every $\lambda_n\in \Lambda_n$,
      $\cY_{\lambda_n}\to \cY$ is a covering in $(X,\tX)_{\rm{affine},t}$ and for every covering $\cU\to \cY_{\lambda_n}$ in $(X,\tX)_t$, the map 
      $ \cY_{n+1}\to \cY_n\to \cY_{\lambda_n}$ factors through $\cU$.       
  Set $\cW=\varprojlim_n \cY_n\in \widetilde{(X,\tX)_\tau}$. By the construction, $\cW$ is a cofiltered limit of tame coverings of $\cY$. 
  It remains  to show that $\cW$ is tamely acyclic. Let $\cV\to \cU$ be a covering in $(X,\tX)_{\rm{affine},t}$ and
  $\phi:\cW\to \cU$ be a map in $\widetilde{(X,\tX)_\tau}$. 
  By Lemma \ref{lem:ift-fp-limits} (with $(f,\tf)$ induced by $\cU\to\cX$), $\phi$ factors through $\cY_{\lambda_n}$ for some $\lambda_n$. 
  Since $\cY_{\lambda_n}\times_\cU \cV\to \cY_{\lambda_n}$ is a covering in $(X,\tX)_t$, the map $\cY_{n+1}\to\cY_n\to \cY_{\lambda_n}$ factors through $\cY_{\lambda_n}\times_\cU \cV$ so that $\phi$ lifts to a map $\cW\to \cV$.
   This completes the proof.
    \end{proof}

\section{Fiber functors}\label{sec:5}

\def\Fib{\mathrm{Fib}}
\def\aff{\mathrm{aff}}

In this section, we characterize fiber functors of the topoi of the sheaves of sets on $(X,\tX)_{\vet}$ and $(X,\tX)_t$ (see Proposition \ref{prop:tamely-acyclic}). First, we recall the following.
\begin{defn} \label{defi:FF}
Let $(C,\gamma)$ be a site admitting finite limits.
Recall that a \emph{fiber functor} of a topos $\Sh(C,\gamma)$ of sheaves of sets, is  a functor $\phi: \Sh(C,\gamma) \to \Set$ which preserves colimits and finite limits. Let $\Fib(\Sh(C,\gamma))$ denote the category of fiber functors of $\Sh(C,\gamma)$.
\end{defn}

In what follows, let $\gamma$ denote either the v-\'etale topology or the tame topology on $(X,\tX)_\tau$.
By Lemma \ref{lem:reduction-affine}, there is an equivalence of categories of fiber functors
\[\Fib(\Sh((X,\tX)_{{\rm affine},\gamma})) \simeq \Fib(\Sh((X,\tX)_{\gamma})).\]
Hence, by \cite[Pro.7.13]{Joh77} and Remark \ref{rmk-def;XSt}(1), there is a bijection between fiber functors of 
$\Sh((X,\tX)_{\gamma})$ and cofiltered pro-objects 
\eq{Pbullet}{\cP_\bullet=\catprojlim {\lambda\in \Lambda} \cP_\lambda \;\text{ with }  \cP_\lambda=(P,\tP)\in (X,\tX)_{{\rm affine},\gamma}}
indexed by 
a cofiltered category $\Lambda$, which satisfies the $\gamma$-locality condition: For every $\gamma$-covering 
$\cV \rmapo{u} \cU$, the morphism of sets
\[
 \varinjlim_{\lambda\in \Lambda} \Hom_{(X,\tX)_\tau}(\cP_\lambda, \cV) \to \varinjlim_{\lambda\in \Lambda} \Hom_{(X,\tX)_\tau}(\cP_\lambda,  \cU)\]
 is surjective. By Lemma \ref{lem:ift-fp-limits}, the latter condition is equivalent to that
$\cT=\varprojlim_{\lambda\in \Lambda} \cP\in \widetilde{(X,\tX)_\tau}$ is $\gamma$-acyclic in the sense of Definition \ref{def;acyclic} and the corresponding fiber functor is given by 
\eq{phiT}{ \phi_{\cT}: \Sh((X,\tX)_{\gamma})\to\Set\;;\; F \to F(\cT):=\varinjlim_{\lambda}F(\cP_\lambda).}

\medbreak

\begin{prop}\label{prop:tamely-acyclic}
A pair $\cT=\lim_{i\in I} \cT_i\in \widetilde{(X,\tX)_{\tau}}$ is v-\'etale local (resp. tame local)  if and only if $\cT$ is a coproduct of objects of the form $(\Spec(S),\Spec(\tS))$ such that
$\tS$ is strictly henselian local and $S$ is henselian local
and that $\tS=S\times_{k}\cO_{v}$, where $k$ is the residue field of $S$ equipped with a valuation $v$ such that $(k,v)$ is strictly henselian (resp. $(k,v)$ is tamely closed), and $\cO_{v}\subseteq k$ is the valuation ring.
Moreover, in both cases we have $(T,\tT)\in \widetilde{(X,\tX)_{{\rm int},\tau}}$.
\end{prop}

\begin{remark}\label{rmk:Huberpair}
With the terminology from \cite[Definition 10.8]{HS2020}
the pair $(S,\tS)$ appearing in Proposition \ref{prop:tamely-acyclic} is a strictly henselian (resp. tamely henselian) local Huber pair.

The analogous result for local Huber pairs is \cite[Proposition 10.15]{HS2020}. Our context is  different in the following sense: on a discretely ringed adic space $\Spa(X,\tX)$, going ``local'' means looking at neighborhoods of a point $(x,v,\epsilon)$ for a \emph{fixed} valuation rign $\cO_v\subseteq k(x)$, while in the case of the site $(X,\tX)_t$, then one considers neighborhoods of a point $x\in X$ and their compactifications mapping to to $\tX$, which means that we need to consider \emph{all} the possible valuation rings $\cO_v\subseteq k(x)$ with a map $\Spec(\cO_v)\to \tX$. This was the same philosophy used in the construction of local modulus pairs of \cite[Corollary 3.5]{KelMiya}, therefore our proof will follow their arguments. For this reason, in order to keep the proof self-contained and more readable, we will write it in details.
\end{remark}
We need the following preliminary result:

\begin{lemma}\label{lem:limit}
Let $(Y,\tY)=\lim_{i\in I} (Y_i,\tY_i)$ in $\widetilde{(X,\tX)_{\tau}}$. 
    Let $(f,\tf)\colon (U,\tU)\to (Y,\tY)$ in $\widetilde{(X,\tX)_{\tau}}$ with $f$ an \'etale covering. Then there exists a cofiltered category $J$ and a system of maps
    $(f_{ij},\tf_{ij})\colon (U_{ij},\tU_{ij})\to (Y_i,\tY_i)$  indexed over $I\times J$ such that for all $(i,j)$, $(U_{ij},\tU_{ij})\in (X,\tX)_{{\rm int},\tau}$ and $f_{ij}$ is an \'etale covering, and $\lim_{(i,j)\in I\times J}(U_{ij},\tU_{ij})\to (Y,\tY)$ refines $(f,\tf)$.
Moreover, if $(f,\tf)$ is  a tame covering in the sense of Definition \ref{def;XSt}, then we find such system that $(f_{ij},\tf_{ij})$ are tame coverings for all $i,j$. 
    \begin{proof}
    By definition, $(U, \tU)$ is a limit of $(U_j,\tU_j)\in 
    (X,\tX)_{{\rm affine},t}$ with $j\in J$ a cofiltered set,
    Let $(U',\tU')=\lim(U_j,\tU_j^{\rm int})$. Then $(U',\tU')\to (Y,\tY)$ is a covering that refines $(U,\tU)$, so we can suppose that $(U_j,\tU_j)\in (X,\tX)_{{\rm int},t}$. 
    Fix $i\in I$. We need to show that 
    $(U,\tU)\to (Y,\tY)\to(X,\tX)$ factors through $(U_{j_i},\tU_{j_i})$,
    for some $j_{i}\in J$. 
    To this end, we may assume that $X$ and $\tX$ are affine, and Lemma \ref{lem:ift-fp-limits} gives such factorization. 
    In particular, we find a commutative diagram:
\[ \begin{tikzcd}
    U_{j_{i}}\ar[r,"u_{j_{i}}"]\ar[d,"f_{ij_{i}}"]&\tU_{j_{i}}\ar[d,"\tilde{f}_{ij_{i}}"]\\
        Y_{i}\ar[r,"g_{i}"]&\tY_{i},
    \end{tikzcd}
    \] 
    in which $f_{ij_{i}}$ is \'etale by Remark \ref{rmk:ift-limit-fp-etale} and $\tilde{f}_{ij_{i}}$ are ift, hence it gives a map $(f_{ij_{i}},\tf_{ij_{i}})$ in $(X,\tX)_t$. 
    Moreover, since all the indexing sets are cofiltered, for every $i\in I$,
    and every $j\in J$ as before, there exists $(U_{j'},\tU_{j'})$ that maps to both $(U_j,\tU_j)$ and $(U_{j_{i}},\tU_{j_{i}})$, hence we find a commutative diagram:
\[ 
    \begin{tikzcd}
    U_{j'}\ar[r,"u_{j'}"]\ar[d,"f_{ij'}"]&\tU_{j}\ar[d,"\tilde{f}_{ij'}"]\\
        Y_{i}\ar[r,"g_{i}"]&\tY_{i},
    \end{tikzcd}
    \]
    with $f_{ij'}$ \'etale (since $U_{j'}\to U_{j_{i}}$ and $f_{ij_{i}}$ are \'etale) and $\tf_{ij'}$ ift (since $\tU_{j'}\to \tU_{j_{i}}$ are ift by \cite[Proposition 2.2.5]{Temkin-insep-loc-uni}
     and $f_{ij_{i}}$ are \'etale), then up to replacing $J$ by a cofinal set we can suppose that for every $(U_j,\tU_j)$ there exist $i\in I$, 
    and a map $(U_j,\tU_j)\to (Y_{i},\tY_{i})$ as in the statement with $(U_j,\tU_j)\in (X,\tX)_{{\rm int},t}$.   
    
Now, assume that $(f,\tf)$ is a tame covering and prove that $(f_{ij},\tf_{ij})$ as constructed above
is a tame covering for a sufficiently large $i,j$.
 We proceed as the proof of  \cite[Theorem 4.6]{HS2020}.
Let $Z_i\subseteq \Spa(Y_i,\tY_i)$ be the set of triples $(y_i,w_i,\epsilon_{w_i})$ such that there is no $(x_i,v_i,\epsilon_{v_i})$ in $\Spa(U_{ij},\tU_{ij})$ tame over $(y_i,w_i,\epsilon_{w_i})$.
Since $(f,\tf)$ is tame, we have $\lim_i Z_i=\varnothing$.
By \cite[Cor.4.4]{Huebner2021}, $Z_i$ is closed, which implies that it is compact in the constructible topology by \cite[\href{https://stacks.math.columbia.edu/tag/08KQ}{Tag 0901}]{stacks-project} since $\Spa(Y_i,\tY_i)$ is spectral.
Since the inverse limit of nonempty compact spaces is nonempty, we must have $Z_{i}=\emptyset$ for a sufficiently large $i$, which completes the proof.
    \end{proof}
\end{lemma}

\noindent
{\it Proof of Proposition \ref{prop:tamely-acyclic}:}
    First of all, we observe that $T\to \tT$ is dense by Remark \ref{rmk-def;XSt}\ref{rmk-def;XSt2}, hence $T$ and $\tT$ have the same number of connected components.
    Let $\tT=\coprod_{i\in I}\tT_i$ and $T=\coprod_{i\in I} T_i$ be the decomposition into the connected components. We claim that $\cT=(T,\tT)$ is acyclic if and only if $\cT_i=(T_i,\tT_i)$ are acyclic for all $i$. 
    Indeed, assume that $\cT$ is acyclic and take coverings $\cV_i\to \cU_i$ in $(X,\tX)_{\tau}$ and maps $\phi_i: \cT_i\to \cU_i$ in $\widetilde{(X,\tX)_{\tau}}$ for $i\in I$.
    Fixing $i\in I$, it gives rise to a covering 
    $\cV:=\sqcup_{j\neq i} \cU_j\sqcup \cV_i \to \cU=\sqcup_{j\in I} \cU_j$ and
    a map $\phi:\cT \to \cU$ in an obvious way. By the assumption, $\phi$ factors through $\cV$ and the image of the map $\cT_i\to \cT\to \cV$  lands in $\cV_i$ since $T_i$ and $\tT_i$ are connected. Thus, $\phi_i$ factors through $\cV_i$ showing that $\cT_i$ is acyclic. 
    On the other hand, assume that $\cT_i$ are acyclic for all $i\in I$.
    Let $\cV\to \cU$ be a covering in $(X,\tX)_{\tau}$ and $\phi:\cT\to \cU$ be a map in $\widetilde{(X,\tX)_{\tau}}$.
   For each $i\in I$, the map $\cT_i\to\cT\to \cU$ factors through a map $\psi_i:\cT_i\to \cV$ since $\cT_i$ are acyclic.
   Then,  $\psi=\sqcup_{i\in I} \psi_i:\cT\to \cV$ gives a lift of $\phi$ showing that $\cT$ is acyclic. For the rest of the proof, we assume that $(T,\tT)$ is connected.
\medbreak

        $\Rightarrow$: Assume $\cT$ is v-\'etale local (resp. tame local). 
        Recall that every v-\'etale covering is also a tame covering.
    Let $\cT=(T,\tT)=(\Spec(S),\Spec(\tS))$ and write 
    \[\cT=\lim_{\alpha\in A}\cT_\alpha\;\text{ with }\cT_\alpha=(\Spec(S_\alpha),\Spec(\tS_\alpha))\in (X,\tX)_{{\rm affine},t}\]
    so $S=\colim_\alpha S_\alpha$ and $\tS=\colim_\alpha \tS_\alpha$.
    Let $(\tS)^{\rm int}$ (resp. $(\tS_\alpha)^{\rm int}$) be the integral closure of $\tS$ in $S$ (resp. $\tS_\alpha$ in $S_\alpha$).
 We have $(\tS)^{\rm int}=\colim_\alpha (\tS_\alpha)^{\rm int}$ so we have
 \[(\Spec(S),\Spec((\tS)^{\rm int}))=\lim_{\alpha\in A}(\cT_\alpha)^{\rm int}\]
 with  
 $\cT_\alpha^{\rm int}=(\Spec(S_\alpha),\Spec((\tS_\alpha)^{\rm int}))\in (X,\tX)_{{\rm int},t}$.
    Since $(\cT_\alpha)^{\rm int}\to \cT_\alpha$ is a quasi-modification and $\cT$ is acyclic, the projection
    $\cT \to \cT_\alpha$ factors through $(\cT_\alpha)^{\rm int}$, i.e.,
    $\tS_\alpha\to \tS$ factors through $(\tS_\alpha)^{\rm int}$. This implies $\tS=(\tS)^{\rm int}$ and $\cT=\lim_{\alpha\in A}(\cT_\alpha)^{\rm int}$ so we may assume $\tS_\alpha=(\tS_\alpha)^{\rm int}$ for all $\alpha\in A$.
    \medbreak
    
    Noting that the restriction of the v-\'etale topology on $\tX$ is finer than the \'etale topology, $\tS$ must be strictly henselian.
    \medbreak
    Next, we show that $S$ is a local ring (see \cite[Proposition 3.3, ``$\Rightarrow 2$"]{KelMiya})
    Let $x_1$ and $x_2$ be closed points in $T$. 
    Take open covers $\{\Spec(S[1/g_{\nu, j}])\}_{j\in J_{\nu}}$  of $T - \{x_{\nu}\}$, $\nu=1,2$.
     As $S$ is a localization of $\tS$ we may assume 
     \eq{prop:tamely-acyclic-eq1}{g_{\nu, j}\in \tS\setminus \tS\cap S^\times,\quad \text{for all }\nu,j.}
     If $x_1\neq x_2$, then the ideal $(g_{1,i},g_{2,j})_{i\in J_1,j\in J_2}$ of $\tS$ maps to the unit ideal in $S$. 
     As $T$ is quasi-compact, there exist finite subsets $K_\nu \subset J_\nu$ such that 
     \[\bigsqcup_{i\in K_1} \Spec (S[1/g_{1, i}] )\sqcup \bigsqcup_{j\in K_2} \Spec (S[1/g_{2, j}])\to T\]
     is an open covering. In particular the finitely generated ideal
     $\mathfrak{I}:=(g_{1, i}, g_{2, j})_{i\in K_1, j\in K_2}$ of $\tS$ maps to the unit ideal in $S$.  
     There exists $\alpha$ such that   $g_{1,i},g_{2,j}$ come from $g_{\alpha,1 ,i},g_{\alpha,2,j}\in \tS_{\alpha}$ and 
  the ideal $\mathfrak{I}_\alpha:=(g_{\alpha,1,i},g_{\alpha,2,j})_{i\in K_1, j\in K_2}$ of $\tS_\alpha$ maps to the unit ideal in $S_\alpha$.
Hence, we get the v-\'etale covering: \[
    \bigsqcup_{i\in K_1}\cU_{\alpha,1,i}\sqcup \bigsqcup_{j\in K_2}\cU_{\alpha,2,j}
   \to \cT_\alpha, \] 
   where
        \[ \cU_{\alpha,\nu,i}=(\Spec(S_\alpha[1/g_{\alpha,\nu,i}]),\Spec(\tS_\alpha[\tfrac{\mathfrak{I}_\alpha}{g_{\alpha,\nu,i}}])),\quad \text{for } \nu=1,2 \text{ and } i\in K_\nu.\]
     Since $\cT$ is acyclic, the projection $\cT\to \cT_\alpha$ factors through 
     $\cU_{\alpha,\nu, i}$ for  $\nu=1$ or $2$ and some $i\in K_\nu$, which implies 
     that $g_{1,i}$ or $g_{2,i}$ is a unit in $S$, contradicting \eqref{prop:tamely-acyclic-eq1}. 
     Therefore $x_1 = x_2$ and $S$ is local.
\medbreak
\def\Sa{S_\alpha}
\def\tSa{\tS_\alpha}
\def\cTa{\cT_\alpha}
\def\aa{a_\alpha}
\def\ba{b_\alpha}
\def\xa{x_\alpha}
\def\fa{f_\alpha}

Let $\fp\subseteq \tS$ be the prime ideal such that $\fp S$ is the maximal ideal of $S$.
We show that $\tS/\fp$ is a valuation ring (see \cite[Proposition 3.3, ``$\Rightarrow 3$"]{KelMiya}). Let $a,b\in \tS\setminus \fp$. Since $\fp S$ is maximal, $a$ and $b$ are invertible in $S$. There exists $\alpha\in A$ such that $a,b$ come from 
$\aa,\ba\in \tS_\alpha$.
Then,
 we have the v-\'etale covering $\cU_{\alpha,a}\sqcup \cU_{\alpha,b} \to \cTa$ with 
\[\cU_{\alpha,a}=(\Spec(\Sa),\Spec(\tSa\Bigl[\frac{\ba}{\aa}\Bigr])),\quad \cU_{\alpha,b}=(\Spec(\Sa),\Spec(\tSa\Bigl[\frac{\aa}{\ba}\Bigr])).\]
Since $\cT$ is acyclic and connected, the projection $\cT\to \cTa$ factors through 
$\cU_{\alpha,a}$ or $\cU_{\alpha,b}$, so $\tSa\to \tS$ factors through either 
$\tSa\Bigr[\ba/\aa\Bigr]$ or $\tSa\Bigl[\aa/\ba\Bigr]$.
Hence, either $b=ha$ or $a=hb$ for the image $h\in \tS$ of $\ba/\aa$ or $\aa/\ba$, which implies that $\tS/\fp$ is a valuation ring. 
\medbreak

Next, we show $\tS\simeq  S\times_{k(\fp)} \tS/\fp$, where $k(\fp)=S/\fp$ is the fraction field of $\tS/\fp$ (see \cite[Proposition 3.3, ``$\Rightarrow 4$"]{KelMiya}). 
Since $S$ is local and $\fp S$ is its maximal ideal, we have $S=\tS_{\fp}$. 
Then, it is enough to check that the map $\fp\to \fp S$ is an isomorphism. The map $\tS\to S$ is injective by Remark \ref{rmk;int}\ref{rmk;int2}, therefore $\fp\to \fp S$ is injective.
 Let $x/h$ in $\fp S$, with $x\in \fp$ and $h\in \tS\setminus \fp$. 
 There exists $\alpha\in A$ such that $x,h$ come from 
$\xa, h_{\alpha}\in \tS_\alpha$ and  $h_{\alpha}$ is a unit in $\Sa$.
  Then,
 we have the v-\'etale covering $\cU_{\alpha,f}\sqcup \cU_{\alpha,x}\to \cT$ with
 \[
 \cU_{\alpha,f}=(\Spec(\Sa),\Spec(\tSa\Bigl[\frac{\xa}{h_{\alpha}}\Bigr])),\quad
 \cU_{\alpha,x}= (\Spec(\Sa[1/\xa]),\Spec(\tSa\Bigl[\frac{h_{\alpha}}{\xa}\Bigr])).
\]
As before, this implies that $\tSa\to \tS$ factors through either 
$\tS[\xa/h_{\alpha}]$ or $\tS[h_{\alpha}/\xa]$. In the latter case, there is $y\in \tS$ such that $yx=h$, 
but this is impossible since $h\not\in \fp$. In the former case, there is $y\in \tS$ such that $y h=x$, hence  $x/h=y\in \fp$, which implies that $\fp\to \fp S$ is surjective.
\medbreak

Next, we show that $S$ is henselian (see \cite[Proposition 3.4]{KelMiya}). We follow the argument in \cite[Lemma 10.7]{HS2020}. 
Let $S\to B$ be finite with $\Spec(B)$ connected. Then, $B$ is semilocal, and to show that $S$ is henselian, it is enough to show that $B/\mathfrak{\fp}B$ is local. Since it is a finite algebra over the field $S/\mathfrak{\fp}S$, it is enough to check that $\Spec(B/\mathfrak{\fp}B)$ is connected. Let $\tB$ be the integral closure of $\tS$ in $B$.
We claim 
\begin{enumerate}[label=(\arabic*)]
    \item\label{claim0;prop:tamely-acyclic0} $\tB$ is local;
    \item\label{claim0;prop:tamely-acyclic1} $\fp B\subset \tB$;
    \item\label{claim0;prop:tamely-acyclic2} $\tB/\fp B$ is the integral closure of $\tS/\fp\tS$ in $B/\fp B$.
\end{enumerate}
Then $\tB/\fp B$ is local by \ref{claim0;prop:tamely-acyclic0}, \ref{claim0;prop:tamely-acyclic1} 
and hence $\Spec(\tB/\fp B)$ is connected. 
We conclude by \cite[\href{https://stacks.math.columbia.edu/tag/03GO}{Tag 03GO}]{stacks-project} that $\Spec(B/\fp B)$ is connected. Indeed, if $B/\fp B=R_1\times R_2$, then $\tB/\fp B = \tR_1\times \tR_2$, with $\tR_i$ the integral closure of $\tS/\fp$ in $R_i$.

Thus it remains to show the three claims above.
For \ref{claim0;prop:tamely-acyclic0} we note that $\tS\to\tB$ is integral and hence $\tB$ is a filtered union of its subrings 
$\tB_i$ finite over $\tS$. Since $\tB_i\to \tB\to B$ are injective  the maps 
$\Spec(B)\to \Spec(\tB)\to \Spec(\tB_i)$ have dense images.
Therefore, as $\Spec(B)$ is connected so are $\Spec(\tB)$ and $\Spec(\tB_i)$. Since $\tS$ is henselian, 
$\tB_i$ is local henselian for all $i$. Since the maps $\tB_i\to \tB_j$ are finite, they are local maps of henselian local rings.
Hence, $\tB$ is local.
\ref{claim0;prop:tamely-acyclic1}. Take $y\in B$ and $m\in \fp$.
Since $B$ is integral over $S$, we can write $y^n=\sum_{i=0}^{n-1} a_iy^i$, for $a_i\in S$, so that 
\[
(my)^n=\sum_{i=0}^{n-1} a_im^{n-i} (my)^i.
\]
As $a_im^{n-i}\in \fp S = \fp\subseteq \tS$, for $i\in [0,n-1]$, we get $my\in \tB$.
\ref{claim0;prop:tamely-acyclic2}. If $\ol{x}\in B/\fp B$ and $\ol{g}\in \tS/\fp[T]$ is monic such that $\ol{g}(\ol{x})=0$, 
then let $x\in B$ and $g\in \tS[T]$ be lifts of $\ol{x}$ and $\ol{g}$. 
Then $g(x)\in \fp B  \subset \tB$, in particular $f(x)$ is integral over $\tS$, so there exists $h(T)\in \tS[T]$ such that $h(g(x))=0$, therefore $x$ is integral over $\tS$ hence it lies in $\tB$, hence $\ol{x}$ lies in $\tB/\fp \tB$.
Thus we have shown that $S$ is henselian.

\medbreak

To conclude the proof of the implication $\Rightarrow$, it is enough to further check that $(k,v)$ is strictly henselian in case 
of the v-\'etale-topology and tamely closed in case of the tame topology, where $k=S/\fp S$ and $v$ is the valuation associated to $\tS/\fp$. The former case holds since $\tS$ is strictly henselian. To show the latter case, take a finite extension $k'/k$ and a valuation $v'$ on $k'$ over $v$ such that $v'/v$ is tame.
We want to prove $k=k'$. 
Since it is separable, there exists $\ol{\omega}\in \cO_{v'}$ such that $k'=k[\ol{\omega}]$ and $\cO_{v'}$ is the integral closure of $\cO_v[\ol{\omega}]$. 
Let $\ol{p}\in \cO_v[T]$ be the monic minimal polynomial of $\ol{\omega}$ over $k$.
Since $S$ is henselian and $\tS=S\times_k \cO_v$, there is $p\in \tS[T]$ that maps to $\ol{p}$ in $k[T]$ giving a finite \'etale extension $S\hookrightarrow S'=S[T]/(p)$, with $S'$ henselian local with residue field $k'=S'/\fp S'$.
Let $\tS':=S'\times_{k'}\cO_{v'}$. 

\begin{claim}\label{claim1;prop:tamely-acyclic}
$\tS'$ is the integral closure of $\tS[T]/(p)$ in $S'$.
\end{claim}

Indeed, let $R=\tS[T]/(p)$. Note that $\tS'$ is integrally closed in $S'$ by definition. 
Furthermore, the image of $T$ in $S'$ lies in $\tS'$ since its image $\ol{\omega}$ in $k'=S'/\fp S'$ lies in $\cO_{v'}$.
Thus the natural map $R\to S'$ factors via $\tS'$ and it remains to show that $\tS'$ is integral over $R$.
To this end we observe that $\fp S'$ lies in $R$ and in $\tS'$ and that the induced map on the quotients 
$R/\fp S'\to \tS'/\fp S'$ is the integral extension $\sO_v[\ol{\omega}]\inj \sO_{v'}$. 
Hence $\tS'$ is integral over $R$.
  
\medbreak
Recall $\cT=(\Spec(S),\Spec(\tS))=\lim_\alpha \cTa$ with
$\cTa=(\Spec(\Sa),\Spec(\tSa))$.
We show $\cT':=(\Spec(S'),\Spec(\tS'))\in \widetilde{(X,\tX)_\tau}$.
Since $\tS=\colim \tS_\alpha$, there exists $\alpha_0\in A$ and $p_{\alpha_0}\in \tS_{\alpha_0}[T]$ mapping to $p$. Letting $p_\alpha$ be the image of $p_{\alpha_0}$ in $\tS_\alpha[T]$, we have $S'=\colim_{\alpha\geq \alpha_0} \Sa'$ with $\Sa'=S_\alpha[T]/(p_\alpha)$.
By construction, $\Sa'=S'_{\alpha_0}\otimes_{S_{\alpha_0}} \Sa$.
By \cite[\href{https://stacks.math.columbia.edu/tag/07RP}{Tag 01SR}]{stacks-project}, $S_\alpha\to \Sa'$ is \'etale for $\alpha \gg \alpha_0$, 
so $\Spec(\tSa')$ is \'etale over $X$ since $\Spec(S_\alpha)$ is \'etale over $X$. 
Let $\tS_\alpha'$ be the integral closure of $\tS_\alpha[T]/(p_\alpha)$ in $\Sa'$. 
Since $\Spec(\tS_\alpha)$ is ift over $\tX$, $\Spec(\tS_\alpha')$ is ift over $\tX$. 
By Claim \ref{claim1;prop:tamely-acyclic}, we have
$\colim_{\alpha} \tS'_\alpha=\tS'$ noting that taking integral closures commutes with filtered colimits.
By construction, $\cTa':=(\Spec(\Sa'),\Spec(\tSa')) \in (X,\tX)_\tau$ and we have
$\cT'=\lim_\alpha \cTa'$.
By Lemma \ref{lem:limit}, $\cTa'\to \cTa$ is a tame covering for $\alpha\gg \alpha_0$. Since $\cT$ is tame acyclic, this implies that for $\alpha\gg \alpha_0$, the map $\Sa\to S$ factors through $\Sa'$, which implies $k=k'$ as desired.
\medbreak

$\Leftarrow$ Take $\cT=(\Spec(S),\Spec(\tS))$ with $\tS=S\times_k \cO_v$ as in Proposition \ref{prop:tamely-acyclic}. 
        We want to show that  for any covering $h: \cV=(V,\tV)\to\cU=(U,\tU)$ in $(X,\tX)_{\gamma}$ with $\gamma=\vet$ or $\gamma=t$,
    \eq{eq0;lifting}{\Hom_{\widetilde{(X,\tX)_\tau}}(\cT,\cV) \to \Hom_{\widetilde{(X,\tX)_\tau}}(\cT,\cU)}
    is surjective.
    Clearly, it suffices to consider the generator coverings so we may assume that $h$ is either a quasi-modification or strict \'etale covering in case $\gamma=\vet$ and 
    a tame covering in case $\gamma=t$.
  Write $T=\Spec(S)$, $\tT=\Spec(\tS)$ and take $(f,\tf):(T,\tT)\to (U,\tU)$.
    
    If $h$ is a strict \'etale covering so that $\tV\to \tU$ is an \'etale covering and 
    $V=\tV\times_{\tU} U$, $\tf$ admits a lift $\tg:\tT\to \tV$ since $\tS$ is strictly henselian. Moreover,  the composite $T\to \tT\rmapo{\tg} \tV$ and $f:T\to U$ induce
    $g:T \to V=  \tV\times_{\tU} U$ so that $(g,\tg)$ gives a lift of $(f,\tf)$.
    
Assume that $h$ is a quasi-modification.      
 Then, $f:T \to U$ lifts to $g: T\to V$ because $V \to U$ is an isomorphism. By the valuative criterion for 
 universal closedness, the composite 
$ \Spec(k)\hookrightarrow T\rmapo{g} V \to \tV$ extends to a morphism
 $q:\Spec(\cO_v)\to \tV$. These morphisms factor through some open affine of $\tV$, so $q$ and $T\rmapo{g} V\to \tV$ glue to give a morphism 
$\tg:\tT=\Spec(\tS) \to \tV$ since $\tT=\Spec(\cO_v \times_k S) = 
\Spec(\cO_v) \sqcup_{\Spec(k)} T$ is the categorical pushout in the category of affine schemes. Thus, we get a lifting $(g,\tg)$ of $(f,\tf)$. 

Finally, assume that $h$ is a tame covering and $(k,v)$ is tamely closed.      
Let $x\in U$ be the image of $\Spec(k)\hookrightarrow T \rmapo{f} U$ and let $v_x$ be the restriction of $v$ to the residue field $\kappa(x)$. Note, as $f$ is pro-\'etale, the field extension $k/\kappa(x)$ is separable algebraic.  
By  assumption, there is a point $y\in V$ and a valuation $v_y$ on its residue field $\kappa(y)$ 
extending $v_x$ such that the extension $(\kappa(y),v_y)/(\kappa(x),v_x)$ is tame. Since $(k,v)$ is tamely closed, 
there exists a map of valued fields $(\kappa(y),v_y)\to (k,v)$ which factors $(\kappa(x),v_x)\to (k,v)$.
    Since $V\to U$ is an \'etale covering and $S$ is henselian, this implies that 
    $f:T=\Spec(S)\to U$ admits a lift $g:T\to V$.
    By the same argument as in the case of quasi-modifications, $g$ extends to 
    $\tg:\tT\to \tV$ so that we get a lifting $(g,\tg)$ of $(f,\tf)$. This completes the proof.
\begin{rmk}\label{rmk;Deligne}
By Remark \ref{rmk-def;XSt}\ref{rmk-def;XSt1} and Deligne's completeness theorem, \cite[Proposition VI.9.0]{SGA42} or 
\cite[Theorem 7.44, Corollary 7.17]{Joh77}, the fiber functors $\phi_{\cT}$ from \eqref{phiT} for $\cT$ satisfying the condition of Proposition \ref{prop:tamely-acyclic} form a conservative family, i.e., a morphism $f$ in $\Sh((X,\tX)_\gamma)$ is an isomorphism if and only if $\phi_{\cT}(f)$ is an isomorphism of sets for all $\gamma$-acyclic $\cT$.
Equivalently, a morphism $\cV\to \cU$ in $(X,\tX)_\tau$ is a $\gamma$-covering if and only if 
\[\Hom_{\widetilde{(X,\tX)_\tau}}(\cT,\cV) \to \Hom_{\widetilde{(X,\tX)_\tau}}(\cT,\cU)\]
is surjective for all such $\cT$, \cite[Proposition IV.6.5(a)]{SGA41}.
\end{rmk}
\begin{rmk}\label{rmk:pair}
If the valuation ring $\cO_v$ has finite height, then there exists $a\in \cO_v$ such that $k=\cO_v[1/a]$ (see \cite[Proposition 0.6.7.2 and Exercise 0.6.4]{FK}). In this case, it is closer to the modulus pairs of \cite{KelMiya}.
\end{rmk}

    \begin{remark}\label{rmk:refinement-covering}
    Let $(\Spec(A),\Spec(A\times_k\cO_v))$ be $\vet$ acyclic. Let $(\Spec(B),\Spec(\tB))\to (\Spec(A),\Spec(A\times_{k}\cO_v))$ be a tame covering in the sense of Definition \ref{def;XSt}\ref{rmk-def;XSt3}.  Since $A$ is henselian local, we can refine it so that $B\to A$ is finite \'etale asssociated to a finite separable extension of $k\hookrightarrow k'$. By tameness, there exists a valuation $w$ on $k'$ extending $v$ such that $\tB\to B\to k'$ factors through $\cO_{w}$. This implies that the map $\tB\to B$ factors through $B\times_{k'}\cO_w$, therefore the covering $(\Spec(B),\Spec(\tB))$ is refined by $(\Spec(B),\Spec(B\times_k\cO_w))$. Moreover, by Lemma \ref{lem:tame-Gal-closure}\ref{lem:tame-Gal-clousre1} we can further refine it so that we have that $k'/k$ is Galois.
\end{remark}

\section{\v Cech comparison}\label{sec:6}

In this section, we prove a comparison theorem of tame cohomology with \v Cech cohomology (see Proposition \ref{thm:cech-comp}).

\medbreak

We fix a quasi-compact open immersion $X\to \tX$ of qcqs schemes.

\begin{defn}\label{def:x-local}
    Let $\cU = (U,\tU)\in (X,\tX)_t$, let $x\in U$ and let $U_x^h=\Spec(\cO_{U,x}^h)$ be the henselization at $x$. An $x$-local object over $\cU$ is $\cT:=(\Spec(B),\Spec(\tB))\in \widetilde{(X,\tX)_t}$ with a map $\cT\to \cU$ such that 
    \begin{enumerate}[label=(\arabic*)]
        \item\label{def:x-local1} $B$ is henselian local with residue field $k$, $\Spec(B)\to U$ factors through $U_x^h$ and the map $\cO_{U,x}^h\to B$ is local and ind-\'etale;
        \item\label{def:x-local2} There is a valuation $v$ on $k$ such that $(k,v)$ is tamely closed and $\Spec k\to U$ restricts to
          $\Spec \cO_v\to \tU$;
        \item\label{def:x-local3} $\tB=B\times_k \cO_v$.
    \end{enumerate}
\end{defn}

\begin{example}\label{ex:localization}
    Let $\cU = (U,\tU)\in (X,\tX)_t$, let $x\in U$ and $k(x)$ be its residue field. Let $(x,v,\epsilon_v)\in \Spa(U,\tU)$ and choose an extension $\ol{v}$ to a separable closure $\ol{k(x)}$ of $k(x)$ and $k(x)\hookrightarrow k(x)_v^t$ be the tame closure of $k(x)$ with respect to the valuation $v$. Let $\cO_{U,x}^h\to \cO_{U,x}^t$ be the ind-\'etale map corresponding to the field extension $k(x)\hookrightarrow k(x)^t_v$ and let $\cO_v^t\subseteq k(x)^t_v$ be the valuation ring of the restriction of $\ol{v}$ to $k(x)^t_v$.
Then, $\cU_{(x,v)} := (\Spec(\cO_{U,x}^t), \Spec(\cO_{U,x}^t\times_{k(x)_v^t} \cO_v^t))$ is an $x$-local object over $\cU$. 
Moreover, $\cU_{(x,v)}\to \cU$ is a cofiltered limit of maps $\cU_i\to \cU$ in $(X,\tX)_t$ which is tame over $(x,v,\epsilon_v)$.

\end{example}

The following Lemma  is a tame version of \cite[Theorem 3.4]{artin}. The statement and the proof are  close
to \cite[Corollary 11.7]{HS2020}, which is a version for discretely ringed  adic spaces.
\begin{lemma}\label{lem:product-x-local}
    Let $\cY =(Y,\tY)\in (X,\tX)_t$ such that $\tY$ satisfies the property that every finite set of points is contained in an affine open. For $1\leq i\leq n$, let $x_i\in Y$ and let $\cP_i$ be $x_i$-local objects over $\tY$. Then $\cT = \cP_1\times_{\cY}\ldots \times_{\cY} \cP_n\in \widetilde{(X,\tX)_\tau}$ is affine and is a disjoint union of $x$-local objects, where either $x=x_i$ for some $i$ or $x$ is a generization of 
    all $x_i$, i.e., the $x_i$ lie in the closure of $x$.
    \begin{proof}
        Let $\cP_i = (\Spec(A_i),\Spec(\tA_i))$ with $\tA_i=A_i\times_{k_i}V_i$, where $k_i$ is the residue field of $A_i$ and $V_i$  
        is a valuation field on $k_i$. Let $\Spec(V_i)\to \tY$ be the induced map on the valuation rings, and let $\ty_i$ be the respective images of the closed points. Let $\Spec(\tA)\subseteq \tY$ be an affine open containing $x_i$ and $\ty_i$ for all $i$. 
        As $\ty_i$ is a specialization of $x_i$, we have natural maps
        \[\tA\to \cO_{\tY, \ty_i}\to \cO_{\tY,x_i}\to A_i\to k_i,\]
        and  by the definition of $\ty_i$ its composition factors via $V_i\inj k_i$.
        Therefore all the maps $\Spec(V_i)\to \tY$ above factor through $\Spec(\tA)$, therefore 
        \[\Spec(\tA_1)\times_{\tY}\ldots \times_{\tY} \Spec(\tA_n) = \Spec(\tA_1\otimes_{\tA} \ldots\otimes_{\tA} \tA_n).\] 
        As $\Spec(\tA)\cap Y$ is quasi-affine we find an open  $\Spec(A)\subseteq \Spec(\tA)\cap Y$ which contains all the $x_i$.
        Then all maps $\Spec(A_i)\to Y$ factor through $\Spec(A)$ hence 
        \[
        \cP_1\times_{\cY}\ldots \times_{\cY} \cP_n = \cP_1\times_{(\Spec(A),\Spec(\tA))}\ldots \times_{(\Spec(A),\Spec(\tA))} \cP_n.
        \] 
        Therefore we are reduced to the case $\cY = (\Spec(A),\Spec(\tA))$, and now the general case follows from the case where $n=2$.
        
     Thus it suffices to consider the following situation. Let $\fp$ and $\fq$ be prime ideals of $A$.
     Let $\cP = (\Spec(B), \Spec(\tB))$ and $\cQ = (\Spec(C), \Spec(\tC))$ be $\fp$-local and $\fq$-local objects, respectively,
     with $\tB=B\times_{k_B} V_B$ and $\tC=C\times_{k_C} V_C$ as in Definition \ref{def:x-local}. 
     Denote by $\fm_B\subset B$ and $\fm_C\subset C$ the maximal ideals.
    By \cite[Theorem 6.3 and Theorem 6.4]{HS2020}, $B\otimes_A C$ is a product of henselian local $A$-algebras and the following holds:
let $D$ be a factor of $B\otimes_A C$ and denote by $\fm$ its maximal ideal and  $L=D/\fm$ its residue field.
\begin{enumerate}[label=(\arabic*)]
    \item\label{enum:product-x-local1} If the maps $B\to D$ and $C\to D$ are not local,  then $L$ is separably closed;
    \item\label{enum:product-x-local2} if $\varphi: B\to D$ is local, then 
    the residue field extension $k_B\to L$ is a separable algebraic extension.
\end{enumerate}
In case  \ref{enum:product-x-local1} the natural map $\tB\otimes_{\tA}\tC\to D$ is surjective. 
Indeed, we have  $\fm_B\subset \tB$ and $\fm_C\subset \tC$ 
and as $B\to D$ and $C\to D$ are not local we have $\fm_B\cdot \fm_C\cdot D=D$.
Thus $D$ is {integral} over $\tB\otimes_{\tA} \tC$ and is strictly henselian {by \cite[Th.3.4(ii)]{artin}.} 
Hence $(D,D)$ is an $\fr$-local object, for $\fr=\fm_D\cap A\subset\fp,\fq$, where we consider the trivial valuation on  $k_D$,
and $(\Spec (D), \Spec (D))$ is a component of $\cP\times_{\cY} \cQ$.

We consider case \ref{enum:product-x-local2}. 
Denote by $\fm$ its maximal ideal of $D$ and by $L=D/\fm$ its residue field.
Let $w$ be the unique valuation on $L$ that extends the valuation $v$ on $k_B$.
Thus $(L,w)$ is a henselian valuation field and its valuation ring $\cO_w$ is equal 
to the integral closure of $V_B$ in $L$, see, e.g., \cite[VI, \S 8, Proposition 6]{BourbakiCA}. 
As $(k_B,v)$ is tamely closed so  is $(L,w)$. 
Let $\tD$ be the integral closure of $\tB\otimes_{\tA} \tC$ in $D$, 
and let $W$ be the image of the map $\tD\to D\to L$, so that we have the following commutative diagram
\begin{equation}\label{eq:cd-integral-closure}
    \begin{tikzcd}
    \tB\ar[r]\ar[d]&\tD\ar[r,hook]\ar[d] &D\ar[d]\\
    V_B\ar[r]&W\ar[r,hook]&L.
\end{tikzcd}
\end{equation}
We claim
\begin{equation}\label{eq:product-x-local1}
\fm \subset \tD.
\end{equation}
Assuming \eqref{eq:product-x-local1} we directly get $\tD=D\times_L W$.
Moreover $W$ contains $\cO_w$. Indeed,
if $a\in L$ is integral over $V_B$  and $\tilde{a}\in D$ is a lift of $a$, 
then we find a monic polynomial $f\in \tB[X]$ with $f(\tilde{a})\in \fm$ which by the claim 
is integral over $\tB\otimes_{\tA} \tC$, hence so is $\tilde{a}$, hence $a\in W$.
Thus $W$ is a henselian valuation ring by \cite[Lemma 11.4]{HS2020} 
and its valuation $w'$ is a generization of $w$, therefore $(L,w')$ is tamely closed. 
Hence $(\Spec D, \Spec \tD)$ is a $x_1$-local object over $\cY$.  
It remains to prove the claim \eqref{eq:product-x-local1}.

Let $m\in\fm$. As $\varphi: B\to D$ is integral, we find 
a monic polynomial $f(X)=X^n+a_1X^{n-1} +\ldots a_n$ in $B[X]$ such that $f^\phi(m)=0$, 
where $f^{\phi}=X^n+\phi(a_1)X^{n-1} +\ldots \phi(a_n)\in D[X]$.
Denote by $\bar{f}\in k_B[X]$ the reduction of $f$ modulo the maximal ideal $\fm_B$.
As $0=f^{\varphi}(m)\equiv \varphi(a_n)$ mod $\fm$,
we have $a_n\in \varphi^{-1}(\fm)=\fm_B$. Thus  
\[\bar{f}= X^e \cdot \bar{g},\]
for some $e\ge 1$ and $g\in k_B[X]$ monic with $g(0)\neq 0$. 
As $B$ is henselian, there exist monic polynomials $h$, $g\in B[X]$ with 
\[f= hg,  \quad h\equiv X^e \text{ mod }\fm_B, \quad g\equiv \bar{g} \text{ mod } \fm_B.\]
It follows that the constant term of $g^\varphi$ is a unit in $D$ and hence so is $g^\varphi(m)$. 
Thus $h^\varphi(m)=0$ in $D$. As $h\in X^e+\fm_B[X]\subset \tB[X]$ we find that $m$ is integral over $\tB$,
hence $m\in \tD$. This yields  claim \eqref{eq:product-x-local1} and completes the proof 
of the lemma.
\end{proof}
\end{lemma}
\def\cZ{\mathcal{Z}}

The following theorem is analogous to \cite[Theorem 4.1]{artin} while its proof is closer to the arguments given in the proof of 
\cite[Proposition 7.14 and Theorem 7.16]{HS2020}.
\begin{lemma}\label{thm:main-cech}
Let $\cY =(Y,\tY)\in (X,\tX)_t$ such that $\tY$ satisfies the property that every finite set of points is contained in an affine open.
Let $\cU=(U,\tU)\xrightarrow{(\phi,\tilde{\phi})} \cY$ be a tame covering.
Then, for a tame covering $\cV\to \cU^{\times_{\cY} n}$, there is a tame covering $\cU'\to \cU$ such that 
${\cU'}^{\times_{\cY}n}\to \cU^{\times_{\cY}n}$ factors through $\cV$. 
\end{lemma}
\begin{proof}
Since $\tU$ is quasi-compact and the open immersion $U\inj \tU$ is quasi-compact we find finitely many affine pairs 
$\cU_j=(U_j, \tU_j)\in (X,\tX)_{{\rm affine}, t}$ such that $U=\cup_j U_j$ and $\tU=\cup_j \tU_j$ are open coverings.
Thus the disjoint union $\sqcup_j\cU_j$ is an affine pair which tamely covers $\cU$. 
By Lemma \ref{lem:existence-acyclic} applied to this disjoint union, 
we find  a morphism $\cW=(W,\tW)\to \cU$ in $\widetilde{(X,\tX)_t}$,  which is a limit of tame coverings $\cW_i=(W_i, \tW_i)\to \cU$ in $(X,\tX)_{{\rm affine},t}$ with $\cW_i\in (X,\tX)_{{\rm int},t}$,  
such that $\cW$ is tamely acyclic. Note that every connected component $\cP$ of $\cW$ is an $x$-local object over $\cY$, 
for some points $x\in Y$. Indeed, by Proposition \ref{prop:tamely-acyclic}, $\cP$ satisfies the conditions \ref{def:x-local2} and \ref{def:x-local3} of Definition \ref{def:x-local}, and  
it suffices to check that $\cP$ satisfies \ref{def:x-local1}. 
Denote by $x$ the image in $Y$ of the closed point of $P=\Spec S$.
Then we get a natural local morphism $\cO_{Y,x}^h\to S$. It is ind-\'etale as $S$ is a component of 
$\colim_i \cO(W_i)$ and $W_i\to U$  and $U\to Y$ are \'etale.

Thus, $\cW^{n}$ (product over $\cY$) is a  disjoint union of  
$\cP_1\times_{\cY} ... \times_{\cY} \cP_n$, where $\cP_j$ are 
$x_j$-local objects for some $x_j\in Y$. Hence, by Lemma \ref{lem:product-x-local} and 
Proposition \ref{prop:tamely-acyclic}, $\cW^{n}$ is tamely acyclic.
Thus, for any tame covering $\cV\to \cU^n$ the map $\cW^{n}\to \cU^n$ factors via $\cV$. 
Noting $\cW_{i}^{n}\in (X,\tX)_{{\rm int},t}$, Lemma \ref{lem:ift-fp-limits} and Remark \ref{rmk:ift-limit-fp-etale} imply that 
there is $i_0$ and a map $\cW_{i_0}^{n}\to \cV$ in $(X,\tX)_{{\rm affine},t}$ that factors $\cW_{i_0}^{n}\to \cU^n$, hence by choosing
$\cU'= \cW_{i_0}$ we conclude the proof.
\end{proof}

We are now ready to prove the comparison with \v Cech cohomology, analogously to \cite[Theorem 7.16]{HS2020} and \cite[Corollary 4.2]{artin}. Let us fix the notations: for $\cY\in (X,\tX)_t$, we consider the \v Cech complex\[
\check C^\bullet(\cY_t,-)\colon \PSh((X,\tX)_t,\Ab)\to \Cpx(\Ab)\quad F\mapsto \colim_{\cU\to \cY} F(\cU^{\times_\cY\bullet}),
\] 
where the colimit runs over all the tame coverings $\cU\to \cY$. This is an exact functor, and its cohomology is the \v Cech cohomology $\check H^q(\cY_t,-)$
\begin{thm}\label{thm:cech-comp}
Let $F$ and $G$ be presheaves of abelian groups on $(X,\tX)_t$ such that $a_tF\cong a_tG$, and let $\cY =(Y,\tY)\in (X,\tX)_t$ such that $\tY$ satisfies the property that every finite set of points is contained in an affine open. Then there is an isomorphism of complexes\[
\check C^\bullet(\cY_t,F)\cong \check C^\bullet(\cY_t,G)
\] 
In particular, for $F$ a sheaf of abelian groups on $(X,\tX)_t$ and $\cY$ as above, the map comparing 
tame \v{C}ech cohomology with the cohomology of the tame site\[
\check H^q(\cY_t,F)\to H^q_t(\cY,F)
\]
is an isomorphism. 
\begin{proof}
Let $\iota \colon \Sh((X,\tX)_t,\Ab)\hookrightarrow \PSh((X,\tX)_t,\Ab)$  be the inclusion. It is enough to show that the map $\check C^\bullet(\cY_t,F)\to \check C^\bullet(\cY_t,\iota a_tF)$ is an isomorphism. Let $K$ and $K'$ be the kernel and cokernel respectively of $F\to \iota a_tF$: we have exact sequences of presheaves\[
0\to K\to F\to F/K\to 0\quad 0\to F/K\to \iota a_tF\to K'\to 0.
\]
Thus, it is enough to show that if $a_tF=0$, then $\check C^q_t(\cY,F)=0$ for all $q$. Let $\alpha \in F(\sU^{\times_{\sY} q})$ for a tame covering $\cU \to \cY$. Then there is a tame covering $\sV\to \sU^{\times_{\sY}q}$ such that  $\alpha\mapsto 0$ in $F(\sV)$. 
By Lemma \ref{thm:main-cech}, there is a tame covering
$\sU'\to \sU$ such that ${\sU'}^{\times_{\sY} q}\to \sU^{\times_{\sY} q}$ factors through $\cV$.
Hence, $\alpha\mapsto 0$ in 
$F({\sU'}^{\times_{\sY} q})$, which implies that it is zero in $\check{C}^q(\cY,F)$.
The second part follows from the first: indeed let $\cH^q(F)$ be the presheaf $\cU\to H^q(\cU_t,F)$ on $(X,\tX)_t$: since $a_t\cH^q(F)=0$ for $q>0$, by the previous part we have that $\check H^p(\cY_t,\cH^q(F)) = 0$ for $q>0$, so the \v Cech-to-cohomology spectral sequence
\[ E_2^{p,q}=\check H^p(\cY_t,\cH^q(F)) \Rightarrow H^{p+q}(\cY_t,F)\]
degenerates.

\end{proof}
\end{thm}

\section{Computation of tame cohomology}\label{sec:7}
We fix  a quasi-compact open immersion $X\to \tX$ of qcqs schemes.

 \def\Al{A_\lambda}
    \def\tAl{\tA_\lambda}
    \def\Bl{B_\lambda}
    \def\tBl{\tB_\lambda}
\begin{prop}\label{prop:tame-local-vanishing}
Let $p$ be a prime and let $F$ be a sheaf of  $\Z_{(p)}$-modules on $(X,\tX)_t$. Let $(U,\tU)\in \widetilde{ (X,\tX)}_\tau$ connected and v-\'etale acyclic.
By Proposition \ref{prop:tamely-acyclic}, $\cU=(U,\tU)=(\Spec A, \Spec \tA)$ where $A$ is a henselian local ring with residue field $K$ and $\tA=A\times_K \cO_v$ for a strictly henselian valuation $v$ .
Write $\cU$ as the limit of a cofiltered system 
$\{\cU_\lambda\}_{\lambda\in \Lambda}$ in $(X,\tX)_\tau$ with 
$\cU_\lambda=(\Spec \Al, \Spec \tAl)$.
If the residue characteristic of $\cO_v$ is $p$, then 
\[\varinjlim_{\lambda\in \Lambda} H^i((\cU_\lambda)_t, F)=0,\quad \text{for } i\ge 1.\]
\end{prop}
\begin{proof}
By Theorem \ref{thm:cech-comp}, we have that
\[
\varinjlim_{\lambda\in \Lambda} H^i((\cU_\lambda)_t, F) = \varinjlim_{\lambda\in \Lambda}\varinjlim_{\cV_\lambda\to \cU_{\lambda}}\pi_{-i}F(\cV_\lambda^{\times_{\cU_{\lambda}}\bullet}),
\]
where the colimit is indexed over tame covers of $\cU_{\lambda}$. 
By Lemma \ref{lem:limit}, this is equal to
\begin{equation}\label{eq1;prop:tame-local-vanishing}
\varinjlim_{\cV\to \cU} \pi_{-i}F(\cV^{\times_{\cU}\bullet}),
\end{equation}
where the colimit is indexed over tame covers of $\cU$ in the sense of Definition \ref{def;XSt}
and $F$ is left Kan extended to $\widetilde{ (X,\tX)}_\tau$.
 By Remark \ref{rmk:refinement-covering}, we can further suppose that $\cV$ are of the form $(V,\tV)=(\Spec(B),\Spec(\tB))$ with $A\to B$ is a finite \'etale map of henselian local rings 
associated to the residue field extension $L/K$ which is Galois 
and tame with respect to $v$ and $\tB=B\times_L \sO_w$ with $w$ the valuation on $L$ extending $v$. Set
\[B_n:=B^{\otimes_{A} n}, \quad V_n:=\Spec B_n \quad\text{and} \quad
\tB_n:=\tB^{\otimes_{\tA}n },\quad \tV_n:=\Spec \tB_n.\]
Let $B_n^{int}$ be the integral closure of $\tB_n$ in $B_n$ and set $\tV_n^{int}=\Spec \tB_n^{int}$.
As $(V_n,\tV_n^{int})\to (V_n,\tV_n)$ is a quasi-modification, see Remark \ref{rmk;int},
the desired vanishing holds once we know that the following complex is exact
\eq{prop:tame-local-vanishing1}{
0\to F(U,\tU)\to F(V,\tV^{\rm int})\to F(V_2,\tV^{\rm int}_2)\to F(V_3,\tV_3^{\rm int})\to\cdots. }
By \cite[VI, \S8, No. 6, Proposition 6]{BourbakiCA} the ring $\sO_w$ is also the integral closure of $\sO_v$ in $L$.
Hence the Galois group ${\rm Gal}(L/K)$ is equal to the decomposition group ${\rm Aut}(\sO_w/\sO_v)$.
Moreover, as the category of finite separable field extensions of $K$ is equivalent to the category of finite local \'etale $A$-algebras,
we can identify the $A$-algebra automorphisms of $B$ with ${\rm Gal}(L/K)$.
Hence  $G:={\rm Gal}(L/K)={\rm Aut}(B/A)={\rm Aut}(\tB/\tA)$. 
As in \cite[Example 2.6]{MilneEtCoh} the isomorphism $B_2\to \prod_{\sigma\in G} B$, $b_0\otimes b_1\mapsto (\sigma(b_0)b_1)_\sigma$
and induction give the isomorphism for $n\ge 2$
\[\varphi_n: B_n\to \prod_{(\sigma_0,\ldots, \sigma_{n-2})\in G^{n-1}}B\]
with 
\[\varphi_n(b_0\otimes \ldots \otimes b_{n-1})_{(\sigma_0,\ldots, \sigma_{n-2})}= 
(\sigma_{n-2}\cdots \sigma_0)(b_0)\cdot (\sigma_{n-2}\cdots \sigma_1)(b_1)\cdots \sigma_{n-2}(b_{n-2})\cdot b_{n-1}.\]
As $\tB$ is integral over $\tA$, so is $\tB_n$, hence $\tB_n^{\rm int}$ is the integral closure of $\tA$ in {$B_n$}
and thus $\varphi_n$ restricts to an isomorphism
\[\tB_n^{\rm int}\to \prod_{(\sigma_0,\ldots,\sigma_{n-2})} \tB.\]
We thus find isomorphisms 
\[(V_n,\tV_n^{\rm int})\cong (V,\tV)\times G^{n-1}\]
as in \cite[III, Example 2.6]{MilneEtCoh} and can therefore identify the cohomology of \eqref{prop:tame-local-vanishing1}
with  Galois cohomology
\[\pi_{-i}F(\cV_\lambda^{\times_{\cU_{\lambda}}\bullet})= H^i(G, F(V,\tV)).\]
This vanishes as  $F(V,\tV)$ is a $\Z_{(p)}$-module and  the order of $G$ is invertible in $\Z_{(p)}$ by tameness.
\end{proof}

A version of the following theorem for discretely ringed adic spaces with $p$-torsion coefficients 
holds by \cite[Proposition 8.5]{Huebner2021} and \cite[Corollary 5.7]{Huebner2025}.  
The method of proof  presented below differs from the one in {\em loc. cit.}
\begin{thm}\label{thm:tame-vs-etale}
Let $X\inj \tX$ be a quasi-compact open immersion of qcqs schemes.
For $\cV:=(V,\tV)\in (X,\tX)_t$, let $j^{\cV}:(V,\tV)_t\to \tV_{\et}$ be the morphisms of sites induced by the functor $\tV_{\et} \to (V,\tV)_t$ sending $(\tW\to\tV)$ to $(V\times_{\tV} \tW,\tW)$. 
 Let $F$ be a sheaf of abelian groups on $(X,\tX)_t$ such that the following condition is satisfied:
\begin{enumerate}[label=(p)]
\item\label{thm:tame-vs-etale1}
for every $(U,\tU)\to (X,\tX)$ in $(X,\tX)_t$, the stalk of the Zariski sheaf on $\tU$ given by $W\mapsto F(W\cap U, W)$
at any point $x\in \tU$ is a $\Z_{(p_x)}$-module, where 
$p_x$ is the exponential characteristic of $\kappa(x)$.
\end{enumerate}
Then, for $\cU=(U,\tU)$ in $(X,\tX)_t$, we have canonical isomorphisms
\[
H^i(\cU_t, F)\cong H^i(\cU_{\vet}, \nu_*F) \cong \colim_{\cV\to\cU} H^i(\tV_{\et}, j^{\cW}_*F)\cong \colim_{\cW\to\cU}H^i(\tV_{\et}, j^{\cV}_*F), \quad i\ge 0,
\]
where the first colimit runs along all quasi-modifications $\cV\to \cU$ with $\cV$ integral, and the second colimit runs along all admissible blow-ups $\cW\to \cU$.
\end{thm}

\begin{proof}
Note that $j^{\cV}$ is the composition of the morphism of sites 
\[\nu:(V,\tV)_t \to (V,\tV)_{\vet}\]
corresponding to the inclusion functor and the morphism of sites
\[\lambda^{\cV}\colon (V,\tV)_{\vet}\to \tV_{\et}\]
defined by the functor 
  $ \tV_{\et} \to (V,\tV)_{\vet}:\tW/\tV\mapsto (\tW\times_{\tV} V,\tW)$.
By Lemma \ref{lem:etaleCoh}, we have
\[H^i(\cU_{t},F) = H^i(\cU_{\vet},R\nu_*F) = \colim_{\cV\to \cU} H^i(\tV_{\et},\lambda^{\cV}_*R\nu_*F)= \colim_{\cW\to \cU} H^i(\tV_{\et},\lambda^{\cW}_*R\nu_*F).\]
Hence, it suffices to show $R^i\nu_*F=0$ for $i\ge 1$.
By Remark \ref{rmk;Deligne},
this follows from Proposition \ref{prop:tame-local-vanishing}
and the fact that the assumptions of {\em loc. cit.} are satisfied  by condition \ref{thm:tame-vs-etale1}. This completes the proof. 
\end{proof}

\begin{rmk}\label{rmk:tames-vs-etale}
Note that the condition \ref{thm:tame-vs-etale1} of Theorem \ref{thm:tame-vs-etale} is satisfied
if $F=\Omega^{q,t}$, see \ref{para:example-tame-sheaves}\ref{para:example-tame-sheaves2}, 
or $F=\W_S\Omega^{q,t}$  (see Definition \ref{def:DRWOmegaplus} below). Also it holds if $\tX$ is a $\Z_{(p)}$-scheme
and $F$ is any tame sheaf of $\Z_{(p)}$-modules.    
\end{rmk}

\section{Construction of tame sheaves}\label{sec:8}
Let $X\inj \tX$ be a quasi-compact open immersion of qcqs schemes. 
We give a method  to extend \'etale sheaves defined over  $X_{\et}$ to  sheaves on $(X,\tX)_t$.

\def\Vals{\mathrm{Val}^{\mathrm{s}}}

\begin{para}\label{para:constr-tame-sheaves}
For $\cX=(X,\tX)$, we let $\Vals_{\cX}$ be the category whose objects are triples $(L,w,\epsilon)$, 
where $L$ is a finite separable field extension of a residue field $k(x)$ of a point of $X$, 
$w$ is a valuation on $L$, and 
$\epsilon\colon \Spec(\cO_w)\to \tX$ is a morphism, which restricts to the map
$\Spec L \to \Spec k(x)\to X$,
and morphisms $(L,w,\epsilon)\to (L',w',\epsilon')$ are given by valued extensions 
(the compatibility with $\epsilon$ is  automatic). 
Note that for any $(x,v,\epsilon)\in \Spa(X,\tX)$ and any finite separable extension of valuation fields 
$(L, w)/(k(x), v)$ uniquely determines an element $(L,w,\epsilon')\in \Vals_{\cX}$ with $\epsilon'$ equal to the composition
$\epsilon':\Spec \sO_w\to \Spec \sO_v\xrightarrow{\epsilon} \tX$. 

For $(L, w,\epsilon)\in \Vals_{\cX}$ we set
\eq{para:constr-tame-sheaves0}{
\sO^h_{X,L}=\varinjlim_{\Spec L \to U\to X} \sO(U) \quad \text{and} \quad \sO^h_{\tX,L,w}= \sO^h_{X,L}\times_L \sO_w,}
where the direct limit is over  all étale maps $U\to X$ which factor $\Spec L\to X$.
Note that $\sO^h_{X,L}$ is the unique henselian local ring with residue field $L$ which is finite étale over $\sO_{X,x}^h$,
corresponding to the field extension $L/k(x)$. In particular the association 
$(L,w,\epsilon)\mapsto \sO^h_{X,L}$  defines a functor from $\Vals_{\cX}$ to the category of henselian local rings which are 
ind-étale over $X$.

Let $F$ be a sheaf on $X_{\et}$. We write $F(\sO^h_{X,L}):=\varinjlim_{\Spec L\to U\to X} F(U)$. 
Let 
\[\beta=\{F_w\subset F(\sO^h_{X,L})\}_{(L,w,\epsilon)\in \Vals_{\cX}},\]
be  a collection of subsets satisfying the following condition
\begin{enumerate}[label=({$\beta$}1)]
    \item\label{beta1} for any $(L,w,\epsilon)\to (L_1,w_1,\epsilon_1)$ in $\Vals_{\cX}$, 
    the pullback map $F(\sO^h_{X,L})\to F(\sO^h_{X,L_1})$ 
restricts to $F_w\to F_{w_1}$.
\end{enumerate}
For $(U,\tU)\in (X,\tX)_t$ we define
\[F_{\beta}(U,\tU):=\left\{a\in F(U)\,\middle\vert\,
\begin{minipage}[c]{7cm} for all $(x,v,\epsilon)\in \Spa(U,\tU)$ 
there exists  a finite tame extension $(L,w)/ (k(x), v)$, such that  $a_L\in F_w$\end{minipage}\right\},\]
where $a_L$ denotes the pullback of $a\in F(U)$ along  $\Spec \sO^h_{X,L}\to U$. 

Note that by Lemma \ref{lem:tame-Gal-closure} and \ref{beta1} it suffices to consider in the definition of $F_\beta(U,\tU)$ 
only the finite tame {\em Galois} extensions $(L,w)/(k(x),v)$. 
\end{para}

\begin{prop}\label{prop:constr-tame-sheaves}
Assumptions as in \ref{para:constr-tame-sheaves} above.
The assignment $(U,\tU)\mapsto F_{\beta}(U,\tU)$ defines a sheaf on $(X,\tX)_t$. 
\end{prop}

\begin{proof}
We start by showing that $F_\beta$ is a presheaf.
Let $(u,\tu): (U',\tU')\to (U,\tU)$ be a morphism in $(X,\tX)_t$ and take $a\in F_{\beta}(U,\tU)$.
Let $(y,v,\epsilon)\in \Spa(U',\tU')$. Set $x:=u(y)\in U$ and denote by $v_x$ the restriction of $v$ to $k(x)$.
Note that $\cO_{v_x}=\sO_v\cap k(x)$, the intersection taken inside $k(y)$. 
Hence $\epsilon: \Spec \sO_v\to \tU'\to \tU$ factors uniquely via a map $\epsilon_x:\Spec \sO_{v_x}\to \tU$.
We obtain a point $(x, v_x,\epsilon_x)\in \Spa(U,\tU)$. 
By definition  there exists a finite tame extension 
$(L,w)/(k(x),v_x)$ such that $a_L\in F_w$. 
Denote by $L_1$ the composition field of $L$ and $k(y)$ in a separable closure of $k(x)$ and choose a valuation $v_1$ 
on $L_1$ extending $v$. Then the extension $(L_1,v_1)/(k(y),v)$ is tame, 
by Lemma \ref{lem:tame-Gal-closure}\ref{lem:tame-Gal-clousre2}.
Now $u^*(a)_{L_1}$,   the image of $u^*(a)$ under $F(U')\to  F(\sO^h_{U',L_1})$, 
is equal to the image of the pullback of $a_L$ under $F(\sO^h_{U,L})\to F(\sO^h_{U',L_1})$. 
As $a_L\in F_w$ we find $u^*(a)_{L_1}\in F_{w_1}$ by \ref{beta1} in \ref{para:constr-tame-sheaves}.
This shows $u^*(a)\in F_\beta(U',\tU')$. Hence $F_\beta$ is a presheaf.

As $F$ is an \'etale sheaf on $X$, $F_\beta$ will be a sheaf on $(X,\tX)_t$ if we show the following:
Let $\{(U_i,\tU_i)\to (U,\tU)\}_{i\in I}$ be a tame covering in $(X,\tX)_t$ and let $a\in F(U)$,
then
\eq{prop:constr-tame-sheaves1}{ a_{|U_i}\in F_{\beta}(U_i,\tU_i), \quad \text{for all }i \, \Longrightarrow\,  a\in F_\beta(U,\tU).}
Let $(x,v,\epsilon)\in \Spa(U,\tU)$. By definition of tame covering we find  an index  $i$ and a point 
$(y,w,\epsilon')\in \Spa(U_i,\tU_i)$ over $(x,v,\epsilon)$ such that $(k(y),w)/(k(x),v)$ is a finite tame extension.
As $a_{|U_i}\in F_{\beta}(U_i, \tU_i)$ we find a finite tame  extension $(L,w_1)/(k(y), w)$ such that 
$(a_{|U_i})_L\in F_{w_1}$. 
Hence, we get \eqref{prop:constr-tame-sheaves1}.
\end{proof}

\begin{rmk}\label{rmk:functoriality}
    By definition, for all $F$, $\beta$ as in \ref{para:constr-tame-sheaves} above we have a pullback diagram 
    \[
    \begin{tikzcd}
        F_\beta(U,\tU)\ar[r]\ar[d]& \prod \colim F_w\ar[d]\\
        F(U)\ar[r]&\prod \colim F(\sO^h_{U,L})
    \end{tikzcd}
    \]
where:
    \begin{itemize}
    \item the product ranges over all elements $(x,v,\epsilon)$ of $\Spa(U,\tU)$
    \item the colimit ranges over all $(k(x),v)\subseteq (L,w)$ finite tame.
    \end{itemize}
    
    This implies that if $\phi\colon F\to G$ is a map of sheaves on $X_\et$ and $F$ and $G$ are equipped with $\beta$-families $\beta_F:=\{F_w\}$ and $\beta_G:=\{G_w\}$ such that for every $(U,\tU)\in (X,\tX)_t$ and every $(x,v,\epsilon)\in \Spa(U,\tU)$ there is a cofinal system of  tame extensions $S_{(x,v,\epsilon)}:=\{(k(x),v)\subseteq (L,w)\}$ such that the map $\varphi: F(\sO^h_{X,L})\to G(\sO^h_{X,L})$ restricts to a map 
    $\phi_w\colon F_w\to G_w$ for all $(L,w)\in S_{x,v,\epsilon}$, then $\phi$ induces a map $F_{\beta_F}\to G_{\beta_G}$ 
    of sheaves on $(X,\tX)_t$,  denoted by $\varphi$ as well. 
\end{rmk}

\begin{example}\label{para:example-tame-sheaves}
In the following $(L,w,\epsilon)$ is always in $\Vals_{\cX}$.
We set $A_{L}=\sO^h_{X,L}$ and $A_{L,w}=\sO^h_{\tX,L,w}$, see \eqref{para:constr-tame-sheaves0} for notation
\begin{enumerate}[label=(\arabic*)]
\item Let $F$ be a sheaf on $X_{\et}$. Set
$F_{w,0}=F(A_L)$, if $(L,w,\epsilon)\in \Vals_{(X,X)}$, and $F_{w,0}=0$, else. 
Set $F_{w, {\rm triv}}=F(A_L)$, for all $(L,w,\epsilon)$. Then $\beta_0=\{F_{w,0}\}_{(L,w,\epsilon)}$
and $\beta_{\rm triv}=\{F_{w,{\rm triv}}\}$ are families as in \ref{para:constr-tame-sheaves} and we have 
\[\imath_* F_{\beta_0}= j_!F\quad \text{and}\quad \imath_* F_{\beta_{\rm triv}}=j_*F,\]
where $j: X\to \tX$ is the open immersion and $\imath: (X,\tX)_t\to \tX_{\et}$ is the morphism of sites induced by the functor 
$\tU\mapsto (\tU\times_{\tX} X, \tU)$.
\item\label{para:example-tame-sheaves1} Let $F$ be as above  and assume there is a presheaf $\tF$ on $\Sch_{\tX}$
extending $F$. Then we can define $\beta_{\tF}=\{\tF(A_{L,w})\subset F(A_L)\}_{(L,w,\epsilon)}$.
We get a sheaf $F_{\beta_{\tF}}$ on $(X,\tX)_t$. Notice that the resulting sheaf $\iota_*F_{\beta_{\tF}}$ 
may be different from $\tF_{|\tX_{\et}}$, e.g., 
for $\tF=\cO$, the structure sheaf on $\Sch_{\tX}$,
the sheaf $\cO_{|\tX_{\et}}$ is different from $\iota_*\cO_{\beta_{\tF}}=\iota_*\sO^t$, where $\cO^t$ is defined in \ref{para:example-tame-sheaves2}, see Lemma \ref{lem:Gat} below.
\item\label{para:example-tame-sheaves1.5} Given a collection $\beta$ as in \eqref{para:constr-tame-sheaves} on a sheaf $F$, we can 
define a new family 
\[\beta^{div}=\{F^{div}_w\subset F(A_L)\}_{(L,w,\epsilon)\in \Vals_{\sX}}\]
by setting
\[F^{div}_w:=
\begin{cases}
F_w &\begin{minipage}[t]{0.5\textwidth}
if  $\Spec L$ maps to a generic point of  $X$ \\
 and  $w$ is a discrete valuation,
\end{minipage}\\
F(A_L) & \text{else}.    
\end{cases}\]
If $\beta$ satisfies \ref{beta1}, then so does $\beta^{div}$. Here note that if $(L,w,\epsilon)\to (L',w',\epsilon')$ is a morphism in 
$\Vals_{\sX}$ and $w$ is discrete, then $w'$ is discrete as well. We denote by $F^{div}_{\beta}:= F_{\beta^{div}}$ 
the corresponding tame sheaf. For sections in $F^{div}_{\beta}(U,\tU)$ we only put conditions induced by $\beta$ 
along tame extensions of discrete valuations on the generic points of $U$ with center in $\tU$, hence 
$F_\beta\subset F^{div}_{\beta}$.
\item\label{para:example-tame-sheaves2} 
Let $F=\Omega^q$ be the \'etale sheaf of $q$th absolute differential forms  on $X_{\et}$.
Denote by  $\Omega^*_{A_{L,w}}(\log)$ the graded $\Omega^*_{A_{L,w}}$-subalgebra of 
$\Omega^*_{A_L}$ generated by  $\dlog (A_L^\times)$. 
Note $\Omega^q_{A_{L,w}}\subset \Omega^q_{A_{L,w}}(\log )$.
Set 
\[\beta_{\log}:= \{\Omega^q_{A_{L,w}}(\log)\subset\Omega^q_{A_L}\}_{(L,w,\epsilon)}.\]
Then $\beta_{\log}$ satisfies  \ref{beta1} in \ref{para:constr-tame-sheaves} and we get a sheaf
\eq{para:example-tame-sheaves21}{\Omega^{q,t}:=\Omega^q_{\beta_{\log}} \quad \text{on } (X,\tX)_t.}
Note for $q=0$, we write $\cO^t=\Omega^{0,t}$. This is a special case of  \ref{para:example-tame-sheaves1},
where we take $\tF$ to be the structure sheaf $\cO$ on $\Sch_{\tX}$.
 Note also that the differential of the de Rham complex induces a well-defined differential 
$d: \Omega^{q,t}\to \Omega^{q+1,t}$ giving rise to a complex of 
sheaves $\Omega^{\bullet,t}$ on $(X,\tX)_t$. Using \ref{para:example-tame-sheaves1.5} above we get an inlcusion of sheaves on $(X,\tX)_t$
\[\Omega^{q,t}\subset \Omega^{q, div}:=(\Omega^q_{\beta_{\log}})^{div}.\]
Furthermore, given a morphism $\tX\to S$ we can similarly define the complex
$\Omega^{\bullet,t}_{/S}=\Omega^\bullet_{/S, \beta_{\log/S}}$, where 
$\beta_{\log/S}=\{\Omega^q_{A_{L,w}/S}(\log)\subset\Omega^q_{A_L/S}\}_{(L,w,\epsilon)}$.
\item\label{para:example-tame-sheaves3} Let $W_n\Omega^\bullet$  denote the $p$-typical de Rham--Witt complex,
it is defined as an \'etale sheaf on all schemes, by \cite{Hesselholt-dRW}. 
Denote by $W_n\Omega^*_{A_{L,w}}(\log)$ the graded $W_n\Omega^*_{A_{L,w}}$-subalgebra of $W_n\Omega^*_{A_L}$ generated
by $\dlog[a]$, $a\in {A_L}^\times$, where $[-]: A_L\to W_n(A_L)$ denotes the multiplicative lift.
Note $W_n\Omega^q_{A_{L,w}}\subset W_n\Omega^q_{A_{L,w}}(\log)$.
Then $\beta_{\log,n}=\{W_n\Omega^q_{A_{L,w}}(\log)\subset W_n\Omega^q_{A_L}\}_{(L,w,\epsilon)}$
satisfies the assumption from \ref{para:constr-tame-sheaves} and we get  sheaves 
\[W_n\Omega^{q,t}:=W_n\Omega^q_{\beta_{\log,n}}\subset W_n\Omega^{q,div}:=(W_n\Omega^{q}_{\beta_{\log,n}})^{div}\quad 
\text{on } (X,\tX)_t.\]
Note that the differential, the restriction map, as well as  Frobenius and Verschiebung on the de Rham--Witt complex 
induce well-defined maps on $W_n\Omega^{q,t}$ and $W_n\Omega^{q,div}$. In particular $W_\bullet\Omega^{\bullet, t}$
and $W_n\Omega^{\bullet,div}$ are Witt complexes in the sense of \cite[Definition 4]{Hesselholt-dRW}.
\end{enumerate}    
\end{example}
The following is a global version of the classical result that the integral closure of a subring  $A$ of a field $K$ 
is the intersection of all the valuation rings on $K$ containing $A$, see \cite[VI, \S1, Theorem 3]{BourbakiCA}.
See also \cite[Lemma 1.3.3]{DAddezio} for a similar computation in the setup of marked rings.
\begin{lemma}\label{lem:Gat}
Let $\sO^t=\Omega^{0,t}$ and $\sO^{div}=\Omega^{0,div}$ in the notation 
from \ref{para:example-tame-sheaves}\ref{para:example-tame-sheaves2}.
For all $(U,\tU)\in (X,\tX)_t$     we have 
\[\sO^t(U,\tU)=\sO(\tU^{\rm int}),\]
where $\tU^{\rm int}$ denotes the integral closure of $\tU$ in $U$.
Moreover, if $\tU$ is a Nagata scheme and $U$ is normal, then we have 
\[\sO^t(U,\tU)=\sO^{div}(U,\tU).\]
\end{lemma}
\begin{proof}
By Lemma \ref{lem:blow-up-inv} we can assume $\tU=\tU^{\rm int}$.
We notice that both sides of the first equality in the statement 
are contained in $\sO(U)$  and that this "$\supset$" inclusion holds by definition of the left hand side.
As both sides are sheaves on $\tU$ it suffices to check the other inclusion for $\tU$ affine.
As an affine open cover $U=\cup_i U_i$ induces an affine open cover 
$\tU=\cup_i \tU_i^{\rm int}$ we can assume $(U,\tU)=(\Spec A, \Spec \tA)$ with $\tA$ integrally closed  in $A$.
Let $x\in A\setminus \tA$. 

\begin{claim}
There exist a  prime ideal $\fp$ in $A$ and a valuation $v$ on $K=\Frac(A/\fp)$ such that 
$\tA\to A\to K$ factors through $\cO_v$ and $v(x)<0$. 
\end{claim}

Admitting the claim, we also have $w(x)<0$ for all (tame) extensions $(L,w)/(K,v)$ of valuation fields and $x$ cannot lie in $\sO^t(U,\tU)$, which completes the proof of the lemma.

\medskip

We prove the claim. Denote by $C=\tA[1/x]$ the subring of $A_x$ generated by the image of $\tA$ and $1/x$.
As $x$ is not integral over $\tA$ it follows from  \cite[VI, \S1, no. 2, Lemma 1]{BourbakiCA}, 
that $A_x$ is not the zero-ring and  that there exists a maximal ideal $\fm\subset C$ such that
$1/x\in \fm$ and  $\fm\cap \tA$ is a maximal ideal of $\tA$.
Consider the  localization $C_\fm$ of $C$. By the flatness of $C_\fm$ over $C$,
the inclusion $C\inj A_x$ induces an injection of rings
$C_\fm=C\otimes_C C_\fm \inj A_x\otimes_C C_{\fm}$. As $C_{\fm}$ is not the zero ring, the ring 
$A_x\otimes_C C_{\fm}$ is not zero and hence has a prime ideal; 
it corresponds to a prime ideal $\fp$ of $A$, which does not contain $x$ and has empty intersection with
$C\setminus \fm$. Set $\fp_0:=C\cap \fp A_x$. By construction we have $\fp_0\subset \fm$, and in fact
this inclusion is strict as else we would have $1/x\in \fp A_x$ since $1/x\in \fm$. 
Thus, $C/\fp_0$ is not a field.
Set $K_0:=\Frac(C/\fp_0)\subset K:=\Frac(A_x/\fp A_x)=\Frac(A/\fp)$.
Let $v_0$ be a valuation on $K_0$ such that $C/\fp_0\subset \sO_{v_0}\subset K_0$
and that $\fm_{v_0}\cap C/\fp_0$
is the image of $\fm$ in $C/\fp_0$. 
Let $v$ be a valuation on $K$ extending $v_0$.
Thus we obtain a commutative diagram
\[
   \begin{tikzcd}
   \tA\ar[r]\ar[d] & C/\fp_0\ar[r]\ar[d] & \sO_{v_0}\ar[r] & \sO_v\ar[d]\\
  A\ar[r] & A_x/\fp A_x\ar[rr] &  & K.
   \end{tikzcd}
  \]
Note $v(x)=v_0(x)<0$ since $1/x\in \fm$ so that its image in $C/\fp_0$ is in $\fm_{v_0}$.
This completes the proof of the claim.

For the last statement, we observe that as $\tU$ is Nagata and $U$ is normal
the integral closure $\tU^{int}$ is finite over $\tU$ and hence is locally noetherian again. As $U$ is normal so is $\tU^{int}$. We may therefore assume that $\tU^{int}$ is noetherian, integral and normal.
Hence
\[\sO^t(U,\tU)=\sO(\tU^{\int})=\bigcap_{x\in  {\tU}^{int,(1)}}\sO_{\tU^{int},x}= \sO^{div}(U,\tU),\]
where $\tU^{int, (1)}$ is the set of 1-codimensional points in $U^{int}$ and  
where the last equality follows from $\sO^{div}(U,\tU)=\sO^{div}(U,\tU^{int})$ and the definition of $\sO^{div}$, 
see \ref{para:example-tame-sheaves1.5}.
\end{proof}

\begin{rmk}\label{rmk:Ot}
Combining Lemma \ref{lem:Gat} and Theorem \ref{thm:tame-vs-etale}, we conclude that for $S$ a qcqs scheme 
and $X$ a normal separated $S$-scheme of finite type and $\tX$ any $S$-compactification, the following isomorphism holds for all $i\geq 0$:\[
H^i((X,\tX)_t,\cO^t) \cong \colim_{\tY} H^i(\tY_{\Zar},\cO_{\tY}),
\] 
where the colimit is over all quasi-modifications $(X,\tY)\to (X,\tX)$ with $\tY$ normal. An analogous result for certain adic spaces was obtained in  
\cite[Corollary 13.6]{Huebner2021}.
\end{rmk}

\section{Tame big de Rham--Witt complex}\label{sec:9}

In this section we define the tame big de Rham--Witt complex for any quasi-compact open immersion $X\subset \tX$ of qcqs schemes,
see Definition \ref{def:DRWOmegaplus}, generalizing the  tame sheaves $\Omega^{\bullet, t}_{-/\Z}$ and $W_n\Omega^{\bullet, t}$ 
introduced in  \ref{para:example-tame-sheaves}. We compute the big de Rham--Witt complex for local rings in the tame site, 
see Proposition \ref{prop:DRWOmegaplus} below. In the next section 
we will use these results to analyse the tame cohomology of 
the tame $p$-typical de Rham--Witt complex in positive characteristic.

\begin{para}\label{para:dRW}
We denote by $\W_S\Omega^q_A$ the degree $q$-part of the big de Rham--Witt complex of a ring $A$ at the truncation set $S$,
defined by Hesselholt in \cite[Theorem B]{Hesselholt-dRW}. If $X$ is a scheme, then we denote by $\W_S\Omega^q_X$ 
the \'etale sheaf on $(\rm  Sch/X)_{\et}$ associated to the presheaf
\[(T\to X)\mapsto \W_S\Omega^q_{\Gamma(T,\sO_T)}.\]
For $S$ a finite truncation set, it follows from the fact that $\W_S(A)\to \W_S(B)$ is \'etale if $A\to B$ is 
(see \cite[Theorem 2.4]{van-der-Kallen} or \cite[Theorem B]{Borger}, cf. also \cite[Theorem 1.25]{Hesselholt-dRW}), 
\'etale base change \cite[Theorem C]{Hesselholt-dRW}, and descent that 
\[\Gamma(\Spec A, \W_S\Omega^q_X)=\W_S\Omega^q_A, \quad \text{for any }  \Spec A\to X.\]
In particular $\W_{\{1\}}\Omega^q_X$  is the usual sheaf of absolute Kähler differentials of degree $q$.
Note however, that $\W_S\Omega^*_A$ is not a complex in general as $dd(a)=\dlog[-1]da$, 
for $a\in \W_S\Omega^*_A$ 
and $\dlog[-1]$ does not vanish in general. 
If $X$ is an $\F_p$-scheme, then 
\[\W_{\{1,p,\ldots, p^{n-1}\}}\Omega^q_X=W_n\Omega^q_X\]
is the degree $q$ and length $n$ part of the $p$-typical de Rham--Witt complex of Bloch--Deligne--Illusie \cite{IllusiedRW},
see  \cite[Remark 4.2(c)]{Hesselholt-dRW}. Furthermore, 
in this case the big de Rham--Witt complex at a general truncation set decomposes
as a product of $p$-typical ones, see \cite[Corollary 1.2.6]{Hesselholt-Madsen-NilEnd}.
\end{para}

\begin{lemma}\label{lem:dRW-quotient}
Let $A$ be a ring and $I\subset A$ an ideal.
Then the kernel of the natural surjection $\W_S\Omega^*_A\to \W_S\Omega^*_{A/I}$ is the  graded ideal
$J_S^*$ which in degree $q$ is given by 
\[J_S^q:= \W_S(I)\cdot \W_S\Omega^q_A+d(\W_S(I)\cdot \W_S\Omega^{q-1}_A),\]    
where $\W_S(I):=\Ker(\W_S(A)\to \W_S(A/I))$. 
\end{lemma}
\begin{proof}
This is a version of \cite[Lemma 2.4]{Geisser-Hesselholt} for the big de Rham--Witt complex.
Note  that by the Leibniz rule \cite[Definition 4.1 (i)]{Hesselholt-dRW} we have 
\eq{lem:dRW-quotient0}{J_S^q= \W_S(I)\cdot \W_S\Omega^q_A+d(\W_S(I))\cdot \W_S\Omega^{q-1}_A}
We use this presentation of $J^q_S$ in the following.
We get a functor from truncation sets to anticommutative graded rings 
$S\mapsto \W_S\Omega^*_A/ J_S^*$ and as in {\em loc. cit.} it suffices to show that the differential $d$, Frobenius $F_n$, and Verschiebung $V_n$
on $\W_{-}\Omega^*_A$ induce well-defined maps on $\W_{-}\Omega^*_A/ J_{-}^*$. Indeed, in this case 
the natural surjection of Witt complexes $\W_{-}\Omega^*_A\to \W_{-}\Omega^*_{A/I}$ induces a surjection
of Witt-complexes $\W_{-}\Omega^*_A/J_{-}^*\to \W_{-}\Omega^*_{A/I}$ and the universal property of the de Rham--Witt complex 
yields the inverse map.

Every element in $J_S^*$ can be written as a sum of elements
\[V_{n}([a])\cdot \beta \quad \text{and} \quad dV_n([a])\cdot \beta,\]
where $\beta\in \W_S\Omega^*_A$, $a\in I$,  $[a]\in \W_{S/n}(I)$ denotes the multiplicative lift of $a$, and $V_n$ 
is the $n$th Verschiebung.  By \cite[Definition 4.1 (i)]{Hesselholt-dRW} we have 
\eq{lem:dRW-quotient1}{d(V_n([a])\beta)= dV_n([a]) \beta+ V_n([a]) d\beta}
and 
\[d(dV_n([a])\beta)= \dlog[-1]dV_n([a])\beta- dV_n([a]) d\beta= -dV_n([a])(\dlog[-1]\beta+d\beta).\]
Hence $d(J^*_S)\subset J^{*+1}_S$. As $F_m: \W_S\Omega^*_A\to \W_{S/m}\Omega^*_A$ ($m\ge 1$) is a ring map 
which restricts to a map $\W_S(I)\to\W_{S/m}(I)$ the containment $F_m(J^*_S)\subset J^*_{S/m}$ follows from 
\begin{align}
F_mdV_n([a]) &=  idF_{m_1}V_{n_1}([a])+jF_{m_1}V_{n_1}d([a]) 
                                                      +(c-1)\dlog[-1] F_{m_1}V_{n_1}([a]) \label{lem:dRW-quotient2}\\
             &=  idV_{n_1}([a^{m_1}])+jV_{n_1}([a]^{m_1-1}d[a]) 
                                                        + (c-1) V_{n_1}([a]^{m_1})\dlog[-1] \notag\\
             &=  idV_{n_1}([a^{m_1}]) +jV_{n_1}([a]^{m_1-1})dV_{n_1}([a])
              +(c-1-jn_1+j)V_{n_1}([a]^{m_1})\dlog[-1],             \notag
\end{align}
where $c={\rm gcd}(m,n)$, $m=cm_1$, $n=cn_1$,  and $i,j\in \Z$ satisfy $im+nj=c$.
Here the first equality is the third equation in \cite[Lemma 4.3]{Hesselholt-dRW}, 
the second equality uses $V_{r}F_{s}=F_{s}V_{r}$, for $(r,s)=1$, $F_{s}([a])= [a]^{s}$, and $F_s(d[a])=[a]^{s-1}d[a]$,
and the third equality uses $F_rdV_r([a])=d[a]+ (r-1)\dlog[-1]\cdot [a]$, $F_r(\dlog[-1])=\dlog[-1]$, 
and $V_r(F_r(x)y)=xV_r(y)$, see \cite[Definition 4.1, Lemma 4.3]{Hesselholt-dRW}.
(If $m_1=1$ we can choose $i=1$ and $j=0$ and deduce from the second equality  that $F_mdV_n([a])$ lies in  $J_{S/m}^1$.)
It remains to show $V_m(J^*_{S})\subset J^*_T$ for any truncation set $T$ with $T/m=S$. 
From $\eqref{lem:dRW-quotient1}$ and $V_md=mdV_m$ we see that it suffices to show
$V_m(a\beta)\in J^*_T$, for any $a\in \W_S(I)$. As by construction of the big de Rham--Witt complex a form $\beta$
is a sum of elements $b_0d b_1\cdots db_{q}$ with $b_i\in \W_S(A)$ and since  $a b_0\in \W_S(I)$
it suffices to consider elements $\beta=db_1\cdots db_q$. 
As $db_i= F_mdV_m(b_i)- (m-1)b_i\dlog[-1]$  and $\dlog[-1]\cdot\dlog[-1]=0$ \cite[Lemma 4.3]{Hesselholt-dRW} we find
\begin{align*}
    adb_1\cdots db_q & =a F_m\left(dV_m(b_1)\cdots dV_m(b_q)\right)\\
& -(m-1)\sum_{i=1}^q ab_i F_m\left(dV_m(b_1)\cdots \underbrace{\dlog[-1]}_{i\text{th spot}}\cdots dV_m(b_q)\right). 
\end{align*}
Thus $V_m(adb_1\cdots db_q)\in J_T^*$ follows from the formula $V_m(xF_m(y))=V_m(x)y$. 
This completes the proof of the lemma.
\end{proof}

In the following, the {\em truncation set $S$ is always finite}.

\begin{para}\label{para:log-dRW}
Let $(A,\fm_A)$ be a local ring with residue field $K=A/\fm_A$. Let $v$ be a valuation on $K$ 
with valuation ring $\sO_v$ and set $\tA=A\times_K \sO_v$. 
Note that $\fm_A$ is a prime ideal of $\tA$.
In this situation, we define $\W_S\Omega^q_{\tA}(\log)$
to be the degree $q$-part of the graded  $\W_S\Omega^*_{\tA}$-subalgebra of $\W_S\Omega^*_A$ generated 
by forms $\dlog[a]$, with $a\in A^\times$, where $[a]\in \W_S(A)$ denotes the multiplicative lift. 

This extends the definition of $\Omega^q_{\tA}(\log)$ from \ref{para:example-tame-sheaves}\ref{para:example-tame-sheaves2}.
Note that $\W_S\Omega^0_{\tA}(\log )=\W_S(\tA)$.
Moreover, $d$, $F_n$, and $V_n$ restrict to maps on $\W_{-}\Omega^*_{\tA}(\log)$. 
Indeed for $F_n$ this follows from the fact that $F_n$ is a ring map and  the formula $F_n\dlog[a]=\dlog[a]$;
for $V_n$ this follows from the latter formula and the $F_n$-linearity of $V_n$; 
for $d$ this follows from the Leibniz rule and the formula
\eq{para:log-dRW1}{d\dlog[a]=-\dlog[a]\dlog[a]+\dlog[-1]\dlog[a]=2 \dlog[-1]\dlog[a]=0,}
where the first equality uses Leibniz rule and $dd[a]=\dlog[-1]d[a]$ and the second
equality uses $d[a]d[a]=-\dlog[-1]F_2d[a]=-[a]\dlog[-1]d[a]$.
\end{para}

\begin{lemma}\label{lem:locetbc-log}
Let $A,\fm_A, K=A/\fm_A, v,\tA$  be as in \ref{para:log-dRW} above. Let $\tA\to\tB$ be an \'etale local homomorphism of local rings.
Then $\tB/\fm_A \tB$ is a valuation ring $\sO_w$, which is unramified over $\sO_v$ and $\tB=B\times_L \sO_w$,
where $B=\tB\otimes_{\tA} A$ is a local ring with maximal ideal $\fm_A B$ and residue field  $L=\Frac(\sO_w)$. Moreover, 
$\W_S(\tB)$ is a faithfully flat \'etale $\W_S(\tA)$-algebra and  
multiplication induces a canonical isomorphism
\eq{lem:locetbc-log1}{\W_S(\tB)\otimes_{\W_S(\tA)} \W_S\Omega^q_{\tA}(\log)\xrightarrow{\simeq} 
\W_S\Omega^q_{\tB}(\log).}
 \end{lemma}
\begin{proof}
By assumption  $\sO_v=\tA/\fm_A\to \tB/\fm_A \tB$ is an \'etale local homomorphism and it follows from the arguments 
in \cite[Remark 6.1.12(vi)]{Gabber-Ramero} that $\tB/\fm_A \tB=\sO_w$ is a valuation ring, which is 
unramified over $\sO_v$ by \cite[Lemma 6.2.5]{Gabber-Ramero}. 
In particular, $\sO_w$ is the integral closure of $\sO_v$ in $L$ and hence 
$L=K\otimes_{\sO_v}\sO_w= A\otimes_{\tA} \tB/\fm_A \tB$. It follows that $B=A\otimes_{\tA}\tB$
is a local ring with maximal ideal $A\otimes_{\tA}\fm_A\tB=\fm_A B$, where we use for
the last equality that $A=\tA_{\fm_A}$ is a localization of $\tA$. On the other hand 
as $\tB$ is flat over $\tA$ and $\fm_A$ is an ideal of $\tA$ we have 
$\fm_A\otimes_{\tA} \tB=\fm_A\tB$. Thus $\fm_A B=\fm_A\tB$.
The commutative diagram with exact rows
\[\xymatrix{
0\ar[r] & \fm_A \tB\ar@{=}[d]\ar[r] & \tB\ar[r]\ar[d] & \tB/\fm_A\tB\ar@{=}[d]\ar[r] & 0\\
0\ar[r] & \fm_A B\ar[r] & B\times_L \sO_w \ar[r] & \sO_w\ar[r] & 0
}\]
therefore yields $\tB=B\times_L \sO_w$. 

As already noted in \ref{para:dRW}, the ring $\W_S(\tB)$
is an \'etale $\W_S(\tA)$ algebra. It is also faithfully flat if we show that
$\Spec \W_S(\tB)\to \Spec \W_S(\tA)$ is surjective. Let $p={\rm char}(\sO_v/\fm_v)$. 
Then $\tA$ is a $\Z_{(p)}$-algebra and hence it suffices to show that 
$\Spec W_n(\tB)\to \Spec W_n(\tA)$ is surjective for all $n\ge 1$, e.g., 
\cite[Proposition 1.10]{Hesselholt-dRW}. 
By \cite[8.1 Proposition]{Borger} there is an 
integral ring map $W_n(\tA)\to W_{n-1}(\tA)\times \tA$ with nilpotent kernel 
and which is functorial in $\tA$. Thus the induced map 
$\Spec (W_{n-1}(\tA)\times \tA)\to \Spec \, W_n(\tA)$ is surjective,
and similarly with $\tB$ instead of $\tA$. 
As $\Spec (W_{n-1}(\tB)\times \tB) \to \Spec (W_{n-1}(\tA)\times \tA)$ is surjective by induction,
we see that $\Spec W_n(\tB)\to\Spec W_n(\tA)$ is surjective as well.

Finally, we show that \eqref{lem:locetbc-log1} is an isomorphism.
Consider the composition 
\[\W_S(\tB)\otimes_{\W_S(\tA)} \W_S\Omega^q_{\tA}(\log)
\to \W_S(\tB)\otimes_{\W_S(\tA)} \W_S\Omega^q_{A}\xrightarrow{\simeq} \W_S(B)\otimes_{\W_S(A)} \W_S\Omega^q_A \xrightarrow{\simeq}\W_S\Omega^q_B\]
where the first map is injective as $\W_S(\tB)$ is flat over $\W_S(\tA)$,
the second map is an isomorphism as 
$\W_S(B)=\W_S(\tB\otimes_{\tA} A)=\W_S(\tB)\otimes_{\W_S(\tA)} \W_S(A)$, 
see \cite[9.4 Corollary]{Borger}, and the last map is an isomorphism by \'etale base change,
see \cite[Theorem C]{Hesselholt-dRW}. Hence \eqref{lem:locetbc-log1} is injective
and by \'etale base change for $\tB/\tA$ it remains to show that the elements 
$\dlog[b]$, for $b\in B^\times$, are in the image of \eqref{lem:locetbc-log1} for $q=1$.
Let $\bar{b}$ be the image of $b$ in $\sO_w$. As $\sO_w/\sO_V$ is unramified 
we find $a\in A^\times$ and $u\in \tB^\times$ such that $\bar{b}= \bar{a} \bar{u}$.
Thus $b=au(1+c)$, for some $c\in \fm_B$. Hence
\[\dlog [b] = \dlog[a] +\dlog [u(1+c)].\]
Clearly the first summand lies in the image of \eqref{lem:locetbc-log1} and  
as $u(1+c)\in \tB^\times$ the second summand lies in $\W_S\Omega^1_{\tB}$
and therefore also in the image of \eqref{lem:locetbc-log1} by \'etale base change for
$\tB/\tA$. This completes the proof.
\end{proof}

\begin{defn}\label{def:DRWOmegaplus}
Let $X\inj \tX$ be a quasi-compact open immersion between qcqs schemes.
We denote by $\W_S\Omega^{q,t}$ the tame sheaf on $(X,\tX)_t$ associated by \ref{para:constr-tame-sheaves} 
and Proposition \ref{prop:constr-tame-sheaves} to
\[\beta=\{\W_S\Omega^q_{\sO^h_{\tX,L,w}}(\log )\subset \W_S\Omega^q_{\sO_{X,L}^h}\}_{(L,w, \epsilon)\in \Vals_{\cX}}.\]
Considering only the tame extensions over the generic points of $X$ and discrete valuations yields the tame sheaf
$\W_S\Omega^{q,div}$, see  \ref{para:constr-tame-sheaves}\ref{para:example-tame-sheaves1.5}.
Note $\W_S\Omega^{q,t}\subset \W_S\Omega^{q,div}$. 

If $S=\{1\}$, these sheaves are equal to the sheaves $\Omega^{q,t}$ and $\Omega^{q, div}$ 
defined in \ref{para:example-tame-sheaves}\ref{para:example-tame-sheaves2}
and if $S=\{1, p,\ldots, p^{n-1}\}$ these sheaves are equal to the sheaves $W_n\Omega^{q,t}$ and $W_n\Omega^{q, div}$ defined 
in \ref{para:example-tame-sheaves}\ref{para:example-tame-sheaves3}.
By the above $\W_{-}\Omega^{\bullet,t}(U,\tU)$ and $\W_{-}\Omega^{\bullet,div}(U,\tU)$ are Witt complexes over 
$\sO(\tU)$, in the sense of \cite[Definition 4.1]{Hesselholt-dRW}, for all $(U,\tU)\in (X,\tX)_t$. 
\end{defn}

We aim to give a more explicit description of $\W_S\Omega^{q,t}$ and $\W_S\Omega^{q,div}$ 
in various situations of geometric and arithmetic interest.
This requires some preliminary work.

\begin{lemma}\label{lem:ses-log-dRW}
In the situation of \ref{para:log-dRW}, let $J^q_S$ be the kernel of 
$\W_S\Omega^q_A\to \W_S\Omega^q_K$.
Then the diagram with vertical maps the natural inclusions
\[\xymatrix{
0\ar[r] &J^q_S\ar@{=}[d]\ar[r] & \W_S\Omega^q_A\ar[r] &\W_S\Omega^q_K\ar[r]& 0\\
0\ar[r] & J^q_S\ar[r]    & \W_S\Omega^q_{\tA}(\log )\ar[r]\ar[u] & \W_S\Omega^q_{\sO_v}(\log )\ar[r]\ar[u] & 0
}\]
is commutative with exact rows and cartesian square on the right. 
Moreover,
\eq{lem:ses-log-dRW0}{J^q_S= 
\W_S(\fm_A)\cdot\W_S\Omega^q_{\tA}(\log)+ d(\W_S(\fm_A)\cdot\W_S\Omega^{q-1}_{\tA}(\log)).}
\end{lemma}
\begin{proof}
For the first statement it suffices  to show $J_S^q\subset \W_S\Omega^q_{\tA}(\log)$. 
Thus it remains to show \eqref{lem:ses-log-dRW0}. 
This "$\supset$" inclusion follows immediately from Lemma \ref{lem:dRW-quotient}.
As the right hand side in \eqref{lem:ses-log-dRW0}  is closed under $d$ (by the Leibniz rule and
the formula $dda=\dlog[-1]da$) it remains by Lemma \ref{lem:dRW-quotient} to show
\eq{lem:ses-log-dRW01}{\W_S(\fm_A)\cdot \W_S\Omega^q_A\subset 
\W_S(\fm_A)\cdot\W_S\Omega^q_{\tA}(\log)+ d(\W_S(\fm_A)\cdot\W_S\Omega^{q-1}_{\tA}(\log))=:I^q_S.}
We prove this by induction over $q$, the case $q=0$ being trivial. 
For $q=1$, we prove for later use the stronger statement 
\eq{lem:ses-log-dRW1}{\W_S(\fm_A)\cdot \W_S\Omega^1_A\subset 
\W_S(\fm_A)\cdot \W_S\Omega^1_{\tA}(\log)+ d\W^0_S(\fm_A)=:H^1_S\subset I^1_S,}
where $\W_S^0(\fm_A)=\ker(\W_S(\fm_A)\to \fm_A)$ is the kernel of the restriction map.
It suffices to show that any element of the form $V_m([a])dV_n([b])$ with $a\in \fm_A$, $b\in A$, 
and $n,m\in S$,  lies in $H^1_S$.
As $\fm_A\subset \tA$, we may assume $b\in A\setminus \fm_A=A^\times$. We have
\begin{align*}
V_m([a])dV_n([b])& = V_m([a]F_m dV_n([b]))\\
                 &= iV_m([a]dV_{n_1}([b]^{m_1}))+jV_m([a]V_{n_1}([b]^{m_1-1}d[b]))\\
                 &\phantom{=} + (c-1) V_m([a]V_{n_1}([b]^{m_1})\dlog[-1]),
\end{align*}
where $c={\rm gcd}(m,n)$, $m=cm_1$, $n=cn_1$,  and $i,j\in \Z$ satisfy $im+nj=c$.
Here the second equality follows from the second equality in \eqref{lem:dRW-quotient2} (with $a$ there replaced by $b$ here).
The third summand
\[(c-1)V_m([a]V_{n_1}([b]^{m_1})\dlog[-1])= (c-1)V_{mn_1}([a^{n_1}b^{m_1}])\dlog[-1]\]
clearly is contained in $H^1_S$, as is the second summand
\[jV_m([a]V_{n_1}([b]^{m_1-1}d[b]))= jV_{mn_1}([a^{n_1}b^{m_1}])\dlog[b].\]
If $n_1=1$, then we can take $j=1$ and $i=0$ and hence when considering the first summand we can assume $n_1\ge 2$.
Using the Leibniz rule and $V_md=mdV_m$ we can rewrite the first summand as
\[iV_m([a]dV_{n_1}([b]^{m_1}))= imdV_{mn_1}([a^{n_1}b^{m_1}])- iV_m(V_{n_1}([b]^{m_1}) d[a])\]
and it remains to show that $V_m(V_{n_1}([b]^{m_1}) d[a])$ ($n_1\ge 2$) is contained in $H^1_S$.
We can write 
\[[a]=[a+1]-[1] +V_r(\alpha),\]
for some $r>1$ in $S$ and $\alpha\in \W_{S/r}(\fm_A)$.
(Note that $\alpha$ lies in $\W_{S/r}(\fm_A)$ as $[a]+[1]-[a+1]$ maps to zero in $\W_S(A/\fm_A)$.)
Thus 
\begin{align}
V_m(V_{n_1}([b]^{m_1}) d[a]) &  = V_{mn_1}([b]^{m_1}[a]^{n_1-1} d[a])\label{lem:ses-log-dRW2}\\
                             & = V_{mn_1}([b]^{m_1}[a]^{n_1-1}[a+1])\dlog[a+1] + 
                                 V_{mn_1}([b]^{m_1}[a]^{n_1-1} dV_r(\alpha)).\notag
\end{align}
As $n_1\ge 2$, the first summand in this last equality is contained in $H^1_S$. Using 
\[dV_r(\alpha)= F_{mn_1}dV_{mn_1r}(\alpha)- (mn_1-1)\dlog[-1]V_r(\alpha),\]
see \cite[Definition 4.1(iv)]{Hesselholt-dRW}, we rewrite the second summand of the last equality in \eqref{lem:ses-log-dRW2}  as follows
\begin{multline*}
V_{mn_1}([b]^{m_1}[a]^{n_1-1} dV_r(\alpha))\\
= V_{mn_1}([b^{m_1}a^{n_1-1}])dV_{mn_1r}(\alpha)
- (mn_1-1)V_{mn_1r}\left([b^{m_1r}a^{(n_1-1)r}]\alpha\right)\dlog[-1],
\end{multline*}
which hence  lies in $H^1_S$ as well.
This completes the proof of claim \eqref{lem:ses-log-dRW1}.

Now we show \eqref{lem:ses-log-dRW01} for $q\ge 1$, i.e.,
for $a\in \W_S(\fm)$ and $\beta\in \W_S\Omega^q_A$
\[a\beta\in I^q_S.\]
We may assume that $\beta=\beta_1\beta_{q-1}$ with $\beta_1\in \W_S\Omega^1_A$ and 
$\beta_{q-1}=db_1\cdots db_{q-1}$, $b_i\in \W_S(A)$.
By \eqref{lem:ses-log-dRW1} we find $a_i, a_i'\in \W_S(\fm_A)$, 
and $\gamma_i\in \W_S\Omega^1_{\tA}(\log )$ such that
\[a\beta=a\beta_1\beta_{q-1}= \sum_i a_i'\gamma_i\beta_{q-1} + \sum_i d(a_i)\beta_{q-1}.\]
Thus by induction it remains to show  $dadb_1\cdots db_{q-1}\in I^q_S$, for $a\in \W_S(\fm_A)$.
By the Leibniz rule and the formula $ddx=\dlog[-1]dx$ we find an element 
$\gamma\in \W_S\Omega^{q-1}_A$ such that 
\[dadb_1\cdots db_{q-1}= d\left(a db_1\cdots db_{q-1}\right)+ a\gamma\dlog[-1],\]
which lies in $I^q_S$ by induction. This completes the proof of the lemma.
\end{proof}

\begin{lemma}\label{lem:DRW-val-tame-bc}
Let $(L,w)/(K,v)$ be a finite tame Galois extension of strictly henselian valuation fields.
Then we have the following equality of $\W_S(\sO_w)$-submodules of $\W_S\Omega^q_L$
\eq{lem:DRW-val-tame-bc0}{\W_S\Omega^q_{\sO_w}(\log )=
\W_S(\sO_w)\cdot \W_S\Omega^q_{\sO_v}(\log )+ d(\W_S(\sO_w))\cdot \W_S\Omega^{q-1}_{\sO_v}(\log ).}
\end{lemma}
\begin{proof}
First note that by étale base change $\W_S\Omega^q_K\to \W_S\Omega^q_L$ is injective. 
Hence both sides of the statement are contained in $\W_S\Omega^q_L$ and clearly the right-hand side is contained in the left-hand side.
The statement is trivially true for $q=0$. 

We consider the case $q=1$.  
Let $p$ be the characteristic exponent of the residue field $\sO_v/\fm_v$. 
As $(L,w)/(K,v)$ is a tame extension,  \cite[Corollary 6.2.14]{Gabber-Ramero} yields:
$L=K[b_1,\ldots, b_k]$ with $b_i^{r_i}=a_i\in K^\times$, $r_i\ge 1$ and $(p,r_i)=1$, and there is an isomorphism 
of the quotient of the valuation groups $\Gamma_L/\Gamma_K$ onto $\Z/r_1\Z\oplus\ldots\oplus \Z/r_k\Z$, 
which sends $w(b_i)$ to $(0,\ldots, 0,1, 0,\ldots, 0)$ with 1 sitting at position $i$. 
Up to replacing $b_i$  and $a_i$ by their inverse we can assume $b_i\in \sO_w$ and $a_i\in \cO_v$.
Set $\Lambda:= \prod_{j=1}^k \{0,1,\ldots, r_j\}$.
We can write any $z\in L$ as  $z=\sum_I x_I b^I$, where $I=(i_1,\ldots, i_k)\in \Lambda$,
$x_I\in K$, and $b^I=b_1^{i_1}\cdots b_k^{i_k}$. 
If $z\in \sO_w$, then $c_I:=x_I b^I\in \sO_w$, for all $I$.
Indeed, we have $w(c_I)= v(x_I)+ \sum_j i_jw(b_j)$ in $\Gamma_L$ which maps to  the residue class of $(i_1,\ldots, i_k)$ under 
$\Gamma_L\to \Gamma_L/\Gamma_K\cong \bigoplus_j \Z/r_j\Z$. Hence all $w(c_I)$, for $I$ as above, 
are pairwise different and therefore $0\le w(z)=\min_I\{w(c_I)\}$. 
Thus any element in $\sO_w$ can be written as a sum of elements of the form 
$xb^I$, with $x\in K$ and  $I\in \Lambda$ such that $xb^I\in \sO_w$. 
By induction over the length of $S$ we see that any element in $\W_S(\sO_w)$ can be written 
as a sum of elements $V_n([xb^I])$ with $xb^I$ as above.

Hence to prove the statement in the case $q=1$ it suffices to show
that the elements 
\begin{enumerate}[label=(\arabic*)]
    \item\label{lem:DRW-val-tame-bc1}
    $V_m([xb^I])dV_n([y b^J])$, for $x,y\in K^\times$,    $I,J\in \Lambda$ with $xb^I, y b^J\in \sO_w$, $n,m\in S$, and
    \item\label{lem:DRW-val-tame-bc2} 
    $\alpha\cdot \dlog[c]$, for $c\in L^\times$, $\alpha\in \W_S(\sO_w)$,
\end{enumerate}
are contained in 
\eq{lem:DRW-val-tame-bc3}{\W_S(\sO_w)\cdot \W_S\Omega^1_{\sO_v}(\log )+ d\W_S(\sO_w).}
We start by considering the elements of type \ref{lem:DRW-val-tame-bc1}. 
Set $\alpha=xb^I$, $\beta=yb^J$ and $e={\rm gcd}(m,n)$ and write $m=em_1$, $n=en_1$, and let $i,j\in \Z$ with $im+jn=e$.
We have 
\[V_m([\alpha])dV_n([\beta]) = V_m([\alpha]F_mdV_n([\beta])).\]
Using the second equality in \eqref{lem:dRW-quotient2} to rewrite $F_mdV_n([\beta])$ and 
using further the relations $V_md=mdV_m$, $F_n(d[a])=[a]^n\dlog[a]$, $V_n(a)b=V_n(aF_n(b))$ and the Leibniz rule we find that the above element is equal to the following element in  $\W_S\Omega^1_L$ 
\[imdV_{mn_1}([\gamma])-iV_{mn_1}([\gamma])\dlog [\alpha] 
    + jV_{mn_1}([\gamma])\dlog[\beta]+(e-1)V_{mn_1}([\gamma])\dlog[-1],\]
where $\gamma=\alpha^{n_1}\beta^{m_1}\in \sO_w$. As $r_j$ is invertible in $K$ we have moreover
\[\dlog[\alpha]=\dlog[x]+\sum_{j=1}^k i_j \dlog[b_j]= \dlog[x]+\sum_{j=1}^k \frac{i_j}{r_j} \dlog[a_j]
\in \W_S\Omega^1_{\sO_v}(\log )\]
and similarly with $\dlog[\beta]$. Thus the elements of type \ref{lem:DRW-val-tame-bc1}
are contained in \eqref{lem:DRW-val-tame-bc3}.
Let $\alpha\cdot\dlog[c]$ be an element of type \ref{lem:DRW-val-tame-bc2}.
Let $r=r_1\cdots r_k$, which is invertible in $K$. Then $w(c^r)$ is mapped to zero in $\Gamma_L/\Gamma_K=\bigoplus_{j=1}^k \Z/r_j\Z$
and thus  we find a unit $u\in \sO_w^\times$ and an element $x\in K^\times$ with $c^r=ux$. Thus
\eq{lem:DRW-val-tame-bc4}{\alpha\cdot \dlog[c]=\tfrac{1}{r}\alpha\cdot\dlog[u]+ \tfrac{1}{r}\alpha\cdot\dlog[x].}
Here the second summand clearly is contained in \eqref{lem:DRW-val-tame-bc3} and 
the first summand is contained in $\W_S\Omega^1_{\sO_w}$, hence is a sum of elements of type  \ref{lem:DRW-val-tame-bc1}
and therefore lies in \eqref{lem:DRW-val-tame-bc3} as well. This completes the proof in the case $q=1$.

For the case $q\ge 1$, we denote by $R^q_S$  the right hand side of \eqref{lem:DRW-val-tame-bc0}.
Note  $R^q_S\subset \W_S\Omega^q_{\sO_w}(\log )$ and  $dR^{q-1}_S\subset R^q_S$, as follows 
from the Leibniz rule and the formula $ddx= \dlog[-1]dx$.
From this latter formula and \eqref{para:log-dRW1} it also follows that any element in
$\W_S\Omega^{q-1}_{\sO_w}(\log )$, which is a product of elements of the form $db$,  with $b\in\W_S(\sO_w)$, 
and $\dlog[c]$, with $c\in L^\times$, can be written as
$d(\beta)+ \beta'\dlog[-1]$ with $\beta, \beta'\in \W_S\Omega^{q-2}_{\sO_w}(\log)$.
Thus any element in $\W_S\Omega^{q}_{\sO_w}(\log )$ is a sum of elements
\[\alpha d\beta\quad \text{and}\quad  \alpha \beta'\dlog[-1]\quad \text{with }\alpha\in \W_S\Omega^1_{\sO_w}(\log), 
\beta, \beta'\in \W_S\Omega^{q-2}_{\sO_w}(\log ).\]
We show by induction over $q$ that these elements are in $R^q_S$.
By induction we have $\alpha\beta'\in R_S^{q-1}$, hence $\alpha\beta'\dlog[-1]\in R^q_S$.
By the Leibniz rule $\alpha d\beta= -d(\alpha\beta)+ d(\alpha)\beta$ and as $d(\alpha\beta)\in dR^{q-1}_S$ by induction, we are
left to consider the element $d(\alpha)\beta$.
By the case $q=1$ we know that $\alpha$ is contained in \eqref{lem:DRW-val-tame-bc3} and hence is
a sum of elements  $b \gamma + db'$, with $b,b'\in \W_S(\sO_w)$ and $\gamma\in \W_S\Omega^1_{\sO_v}(\log )$.
Thus $d(\alpha)\beta$ is a sum of elements
\[d(\gamma) b\beta - \gamma d(b)\beta+ \dlog[-1] d(b')\beta.\]
By induction $b\beta\in R_S^{q-2}$ and $d(b)\beta$, $d(b')\beta\in R_S^{q-1}$. Hence  $d(\alpha)\beta$ is contained in  $R^q_S$.
\end{proof}

For $S=\{1\}$ we get a stronger statement:

\begin{lemma}\label{lem:DRW-val-tame-bcS1}
Let $(L,w)/(K,v)$ be as in Lemma \ref{lem:DRW-val-tame-bc} above.
Then 
\eq{lem:DRW-val-tame-bcS10}{\Omega^q_{\sO_w}(\log)=\sO_w\otimes_{\sO_v}\Omega^q_{\sO_v}(\log).}
\end{lemma}
\begin{proof}
As $\sO_w/\sO_v$ is flat, by \cite[VI, \S3, no. 6, Lemme 1]{BourbakiCA}, the map 
\[\sO_w\otimes_{\sO_v} \Omega^q_{\sO_v}(\log)\to \sO_w\otimes_{\sO_v} \Omega^q_K= L\otimes_K\Omega^q_K=\Omega^q_L\]
is injective, hence it suffices to show $\Omega^q_{\sO_w}(\log)\subset \sO_w\cdot \Omega^q_{\sO_v}(\log)$.
To this end we show  $db$, $\dlog\,c\in \sO_w\cdot \Omega^q_{\sO_v}(\log)$,  for all $b\in \sO_w$ and $c\in L^\times$.
Concerning the elements of type $db$ note that as explained in the beginning of the proof of Lemma \ref{lem:DRW-val-tame-bc} 
(and with the notation from there) it suffices to consider elements $b= x b_1^{i_1}\cdots b_k^{i_k}\in \sO_w$, with
$x\in K$ and $\sqrt[r_i]{b_i}=a_i\in K^\times$. In this case 
\[db= b\cdot\dlog\, b= b\cdot\dlog\, x + \sum_j \tfrac{i_j}{r_j}b\cdot\dlog \,a_j\]
clearly lies in $\sO_w\cdot \Omega^q_{\sO_v}(\log)$. For $\dlog\, c$ the argument around \eqref{lem:DRW-val-tame-bc4} applies.
\end{proof}

\begin{prop}\label{prop:DRW-tame-bc}
Let $\tA=A\times_K \sO_v$ be as in \ref{para:log-dRW} and assume additionally that $\sO_v$ is strictly henselian.
Let $A\to B$ be a finite \'etale map of henselian local rings associated to the  residue field extension $L/K$, 
which we assume to be Galois and tame with respect to $v$. Let $w$ be the extension of $v$ to $L$ and $\sO_w$ its valuation ring
and set $\tB=B\times_L \sO_w$.
Then we have the following equality of $\W_S(\tB)$-submodules of $\W_S\Omega^q_B$
\[\W_S\Omega^q_{\tB}(\log )=\W_S(\tB)\W_S\Omega^q_{\tA}(\log )+ d(\W_S(\tB))\W_S\Omega^{q-1}_{\tA}(\log ).\]
\end{prop}
\begin{proof}
Denote by $R^q_S$ the right-hand side of the equality in the statement. 
The image of $\W_S\Omega^q_{\tB}(\log )$ under the natural map $\W_S\Omega^q_B\to \W_S\Omega^q_L$
is by definition equal to $\W_S\Omega^q_{\sO_w}(\log )$, which  by
Lemma \ref{lem:DRW-val-tame-bc} and our assumption is also equal to the image of $R^q_S$ in $\W_S\Omega^q_{L}$. 
By Lemma \ref{lem:ses-log-dRW} it suffices to show that the group
\[\W_S(\fm_B)\cdot \W_S\Omega^q_B+ d(\W_S(\fm_B)\cdot \W_S\Omega^{q-1}_B)\]
is contained in $R^q_S$, where $\fm_B=\fm_A B$ is the maximal ideal of $B$. 
As $R^q_S$ is closed under the differential $d$
it suffices to show 
\eq{prop:DRW-tame-bc1}{\W_S(\fm_B)\cdot \W_S\Omega^q_B\subset R^q_S.}
As we have a surjection $\W_S(\fm_A\otimes_A B)\to \W_S(\fm_B)$, \cite[Corollary 9.4]{Borger} yields
\eq{prop:DRW-tame-bc2}{\W_S(\fm_B)=\W_S(\fm_A)\cdot \W_S(B).}
By \'etale base change \cite[Theorem C]{Hesselholt-dRW} 
\[ \W_S\Omega^q_B=\W_{S}(B)\otimes_{\W_S(A)}\W_S\Omega^q_A.\]
Thus
\begin{align*}
\W_S(\fm_B)\cdot \W_S\Omega^q_B & =\W_S(B)\cdot \W_S(\fm_A)\cdot\W_S\Omega^q_A\\
                           & \subset \W_S(B)\cdot \left(\W_S(\fm_A)\cdot \W_S\Omega^q_{\tA}(\log) +  d(\W_S(\fm_A)\cdot \W_S\Omega^{q-1}_{\tA}(\log))\right)    \\
                        &\subset  \W_S(\fm_B)\cdot \W_S\Omega^q_{\tA}(\log) +
                          \W_S(B)\cdot d(\W_S(\fm_A)\cdot \W_S\Omega^{q-1}_{\tA}(\log))\\
                &\subset R^q_S + \W_S(\fm_A) \cdot d(\W_S(B))\cdot  \W_S\Omega^{q-1}_{\tA}(\log),
\end{align*}
where we use \eqref{prop:DRW-tame-bc2} for the equality,  \eqref{lem:ses-log-dRW0} for the first inclusion, 
\eqref{prop:DRW-tame-bc2} for the second inclusion, and an application of the Leibniz rule together with \eqref{prop:DRW-tame-bc2}
for the last inclusion.
It remains to show
\eq{prop:DRW-tame-bc3}{ \W_S(\fm_A)\cdot \W_S\Omega^1_B\subset R^1_S.}
We prove this by induction over the cardinality of $S$.
For $S=\{1\}$ this amounts by \'etale base change 
to show that for $a\in \fm_A$, $a_0\in A$ and $b\in B$ the element
$\alpha:=abda_0$ is in $\fm_B\cdot \Omega^1_{\tA}(\log)$. But this is immediate if $a_0\in \fm_A\subset\tA$
and  if $a_0\in A^\times$, this follows from $\alpha=aba_0\dlog a_0$. 
Now we assume \eqref{prop:DRW-tame-bc3} is proven for all truncation sets of smaller cardinality 
than $S$. First observe that $V_n: \W_{S/n}\Omega^1_{\tB}(\log)\to \W_S\Omega^1_{\tB}(\log)$, $n\in S$,
induces a map
\[V_n: R^1_{S/n}\to R^1_S.\]
Indeed, any element in $R^1_{S/n}$ can be written as sum of elements
\begin{multline*}
b_0da_0  + b_1\dlog [a_1]+ db_2 \\
=  b_0 F_ndV_n(a_0)-(n-1)b_0a_0F_n(\dlog[-1])+ b_1 F_n(\dlog[a_1])  + db_2,    
\end{multline*}
where $b_i\in \W_{S/n}(\tB)$, $a_0\in \W_{S/n}(\tA)$,   $a_1\in A^\times$, 
see \cite[Definition 4.1, (iv), (v)]{Hesselholt-dRW} for the equality.
By \cite[Definition 4.1 (iii) and Lemma 4.3]{Hesselholt-dRW} this is  under $V_n$ 
mapped to 
\[V_n(b_0) dV_n(a_0)-(n-1)V_n(b_0a_0)\dlog[-1]+ V_n(b_1)\dlog[a_1]+ ndV_n(b_2),\]
which lies in $R^1_S$. 
We show the inclusion \eqref{prop:DRW-tame-bc3}. 
By \'etale base change, \eqref{lem:ses-log-dRW1}, and \eqref{prop:DRW-tame-bc2} it remains to show
\[\W_S(B)\cdot d\W^0_S(\fm_A)\subset R^1_S,\]
where $\W^0_S(\fm_A)=\Ker(\W_S(\fm_A)\to\fm_A)$. 
An element in the left hand side can be written as a sum of elements
\[bdV_n(a)= dV_n(F_n(b)a)- V_n(a F_ndb),\]
for $b\in \W_S(B)$, $a\in\W_{S/n}(\fm_A)$, and $1<n\in S$.
As $F_n(b)a\in \W_{S/n}(\fm_B)$, we find $dV_n(F_n(b)a)\in R^1_S$.
By induction $aF_ndb\in \W_{S/n}(\fm_A)\cdot \W_{S/n}\Omega^1_B\subset R^1_{S/n}$, and 
hence $V_n(a F_ndb)\in R^1_S$. This completes the proof of the proposition.
\end{proof}

For $S=\{1\}$ we get the stronger statement:

\begin{corollary}\label{cor:DR-tame-bc}
Let $\tB/\tA$ be as in Proposition \ref{prop:DRW-tame-bc} above.
Then 
\[\Omega^q_{\tB}(\log)=\tB\cdot\Omega^q_{\tA}(\log).\]
If, additionally $A$ is noetherian, then we also have
\[\Omega^q_{\tB}(\log)=\tB\otimes_{\tA} \Omega^q_{\tA}(\log).\]
\end{corollary}
\begin{proof}
In case $A$ is noetherian the ring $\tB$ is a flat $\tA$-module as follows from 
the argument in \cite[Proposition 11.8]{Huebner2021} 
and the fact that in the situation at hand we have $B=\tB\otimes_{\tA} A$ and 
$\tB$ is $\tA$-torsion free. In the same way as in the beginning of the 
proof of Lemma \ref{lem:DRW-val-tame-bcS1} we see that the natural maps 
$\tB\otimes_{\tA} \Omega^q_{\tA}(\log)\to \Omega^q_{\tB}(\log)$ is injective. 
Thus it remains to show the first statement in which we allow $A$ also to be non-noetherian.

By Proposition \ref{prop:DRW-tame-bc} it suffices to show 
$d(\tB)\subset \Omega^1_{\tB}(\log)\subset \tB\cdot \Omega^1_{\tA}(\log)$.
By Lemmas \ref{lem:DRW-val-tame-bcS1} and \ref{lem:ses-log-dRW} and since $\fm_B=\fm_A\cdot B$ it remains to show that the group
\[\fm_A\cdot \Omega^1_{\tB}(\log)+ d(\fm_A\cdot\tB)= \fm_A\Omega^1_{\tB}(\log)+\tB\cdot d\fm_A\]
is contained in $\tB\cdot\Omega^1_{\tA}(\log)$. This is clear for $\tB\cdot d\fm_A$. 
By \'etale base change for $\Omega^1_B$ any element in $\fm_A\Omega^1_{\tB}(\log)$ is a sum of elements of the form $abda_0$ with $a\in\fm_A$, $b\in B$,  $a_0\in A$. 
If $a_0\in \fm_A$, then $abda_0 \in \tB\cdot \Omega^1_{\tA}(\log)$; 
if $a_0\in A^\times$, then  $ab d a_0=ab a_0\dlog\, a_0\in \tB\cdot\Omega^1_{\tA}(\log)$.
Hence the statement.
\end{proof}

\begin{para}\label{para:DRW-trace}
Let $A\to B$ be a finite \'etale map between henselian local rings corresponding to a finite Galois extension $L/K$ 
of the residue fields. Assume $K$ is a strictly henselian  valuation field with valuation $v$ and valuation ring $\sO_v$
and that $L$ is tame  over $K$. Denote by $w$ the unique extension of $v$ to $L$ and by $\sO_w$ its valuation ring.
By \cite[VI, \S8, No. 6, Proposition 6]{BourbakiCA} the ring $\sO_w$ is also the integral closure of $\sO_v$ in $L$.
Hence the Galois group ${\rm Gal}(L/K)$ is equal to the decomposition group ${\rm Aut}(\sO_w/\sO_v)$.
Moreover, as the category of finite separable field extensions of $K$ is equivalent to the category of finite local \'etale $A$-algebras,
we can identify the $A$-algebra automorphisms of $B$ with ${\rm Gal}(L/K)$.
Hence  ${\rm Gal}(L/K)={\rm Aut}(B/A)={\rm Aut}(\tB/\tA)$ and   $\tB^{{\rm Gal}(L/K)}=\tA$.
We denote the ring map  $\W_S(\tB)\to \W_S(\tB)$ induced by $\sigma\in {\rm Gal}(L/K)$ 
again by $\sigma$. Hence the usual trace map 
\[\Tr: \W_S(B)\to \W_S(A), \quad b\to \sum_{\sigma \in {\rm Gal}(L/K)} \sigma(b)\]
restricts to 
\[\Tr: \W_S(\tB)\to \W_S(\tA).\]
Thus  the trace
\[\W_S\Omega^q_B\cong\W_S(B)\otimes_{\W_S(A)} \W_S\Omega^q_A\xrightarrow{\Tr\otimes \id} \W_S\Omega^q_A,\]
where the first isomorphism is \'etale base change,  restricts by 
Proposition \ref{prop:DRW-tame-bc} to a trace
\[\Tr: \W_S\Omega^q_{\tB}(\log )\to \W_S\Omega^q_{\tA}(\log ),\]
given by $\beta\mapsto \sum_{\sigma} \sigma(\beta)$.    
\end{para}

\begin{prop}\label{prop:WOmegat-via-val}
 Let $X\subset \tX$ be a quasi-compact open immersion between qcqs schemes.
 Then
 \[\W_S\Omega^{q,t}(X,\tX)=\Ker\left(\W_S\Omega^q(X)\to 
 \prod_{(x,v,\epsilon)\in \Spa(X,\tX)} \frac{\W_S\Omega^q_{k(x)}}{\W_S\Omega^q_{\sO_v}(\log)}\right),\]
 \[\W_S\Omega^{q,div}(X,\tX)=\Ker\left(\W_S\Omega^q(X)\to 
 \prod_{(\eta,v,\epsilon)\in \Spa(X,\tX)^{div}} \frac{\W_S\Omega^q_{k(\eta)}}{\W_S\Omega^q_{\sO_v}(\log)}\right),\]
 where the map into the products is induced by restriction along $\Spec k(x)\to U$, 
 $\sO_v\subset k(x)$ is the valuation ring defined by $v$, and $\Spa(X,\tX)^{div}$ denotes the subset of $\Spa(X,\tX)$ consisting of
 triples $(\eta, v,\epsilon)$ with $\eta$ a generic point of $X$ and $v$ a discrete valuation. 
\end{prop}
\begin{proof}
This "$\supset$" inclusion follows in both cases directly from the definition and the equality
\[\W_S\Omega^q_{\sO^h_{\tX, k(x), v}}(log)= \W_S\Omega^q_{\sO^h_{X,x}}\times_{\W_S\Omega^q_{k(x)}} \W_S\Omega^q_{\sO_v}(\log),\]
see Lemma \ref{lem:ses-log-dRW}. 
For the other inclusion, let $a\in \W_S\Omega^{q,t}(X,\tX)$ (resp. $\W_S\Omega^{q,div}(X,\tX)$)
and take $(x,v,\epsilon)\in \Spa(X,\tX)$ (resp. $\Spa(X,\tX)^{div}$). 
By Definition \ref{para:constr-tame-sheaves} and Lemma \ref{lem:tame-Gal-closure} there exists a finite tame Galois extension
$(L,w)/(K=k(x),v)$, such that $a_L$, the pullback of $a$ along $\Spec \sO_{X,L}^h\to X$,
lies in $\W_S\Omega^q_{\sO^h_{\tX, L,w}}(\log)$.
Let $L^{sh}$ and $K^{sh}$ be the strict henselizations, so that $L^{sh}/K^{sh}$ is
a finite Galois extension of degree $r$, which is prime to the residue characteristic of $K^{sh}$.
Thus $\sO^h_{X, K^{sh}}\to \sO^h_{X, L^{sh}}$ is a finite \'etale extension as in 
\ref{para:DRW-trace} and thus
\[a_{K^{sh}}= \Tr(\tfrac{1}{r} a_{L^{sh}})\in \W_S\Omega^q_{\sO^h_{\tX, K^{sh},v}}(\log).\]
By Lemma \ref{lem:locetbc-log}, the map 
$\W_S(\sO^h_{\tX, K,v})\to \W_S(\sO^h_{\tX, K^{sh},v})$ is faithfully flat and we have 
\[\W_S\Omega^q_{\sO^h_{\tX, K^{sh},v}}(\log)=
\W_S(\sO^h_{\tX, K^{sh},v})\otimes_{\W_S(\sO^h_{\tX, K,v})} \W_S\Omega^q_{\sO^h_{\tX, K,v}}(\log),\]
thus descent yields $a_K\in \W_S\Omega^q_{\sO^h_{\tX, K,v}}(\log)$.
As this is true for all $(x,v,\epsilon)\in \Spa(X,\tX)$,  the element $a$ is contained in 
the respective right hand sides of the statement. 
\end{proof}

\begin{prop}\label{prop:DRWOmegaplus}
Let $(A,\fm_A)$ be a local ring with residue field $K=A/\fm_A$. Let $v$ be a valuation on $K$ 
with valuation ring $\sO_v$ and set $\tA=A\times_K \sO_v$. 
Set $(U,\tU)=(\Spec(A),\Spec(\tA))$.
Then for $q\ge 0$ 
\[\W_S\Omega^{q,t}(U,\tU)=\W_S\Omega^q_{\tA}(\log ).\]
\end{prop}
\begin{proof}
This "$\subset$" inclusion holds by Proposition \ref{prop:WOmegat-via-val} and 
Lemma \ref{lem:ses-log-dRW}.  For the other inclusion, 
let $(x,v,\epsilon)\in \Spa(U,\tU)$. It corresponds to ring maps $A\to k(x)$ and $\tA\to \sO_v$, which induce a morphism $\W_S\Omega^q_{\tA}(\log)\to \W_S\Omega^q_{\sO_v}(\log )$. 
Hence  this inclusion 
$\W_S\Omega^q_{\tA}(\log)\subset \W_S\Omega^{q,t}(U,\tU)$ follows from 
Proposition \ref{prop:WOmegat-via-val}.
\end{proof}

\begin{example}\label{exa:DRWdim1}
Let $\tC$ be an integral, regular scheme of dimension 1. Let
$\sigma: C\inj  \tC$ be a dense open subscheme 
and denote by 
$j: (C,\tC)_t\to \tC_{\Zar}$ the natural morphism of sites.
 \begin{enumerate}[label=(\arabic*)]
     \item\label{exa:DRWdim1.1} Assume  the image of
     $\tC\to \Spec \Z$ is a single point. Then  there is a canonical exact sequence 
     \eq{exa:DRWdim1.2}{0\to \W_S\Omega^q_{\tC}\to j_*\W_S\Omega^{q,t}\to 
     \bigoplus_{x\in \tC\setminus C} i_{x*}\W_S\Omega^{q-1}_{k(x)}\to 0,}
     where $i_x: \Spec k(x)\to \tC$ is the closed immersion defined by $x$. 

     Indeed, we have an exact sequence
     \[0\to \W_S\Omega^q_{\tC}\to \sigma_*\W_S\Omega^{q}_C\to 
     \bigoplus_{x\in \tC\setminus C} 
     i_{x*}\frac{\W_S\Omega^{q}_{k(C)}}{\W_S\Omega^{q}_{\tC, x}}\to 0.\]
     As $\tC$ is defined over a field the big de Rham--Witt decomposes as a product of $p$-typical de Rham--Witt - or classical de Rham complexes and hence this sequence follows from the structure results in \cite{IllusiedRW} together with the fact that locally $\tC$ is a limit of schemes, which are smooth over the prime field,
     by Popescu's Theorem. By Proposition \ref{prop:WOmegat-via-val} there is a similar sequence for $j_*\W_S\Omega^{q,t}$, where the denominator on the right hand side is replaced by 
     $\W_S\Omega^{q}_{\tC, x}(\log)$.  Here we note that 
     it suffices to take the product in the statement of this proposition over all $(\eta,v,\epsilon)\in \Spa(C,\tC)$ with $\eta\in C$ the generic point and such that $\epsilon: \Spec \sO_v\to \tC$, has center in a (closed) point of $\tC\setminus C$, say $x$. 
     But in this case we have inclusions 
     $\sO_{\tC,x}\subset \sO_v\subset k(C)$ and as $\sO_{\tC,x}$ is a
     discrete valuation ring of rank one, it is a maximal subring of $k(C)$ and hence   $\sO_v=\sO_{\tC, x}$. Putting the two sequences together an application of the snake lemma gives the exact sequence
     \eqref{exa:DRWdim1.2} with the right term replaced by 
     $\oplus_{x} \W_S\Omega^q_{\sO_{\tC,x}(\log)}/ \W_S\Omega^q_{\tC,x}$.
     The residue map, e.g. \cite[Definition 2.11]{Rul}, gives an isomorphism 
     \[\Res_{S}^q: \frac{\W_S\Omega^q_{\sO_{\tC,x}(\log)}}{\W_S\Omega^q_{\tC,x}}\xrightarrow{\simeq} \W_S\Omega^{q-1}_{k(x)}.\]
     For this one has to identify the completion  of $\sO_{\tC,x}$
     with $k(x)[[t]]$, which requires the choice of a local parameter and a coefficient field. It is however not hard to show that since we only consider log-poles the residue map is in fact independent of 
     these choices. This yields the exact sequence \eqref{exa:DRWdim1.2}.
\item Assume  $\tC\to \Spec \Z$ is finite, i.e., $\tC$ is the spectrum of
the ring of integers in a number field $K$. Then for $q\ge 1$.
\[j_*\W_S\Omega^{q,t}= \sigma_*\W_S\Omega^q_C.\]
Indeed, in this case $\W_S\Omega^q_{K}=\W_S(K)\otimes_{\W_S(\Q)}\W_S\Omega^q_{\Q}=0$, for $q\ge 1$ and hence 
the above equality follows directly from Proposition 
\ref{prop:WOmegat-via-val}. If $C$ is \'etale over $\Z$, this vanishes
for $q\ge 2$, by \cite[Theorem 6.1]{Hesselholt-dRW} and \'etale base change. 
 \end{enumerate}   
\end{example}

\begin{prop}\label{prop:Omegaplus}
Let $T$ be a normal noetherian scheme,  $\tX\to T$ a smooth morphism, and $X\subset \tX$ a dense open, such that 
$\tX\setminus X$  is the support of a simple normal crossing divisor $D$ over $T$ (i.e., all intersections $D_{i_1}\cap\ldots\cap D_{i_r}$ 
    of the irreducible components of $D$ are smooth over $T$).  Let $j:(X,\tX)_t\to \tX_\Zar$ be the natural morphism of sites from above. Then 
    \[j_*\Omega^{q,t}_{/T}=j_*\Omega^{q, div}_{/T}=\Omega^{q}_{\tX/T}(\log D).\]
\end{prop}
\begin{proof}
We may assume $\tX$ is integral. Clearly, we have the inclusions
\[\Omega^{q}_{\tX/T}(\log D)\subset j_*\Omega^{q,t}_{/T} \subset j_*\Omega^{q, div}_{/T}\]
and the equality 
\[\Omega^{q}_{\tX/T}(\log D)= 
\Ker\left(\Omega^q_{X/T}\to \bigoplus_{x\in D^{(0)}} \Omega^q_{\eta/T}/\Omega^q_{\tX, x/T}(\log D)\right),\]
where $\eta\in \tX$ is the generic point. Thus it remains to show:
Let $R$ be a ring and $(K,v)$ a discrete valuation field, such that $\sO_v$  is an $R$-algebra. 
Let $\alpha\in \Omega^q_{K/R}$ and assume there exists a finite tame Galois extension $(L,w)/(K,v)$ such that 
$\alpha_L\in \Omega^q_{\sO_w}(\log)$. Then $\alpha\in \Omega^q_{\sO_v/R}(\log)$. 

To this end, we may assume by \'etale descent that $(K,v)$ and $(L,w)$ are strictly henselian.
We have a   surjection $\Omega^q_{\sO_w}(\log )\surj \Omega^q_{\sO_w/R}(\log )$ and hence 
by Lemma \ref{lem:DRW-val-tame-bcS1}
\[\Omega^q_{\sO_w/R}(\log)= \sO_w\cdot \Omega^q_{\sO_v/R}(\log).\]
Hence the trace from \ref{para:DRW-trace} for $S=\{1\}$ also induces a trace for relative differentials
\[\Tr: \Omega^q_{\sO_w/R}(\log)\to \Omega^q_{\sO_v/R}(\log ).\]
Denoting by $e=[L:K]$ the degree, which is prime to the residue characteristic of $\sO_v$, we obtain 
$\alpha= \frac{1}{e}\Tr(\alpha_L)\in \Omega^q_{\sO_v/R}(\log )$. This completes the proof.
\end{proof}

\section{Tame de Rham--Witt cohomology in positive characteristic}\label{sec:10}

Throughout this section $k$ is a perfect field of  positive characteristic $p$,
$X$ is a smooth $k$-scheme of finite type and $X\inj \tX$ is a dense open immersion into 
a separated and finite type $k$-scheme $\tX$. By Definition \ref{def:DRWOmegaplus} 
we can consider the tame ($p$-typical) de Rham--Witt complex 
$W_n\Omega^{\bullet, t}$  and the divisorial version $W_n\Omega^{\bullet, {div}}$ on $(X,\tX)_t$.

We set $\sX=(X,\tX)$ and denote
by $j: (X,\tX)_t\to \tX_{\Zar}$ the morphism of sites induced by the functor 
$(\tV\subset \tX) \mapsto (\tV\cap X, \tV)$.

\begin{prop}\label{lem:WnOmegatame}
Assume $\tX$ is smooth and the complement $\tX\setminus X$ is the support of a simple normal 
crossing divisor $D$. Then
\[j_*W_n\Omega^{q,t}=j_*W_n\Omega^{q,div}=W_n\Omega^q_{\tX}(\log D),\]
where the right hand side denotes  $q$-forms of the logarithmic  de Rham--Witt complex, 
associated to the 
smooth log scheme $(\tX, j_*\sO_X^\times\cap \sO_{\tX})$, see \cite{hyodoKato}, \cite{matsuue}.
\end{prop}
\begin{proof}
For definiteness, we set 
\[W_n\Omega^q_{\tX}(\log D):= W_n\Lambda^q_{(\tX, \sM)/(\Spec k, {\rm triv})}\]
with the notation from  \cite[Proposition-Definition 3.10]{matsuue},
where $\sM= j_*\sO_X^\times\cap \sO_{\tX}$ is the log structure defined by $D$ and we equip the base $\Spec k$ 
with the trivial log structure. If $U$ is a smooth $k$-scheme and $x\in U$ is a 1-codimensional point such that its closure 
$D_x$ in $U$ is smooth, and $L$ is the fraction field of $\sO_{U,x}$ 
(resp. the fraction field of its strict henselization at a geometric point $\bar{x}$ over $x$) 
and $w$ is the discrete valuation defined by $x$ on $L$,  
then it follows from \'etale base change and \cite[Corollary 4.4]{matsuue}
that the Zariski stalk at $x$ (resp. the \'etale stalk at $\bar{x}$) of $W_n\Omega^q_{U}(\log D_x)$ is equal to
$W_n\Omega^q_{\sO_w}(\log )$. Furthermore, it follows from the explanations in \cite[9.]{matsuue} and \cite[I, Corollaire 3.9]{IllusiedRW}
that $W_n\Omega^q_{\tX}(\log D)$ is a successive extension of locally free $\sO_{\tX}$-modules and hence is 
Cohen--Macaulay. Thus $W_n\Omega^q_{\tX}(\log D)=j_*W_n\Omega^{q,div}$, by Proposition \ref{prop:WOmegat-via-val}
with $S=\{1, p,\ldots, p^{n-1}\}$. Thus also  $j_*W_n\Omega^{q,t}\subset W_n\Omega^q_{\tX}(\log D)$. 
On the other hand it follows from the local description of $W_n\Omega^q_{\tX}(\log D)$, given by \cite[Corollary 4.4]{matsuue}
and \'etale base change, that the pullback of a local section 
$\alpha\in W_n\Omega^q_{\tX}(\log D)$ along 
$(\Spec L, \Spec \sO_w)\to (X,\tX)$, with $(L,w)$ a valuation field,  
lies in $W_n\Omega^q_{\sO_w}(\log)$.
Hence  the inclusion $W_n\Omega^q_{\tX}(\log D)\subset j_*W_n\Omega^{q,t}$.  
\end{proof}

\begin{para}\label{para:res}
Consider the following condition:
\begin{enumerate}[label={$({\rm princ})_{d}$}]
    \item\label{res2} Assume $(\tX, D)$  is a normal crossing pair of dimension $d$, 
    i.e., $\tX$ is a smooth integral $k$-scheme of dimension $d$ and $D$ is a divisor on $\tX$ which has simple normal crossings. 
    Assume further that $Z\subset \tX$ is a closed (possibly non-reduced) subscheme of codimension 
    $\ge 2$ which is contained in $D$.
    Then there exists a finite sequence of morphisms
    \[\pi:\tY=\tX_n\xrightarrow{\pi_n} \tX_{n-1}\to\ldots\to \tX_1\xrightarrow{\pi_1} \tX_0=\tX,\]
    such that 
    \begin{enumerate}
         \item each $(\tX_j, D_j)$ is a normal crossing pair, where $D_0:=D$ and 
         $D_{j}=\pi^{-1}_j(D_{j-1})_{\rm red}$, $j\ge 1$,
        \item each $\pi_j$, $1\le j\le n$, is the blow-up in a smooth closed subscheme $C_{j-1}\subset X_{j-1}$, which has normal crossings with $D_{j-1}$,
        \item each $C_j$, $0\le j\le n-1$ is contained in $Z_j$, where 
         $Z_0:=Z$ and $Z_j$ is the strict transform of $Z_{j-1}$ under $\pi_j$, 
        \item $\pi^{-1}Z$ is an Cartier divisor on $\tY$.
    \end{enumerate}
    Recall that given a normal crossing pair $(\tX,D)$,  we say that a closed subscheme $C\subset D$ has {\em normal crossings with $D$},
    if for every point $x\in C$ we find a regular parameter sequence $t_1,\ldots, t_r\in \sO_{X,x}$ such that 
    the ideal of $C$ at $x$ is given by $(t_1,\ldots, t_b)$ and $D$ is at $x$ given by ${\rm div}(t_a\cdots t_c)$
    with $1\le a\le b\le c\le r$. 
\end{enumerate}
It is well-known that \ref{res2} holds for $0\le d\le 2$, see e.g.
\cite[\href{https://stacks.math.columbia.edu/tag/0AHH}{Tag 0AHH}]{stacks-project};
in characteristic zero it holds by Hironaka. Notice that \ref{res2} is trivially satisfied
if $D=\emptyset$.
\end{para}

\begin{thm}\label{thm:tameDRWcoh}
Assume $\tX$ is smooth and the complement $\tX\setminus X$ is the support of a simple normal 
crossing divisor $D$. Assume \ref{res2} holds for $d=\dim X$.
    Then 
    \[R\Gamma((X,\tX)_t, W_n\Omega^{q,t})= R\Gamma(\tX_{\Zar}, W_n\Omega^{q}_{\tX}(\log D)).\]
    Moreover, setting 
    \[W\Omega^{\bullet,t}:=R\lim_n W_n\Omega^{\bullet,t},\] 
    we get
    \[R\Gamma((X,\tX)_t, W\Omega^{\bullet,t})= R\Gamma_{\text{log-crys}}((\tX,D)/W(k)),\]
    where the right hand is logarithmic crystalline cohomology, \cite{hyodoKato}:
       \[
    R\Gamma_{\text{log-crys}}((\tX,D)/W(k))=R\lim_n R\Gamma_{\text{log-crys}}((\tX,D)/W_n(k)).\]
    
\end{thm}
\begin{proof}
For the first statement, observe that by Proposition \ref{lem:WnOmegatame} we have a natural map
\[W_n\Omega^{q}_{\tX}(\log D)=j_*W_n\Omega^{q,t}\to Rj_*W_n\Omega^{q,t}.\]
It remains to show that this yields an isomorphism after applying $H^i(\tX_{\Zar}, -)$. 
By Theorem \ref{thm:tame-vs-etale}  we have 
 \[H^i((X,\tX)_t, W_n\Omega^{q,t})=
\varinjlim_{\sY\to \sX} H^i(\tY_{\et}, j^{\sY}_*W_n\Omega^{q,t}),\]
where the direct limit is over all modifications $\sY=(X,\tY)\to \sX$. 
Under \ref{res2} modifications as in \ref{res2} are cofinal in the 
cofiltered category of all modifications to $\sX$,
see, e.g., the argument in the proof of 
\cite[\href{https://stacks.math.columbia.edu/tag/0AHI}{Tag 0AHI}]{stacks-project}.
Using further  that \'etale- and Zariski cohomology of $W_n\Omega^q_{\tX}(\log D)$ coincide, we obtain 
 \[H^i((X,\tX)_t, W_n\Omega^{q,t})=
\varinjlim_{\sY\to \sX} H^i(\tY_{\Zar}, W_n\Omega^{q}_{\tY}(\log D')),\]
where the limit is over modifications induced by morphisms $\pi:\tY\to \tX$ as in  \ref{res2} and 
and where  $D':=\pi^{-1}(D)_{\rm red}$. The pullback
\[\pi^*: H^i(\tX_{\Zar}, W_n\Omega^{q}_{\tX}(\log D ))\to 
H^i(\tX_{\Zar}, W_n\Omega^q_{\tY}(\log D'))\]
is an isomorphism by \cite[Theorem 5.2]{shujilog} and the fact that with the notation from  
{\em loc. cit.} the Nisnevich sheaf given by 
$U\mapsto \underline{\omega}^{\CI}W_n\Omega^q(U,D_{|U})$ is equal to
$W_n\Omega^q_{\tX}(\log D)$. 
For $q=0$, this follows from \cite[Theorem 4.15(4) and Theorem 7.20]{RulSaito}, 
for $q\ge 1$ this holds by \cite[Theorem 5.4 and Remark 5.5(3)]{RenRuelling} 
or the proof of  \cite[Theorem 4.4]{mericicrys}. This yields the first statement.

Considering the total complex of the double complex
$(\ldots \to R\Gamma((X,\tX)_t, W_n\Omega^{q,t})\xrightarrow{d} 
R\Gamma((X,\tX)_t, W_n\Omega^{q+1,t})\to \ldots)$ and applying $R\lim$
yields the second statement.
\end{proof}

{
\begin{remark}\label{rmk:limWnOmegat}
Let $(X,\tX)$ be as at the beginning of this section with  $X$ integral of dimension $d$.
Assume \ref{res2} and assume resolution of singularities hold in dimension $d$, in the sense that 
for any finite type and separated $k$-scheme $Y$ of dimension $d$ and every reduced closed subset $Z\subset Y$
such that $Y\setminus Z$ is smooth, there exists a proper morphism $f: Y_1\to Y$ which is an isomorphism 
over $Y\setminus Z$, such that $Y_1$ is smooth over $k$ and $f^{-1}(Z)_{\red}$ has simple normal crossings. 
Then we have 
\eq{thm:tameDRWcoh1}{ R\lim_n W_n\Omega^{q,t}= \lim_n W_n\Omega^{q,t},}
i.e., $R^i\lim_n W_n\Omega^{q,t}=0$, for all $i\ge 1$. 

Indeed  by resolution of singularities every object of $(X,\tX)_t$ has a covering
consisting of pairs $(V,\tV)$ with $\tV$ smooth and affine, such that 
$\tV\setminus V$ has simple normal crossing support. For such pairs we have 
\[H^i((V,\tV)_t, W_n\Omega^{q,t})=0, \quad \text{for all }i>0,\]
by the first statement of Theorem \ref{thm:tameDRWcoh}, which uses \ref{res2}. Moreover, the restriction maps
$W_{n+1}\Omega^{q,t}(V,\tV)\to W_{n}\Omega^{q,t}(V,\tV)$ are surjective by Proposition \ref{lem:WnOmegatame}, 
whence  \eqref{thm:tameDRWcoh1} holds by \cite[\href{https://stacks.math.columbia.edu/tag/0BKY}{Tag 0BKY}]{stacks-project}. 
\end{remark}
}

\begin{remark}
 We remark, that if $k=\C$, then \ref{res2} holds for any $d$ and the same proof  for Theorme \ref{thm:tameDRWcoh}
 together with \cite{DeligneHodgeII} shows in case $\tX$ is proper
 \[R\Gamma((X,\tX)_t, \Omega^{\bullet,t}_{/\C})\cong R\Gamma_{\rm Betti}(X(\C), \C),\]
 where we use Proposition \ref{prop:Omegaplus} instead of Proposition \ref{lem:WnOmegatame}.
\end{remark}

\begin{definition}\label{defn:tame-syntomic}
Following \cite[after Definition 2.1]{LangerZink-Displays}\footnote{See also \cite[Proposition 4.4]{LangerZink-Displays} for the relation with other definitions of the Nygaard complex in the non-tame setup.} 
we define the $r$th {\em tame Nygaard complex} of level $n$ on $(X,\tX)_t$ ($r\ge 0$) by
\[\sN^rW_n\Omega^{\bullet,t}: 0\to W_{n-1}\sO^t\xrightarrow{d}\ldots \xrightarrow{d} W_{n-1}\Omega^{r-1,t}\xrightarrow{dV}
W_n\Omega^{r,t}\xrightarrow{d}W_n\Omega^{r+1,t}\to\ldots \]
We obtain a pro-complex $\sN^rW_{\sbullet}\Omega^{\bullet,t}$ in which the transition maps 
are induced by restriction. We set 
\[\sN^rW\Omega^{\bullet, t}=R\lim\sN^rW_{n}\Omega^{\bullet,t}.\]

We have two morphism of pro-complexes
\[\iota_r, F_r: \sN^rW_{\sbullet}\Omega^{\bullet}\to W_{\sbullet}\Omega^{\bullet, t},\]
which in degree $q$ are given by 
\[\iota_r^q=\begin{cases}
    p^{r-q-1}V & \text{if }q\le r-1\\  {\rm id} &\text{if } q\ge r,
\end{cases} \quad
F_r^q=\begin{cases}
    {\rm id} & \text{if } q\le r-1\\ p^{q-r} F & \text{if } q\ge r.
\end{cases}\]
It is direct to check that these define morphism of complexes, cf. \cite[(5)]{LangerZink-Displays}. 
We define the $r$th tame syntomic pro-complex on $(X\,\tX)_t$ as 
\[\Z/p^{\sbullet}(r)^t= \Cone(F_r-\iota_r: 
\sN^r W_{\sbullet}\Omega^{\bullet, t}\to W_{\sbullet}\Omega^{\bullet,t})[-1],\]
and the $r$th tame syntomic complex by 
\[\Z_p(r)^t=R\lim \Z/p^{\sbullet}(r)^t.\]
\end{definition}

\begin{definition}\label{defn:WnOmegalog}
Let $W_n\Omega^r_{X,\log}$\footnote{Also denoted by $\nu_n(r)$ or $\nu_n^r$ in the literature.} 
be the sheaf on $X_\et$ of logarithmic de Rham--Witt differentials of degree $r$, 
e.g., \cite[I, 5.7]{IllusiedRW}.  We define a presheaf $W_n\Omega^{r,t}_{\log}$  on 
$(X,\tX)_t$ by 
\[W_n\Omega^{r,t}_{\log}(V,\tV):= \Gamma(V,W_n\Omega^r_{X,\log}).\]
It follows directly from the definition, that $W_n\Omega^{r,t}_{\log}$ coincides with 
$F_{\beta}$, where $F=W_n\Omega^r_{X,\log}$ and  
$\beta=\{F_w\subset F(\cO^h_{X,L})\}_{(L,w,\epsilon\in \Vals_{\cX})}$ with $F_w=F(\cO^h_{X,L})$
in the notation of \ref{para:constr-tame-sheaves}. 
Hence $W_n\Omega^{r,t}_{\log}$ is a sheaf on $(X,\tX)_t$. 
\end{definition}

\begin{thm}\label{thm:syn-vs-log}
We have  an isomorphism of pro-complexes $(X,\tX)_t$ 
\[\Z/p^{\sbullet}(r)^t= W_{\sbullet}\Omega^{r,t}_{\log}[-r].\]
\end{thm}
\begin{proof}
This is the tame version of \cite[Proposition 8.4]{BMS2}.
Indeed, $\iota_r^q$ (resp. $F^q_r$) is nilpotent in degrees $\le r_1$ (resp. $\ge r$) and hence
$\iota_r-F_r$ is an isomorphism of pro-complexes in degrees $\neq r$. 
It remains to show that 
\[{\rm id}-F: W_{\sbullet}\Omega^{r,t}\to W_{\sbullet} \Omega^{r,t}\]
is surjective on $(X,\tX)$ with kernel $W\Omega^{q,t}_{\log}$. 
By Proposition \ref{prop:tamely-acyclic}, Remark \ref{rmk;Deligne}, and Proposition \ref{prop:DRWOmegaplus} 
it suffices to show that the sequence
\eq{thm:syn-vs-log1}{0\to W_{\sbullet}\Omega^{r}_{A, \log}\to W_{\sbullet}\Omega^{r}_{\tA}(\log)\xrightarrow{{\rm id} -F} W_{\sbullet}\Omega^r_{\tA}(\log)\to 0}
is exact for henselian local rings $(A,\fm)$ with a valuation $v$ on the residue field $K=A/\fm$, 
such that $A$ is ind-\'etale over $X$, $(K,v)$ is tamely closed, and $\tA=A\times_K \sO_v$. 
As \eqref{thm:syn-vs-log1} is a subsequence  of the sequence
\[0\to W_{\sbullet}\Omega^{r}_{A, \log}\to W_{\sbullet}\Omega^{r}_{A}
\xrightarrow{{\rm id} -F} W_{\sbullet}\Omega^r_{A},\]
which is exact by \cite[I, Théorème 5.7.2]{IllusiedRW}, 
we get that the sequence \eqref{thm:syn-vs-log1} is
left exact as well. As $\tA$ is strictly henselian, the map
$W_{\sbullet}\Omega^{r}_{\tA}\xrightarrow{{\rm id} -F} W_{\sbullet}\Omega^r_{\tA}$ is surjective, 
by {\em loc. cit.} Thus the formula $Fd\log[x]=\dlog[x]$ yields
the right exactness of \eqref{thm:syn-vs-log1}.
\end{proof}

\begin{corollary}\label{cor:tame-DRW-cycle-map}
There is a natural cycle map
\eq{cor:tame-DRW-cycle-map1}{{\rm CH}^r(X, 2r-i)\to H^i((X,\tX)_t, \Z_p(r)^t)\to 
H^i((X,\tX)_t, W\Omega^{\bullet,t}}),
where the left hand side denotes Bloch's higher Chow groups.
This map is functorial with respect to pullback along $k$-morphisms $(Y,\tY)\to (X,\tX)$, with
$\tY$ a finite type separated $k$-scheme and $Y\subset \tY$ a dense open which is smooth.
This map vanishes for $r>i$. 

Moreover, if $\tX$ is smooth and proper over $k$ and $i=2r$, then this cycle map fits into a
commutative diagram
\[\xymatrix{
\CH^r(\tX)\ar[r]\ar[d] & H^r_{\rm crys}(\tX/W(k))\ar[d]\\
\CH^r(X)\ar[r] & H^r((X,\tX)_t, W\Omega^{\bullet,t}),
}\]
where the top horizontal map is the classical crystalline cycle map,
the left vertical map is restriction along the open immersion $X\inj \tX$, and 
the right vertical map is induced from 
the identification $H^r_{\rm crys}(X/W(k))=H^r((\tX,\tX)_t, W\Omega^{\bullet,t})$ stemming from
Theorem \ref{thm:tameDRWcoh} and pullback along $(\tX,\tX)\to (X,\tX)$.    
\end{corollary}
\begin{proof}
We denote by $\Z(r)_X$ Bloch's complex of Zariski sheaves $U\mapsto z^r(U,*)[-2r]$ 
\cite{Bloch86}, which is an avatar of the motivic complex of weight $r$. 
Let $i:Y=\tX\setminus X\inj \tX$ be the closed immersion, 
and denote by $\sigma: X\inj \tX$ the open immersion and by $c={\rm codim(Y,\tX)}$ the codimension. 
Consider the following composition of maps between pro-complexes
\[\frac{\Z(r)_{\tX}}{i_*\Z(r)_{Y}}\to \sigma_* \Z(r)_X\to \sigma_*\Z/p^{\sbullet}(r)_X\cong 
\sigma_*W_{\sbullet}\Omega^r_{X,\log}[-r]=j_*W_{\sbullet}\Omega^{r,t}_{\log}[-r]\to
Rj_*\Z/p^{\sbullet}(r)^t,\]
where the isomorphism in the middle is \cite[Theorem 8.3]{GL}, the equality holds by definition of 
$W_n\Omega^{q,t}_{\log}$, and the last map is induced by Theorem \ref{thm:syn-vs-log}.
Applying $R\lim R\Gamma(\tX, -)$ and using 
$R\Gamma(\tX_{\Zar}, \frac{\Z(r)_{\tX}}{i_*\Z(r)_{Y}}) \cong \Gamma(X,\Z(r)_X)$,
which holds by  \cite[Theorems (3.3) and (3.4)]{Bloch86} together with \cite{Bloch94},
we get the first map in \eqref{cor:tame-DRW-cycle-map1}. 
The second map is the composition
\[H^i((X,\tX)_t, \Z_p(r)^t)\to H^i((X,\tX), R\lim \sN^r W_n\Omega^{\bullet,t})\
\xrightarrow{\iota_r} H^i((X,\tX), R\lim W_n\Omega^{\bullet,t}).\]
This completes the construction of the cycle map.
The other statements follow immediately from this construction.
\end{proof}

\section{Tame sheaves associated to reciprocity sheaves}\label{sec:11}

\begin{para}\label{para:tameRSC}
We explain how to apply the construction \ref{para:constr-tame-sheaves} to reciprocity sheaves.
Let $k$ be a perfect field of characteristic $p\ge 0$ 
and let $X\inj \tX$ be an open immersion between separated finite type
$k$-schemes such that $\tX\setminus X$ is the support of an effective Cartier divisor 
and $X$ is smooth over $k$. Set $\cX=(X,\tX)$.

Let $F$ be a reciprocity sheaf in the sense of \cite{KSY-RecII}, see also \cite{KSY}. 
In particular it is defined on $\Sm_k$ the category of smooth separated $k$-schemes. 
We denote by $\underline{\omega}^{\CI}F$ the maximal cube-invariant and semipure modulus sheaf with 
transfers, which has $M$-reciprocity and extends $F$, see \cite[(1.0.4)]{shujilog}.
In the following we assume that $F$ additionally defines a sheaf on $X_{\et}$.  

Let $(L,w,\epsilon)\in \Vals_{\cX}$.  We say a discrete valuation $v$ on $L$ is of {\em geometric type over $k$}, if there is some smooth integral $k$-scheme $U$ with function field 
$L$  and some point $y\in U^{(1)}$, such that $\sO_v= \sO_{U,y}$. We say such a 
discrete valuation $v$  is a {\em generalization of $w$}, if $U$ and $y$ can be chosen in such way 
that $\Spec L\to X$ extends to $U\to \tX$ and the image of the closed point 
of $\Spec \sO_w$ in $\tX$ is contained in the closure of the image of $y$ in $\tX$. We denote by 
$\Phi_w$ the set of all discrete valuations of geometric type over $k$ on $L$ generalizing $w$. 
For $v\in \Phi_{w}$ we denote by $L^h_v$ and $\sO_v^h$ the henselizations. 
By taking colimits over geometric models we can define
\[F(L^h_v)\quad\text{and}\quad \underline{\omega}^{\CI}F(\sO_v^h, \fm_v^{-1}),\]
where $\fm_v$ denotes the maximal ideal of $\sO_v^h$.
In particular $\underline{\omega}^{\CI}F(\sO_v^h, \fm_v^{-1})$ is a colimit of groups of the form
$\underline{\omega}^{\CI}F(\tV,D)$, where $\tV$ is smooth and affine  with a morphism $\tV\to \tX$ and $D$ is a smooth divisor on $\tV$.

Define $\log F(w)$ to be the subgroup of $F(L)$ given by 
\eq{para:example-tame-sheaves4.1}{
\log F(w)= \Ker\left(F(L)\to \prod_{v\in \Phi_w} \frac{F(L^h_v)}{\underline{\omega}^{\CI}F(\sO_v^h,\fm_v^{-1})}\right)}
and set
\[\log F_w:= F(\sO^h_{X,L})\times_{F(L)} \log F(w).\]
We claim, that 
\[\beta^{\rm rec}_{\log}:=\{\log F_w\subset F(\sO^h_{X,L})\}_{(L,w,\epsilon)\in \Vals_{\cX}}\]
satisfies the assumption \ref{beta1} of \ref{para:constr-tame-sheaves}.
To this end it suffices to show that if $f:\tV_1\to \tV$ is  a morphism and $D$ is a smooth divisor on $\tV$, such that 
$D_1:=f^{-1}(D)_{\rm red}$ is a smooth divisor and $\tV_1\setminus D_1\to V\setminus D$ is \'etale, then 
$f^*: F(\tV\setminus D)\to F(\tV_1\setminus D_1)$ induces a morphism 
$\underline{\omega}^{\CI}F(\tV, D)\to \underline{\omega}^{\CI}(\tV_1,D_1)$,
which is a consequence of  \cite[Theorem 4.2]{shujilog}.
Thus we obtain a sheaf $F_{\beta^{\rm rec}_{\log}}$ on $(X,\tX)_t$.
This construction yields a functor
\[\RSC_{X_\et}\to {\rm Sh}((X,\tX)_t), \quad F\mapsto F_{\beta^{\rm rec}_{\log}},\]
where $\RSC_{X_{\et}}=\RSC_{\Nis}\cap {\rm Sh}(X_{\et})$ is the category of reciprocity sheaves which are also \'etale sheaves on $X$.
\end{para}

\begin{lemma}\label{lem:tameRSC}
Assume $\tX$ is smooth and $\tX\setminus X$ is the support of a simple normal crossing divisor $D$. 
Let $F$ be a reciprocity sheaf which is also a sheaf of $\Z_{(p)}$-modules on $X_{\et}$ 
and assume that for all $\eta\in D^{(0)}$ the natural map 
\eq{lem:tameRSC1}{\underline{\omega}^{\CI}F(\sO_{\tX,\eta}^{h}, \fm_\eta^{-1})\xrightarrow{\simeq} 
F(K^{h}_\eta)\times_{F(K^{sh}_\eta)} \underline{\omega}^{\CI}F(\sO_{\tX, \eta}^{sh}, \fm_\eta^{-1})}
is an isomorphism, where $K_{\eta}^{(s)h}=\Frac(\sO_{\tX,\eta}^{(s)h})$.
 
 Then with the notation from \ref{para:tameRSC}
\[F_{\beta^{\rm rec}_{\rm log}}(X,\tX)= \underline{\omega}^{\CI}F(\tX, D).\]
\end{lemma}
\begin{proof}
We start by showing this $\supset$ inclusion. Let $a\in \underline{\omega}^{\CI}F(\tX,D)\subset F(X)$.
Let $(x,w,\epsilon)\in \Spa(X,\tX)$. It suffices to show that $a_{k(x)}$ the pullback of $a$ to $F(k(x))$ lies
in $\log F(w)$ as defined in \eqref{para:example-tame-sheaves4.1}. But for $v\in \Phi_w$ we have a natural map
$\Spec \sO_v^h\to \tX$ and it follows from \cite[Theorem 4.2]{shujilog}, that this map induces a morphism
$\underline{\omega}^{\CI}F(\tX,D)\to \underline{\omega}^{\CI}F(\sO_v^h, \fm_v^{-1})$. Hence $a_{k(x)}\in \log F(w)$.
Now the other inclusion. Let $a\in F_{\beta^{\rm rec}_{\log}}(X,\tX)$. 
By \cite[Corollary 2.5]{shujilog} it suffices to show that for each generic point $\eta$ of $D$ the pullback 
of $a$ to $F(K^h_{\eta})$, where $K=k(\tX)$, lies in $\underline{\omega}^{\CI}F(\sO_{\tX,\eta}^h, \fm_{\eta}^{-1})$.
Let $v$ be the discrete valuation on $K$ corresponding to  $\eta$, giving rise to the  point $(\eta, v,\epsilon)\in \Spa(X,\tX)$,
where $\epsilon:\Spec \sO_{\tX,\eta}\to \tX$ is the natural map.
By definition of $F_{\beta^{\rm rec}_{\log}}$ we find a finite tame extension $(L, w)/(K,v)$ such that 
\[a_{L}\in \log F_{w}=\log F(w)=\Ker \left(F(L)\to \frac{F(L_w^h)}{\underline{\omega}^{\CI}F(\sO_{w}^h, \fm_w^{-1})}\right),\]
where we use that $w$ is already a discrete  valuation of geometric type over $k$ on $L$ and hence 
$\Phi_w=\{w\}$. We have a finite morphism of pro-modulus pairs 
(see \cite[Lemma 3.8]{RulSaito}) $f:(\Spec\sO_w^{sh}, e\cdot s_w)\to 
(\Spec \sO_{\tX, \eta}^{sh},  \eta)$,
where $s_w$ denotes the closed point and $e$ 
is the ramification index, which is also equal to the degree of $f$. 
The transfer structure thus induces a pushforward
\[f_*: \underline{\omega}^{\CI}F(\sO_w^{sh}, \fm_w^{-1})
\subset \underline{\omega}^{\CI}F(\sO_w^{sh}, \fm_w^{-e})\to 
\underline{\omega}^{\CI}F(\sO_{\tX,\eta}^{sh}, \fm_\eta^{-1}).\]
By tameness,  $(e,p)=1$ and hence
\[a_{K_\eta^{sh}}= \tfrac{1}{e} f_*(a_{L_w^{sh}})\in 
\underline{\omega}^{\CI}F(\sO_{\tX,\eta}^{sh}, \fm_\eta^{-1}).\]
Thus $a_{K_{\eta}^h}\in \underline{\omega}^{\CI}F(\sO_{\tX,\eta}^h, \fm_{\eta}^{-1})$,
by \eqref{lem:tameRSC1}, which completes the proof.
\end{proof}

\begin{thm}\label{thm:tamecohRSC}
Let $F$ be a reciprocity sheaf and fix integers $i, d\geq 0$.
Assume for any normal crossing pair $(\tY,D)$ with $\dim\tY=d$,
\begin{enumerate}[label=(\alph*)]
    \item\label{thm:tamecohRSCa}  the assignment 
                 $U/\tY\mapsto \underline{\omega}^{\CI}F(U, D_{|U})$ defines an 
                 \'etale sheaf of $\Z_{(p)}$-modules on $\tY$, denoted by $F_{(\tY,D)}$,
    \item\label{thm:tamecohRSCb} 
    $H^i(\tY_{\et}, F_{(\tY,D)})=H^i(\tY_{\Nis}, F_{(\tY,D)})$.
\end{enumerate}
Assume $\tX$ is smooth of dimension $d$ and $D=\tX\setminus X$ is an SNCD.
Assume \ref{para:res}, \ref{res2} holds. Then 
    \[H^i((X,\tX)_t, F_{\beta_{\log}^{\rm rec}})= H^i(\tX_{\Nis}, F_{(\tX,D)}).\]
\end{thm}
\begin{proof}
Under \ref{res2} the modifications as in 
\ref{res2} are cofinal in the cofiltered category of all modifications to $\sX$,
see, e.g., the argument in the proof of 
\cite[\href{https://stacks.math.columbia.edu/tag/0AHI}{Tag 0AHI}]{stacks-project}.
Thus 
 \eq{thm:tamecohRSC1}{H^i((X,\tX)_t, F_{\beta_{\log}^{\rm rec}})
 =\varinjlim_{\sY\to \sX} 
 H^i(\tY_{\et}, j_*^{\sY}F_{\beta_{\log}^{\rm rec}})
= \varinjlim_{\sY\to \sX} H^i(\tY_{\Nis}, F_{(\tY,D)}),}
where  the direct limit is over all modifications $\sY=(X,\tY)\to \sX$ as in \ref{res2}.
The first equality holds by Theorem \ref{thm:tame-vs-etale}, 
and the second equality by assumptions \ref{thm:tamecohRSCa}, \ref{thm:tamecohRSCb}, and Lemma \ref{lem:tameRSC}.  Thus it remains to prove the colimit on the right is constant. 
Let $\pi:\tY\to \tX$ be a morphism as in  \ref{res2} and set $D':=\pi^{-1}(D)_{\rm red}$. 
In this case the pullback
 \[\pi^*: H^i(\tX_{\Nis}, F_{(\tX,D)})\to H^i(\tY_{\Nis}, F_{(\tY, D')})\]
 is an isomorphism by \cite[Theorem 5.2]{shujilog}, hence the statement.
\end{proof}

\begin{example}\label{exa:tamecohRSC}
The following reciprocity sheaves satisfy the assumptions \ref{thm:tamecohRSCa} and \ref{thm:tamecohRSCb} of Theorem \ref{thm:tamecohRSC}:
\begin{itemize}
\item $G$ a smooth unipotent commutative $k$-group, as in this case $\omega^{\CI}G_{(\tY, D_{\red})}=G_{\tY}$, as follows from 
           \cite[Theorem 5.2]{RulSaito};
    \item $\Omega^q_{/k}$, as follows from the description of $\underline{\omega}^{\CI}\Omega^q_{/k}$ in \cite[Corollary 6.8]{RulSaito}
    (if  ${\rm char}(k)=0$) and in \cite[Corollary 6.8]{RulSaito-AbbesSaito} (if ${\rm char}(k)\ge 0$);
    \item $W_n\Omega^q$, if ${\rm char}(k)=p>0$, as follows from 
         the description of $\underline{\omega}^{\CI}W_n\Omega^q_{(\tY, D_{\red})}$ 
         in \cite[Theorem 4.4]{mericicrys} or \cite[Theorem 5.4]{RenRuelling};
    \item $B_n\Omega^q= F^{n-1}dW_n\Omega^q$ and $Z_n\Omega^q=F^n W_{n+1}\Omega^q$ as follows from \cite[Corollary 6.16]{RenRuelling}
    (still assuming ${\rm char}(k)=p>0$).
\end{itemize}
Note that by \cite[Theorem 3.2.8]{CR11}  the cohomology $R\Gamma(\tX, \sO_{\tX})$ is a birational invariant 
and hence so is  $R\Gamma(\tX, W_n\sO_{\tX})$ (using \cite[Proposition 9.10]{BindaRuellingSaito}). 
Thus in this case the direct limit in \eqref{thm:tamecohRSC1} for general modifications
$\sY\to\sX$ stabilizes as soon as the following condition is satisfied for $d=\dim X$:
\begin{enumerate}[label={$({\rm res})_{d}$}]
    \item \label{res1} For $\tY$ an integral and separated finite type $k$-scheme of dimension $d$ 
    with a dense smooth open subscheme  $Y\subset \tX$, there exists a proper morphism 
    $f:\tilde{Z}\to \tY$, such that its restriction  $f^{-1}(Y)\xrightarrow{\simeq} Y$ is an isomorphism and $f^{-1}(\tY\setminus Y)_{\rm red}$ is an SNCD.
\end{enumerate}
If ${\rm char}(k)=0$, then \ref{res1} holds for all $d$ by Hironaka;
if ${\rm char}(k)=p>0$, then \ref{res1} holds for $0\le d\le 3$. 
For this first take a resolution $h: \tY_1\to \tY$,
such that $h^{-1}(Y)\xrightarrow{\simeq} Y$ is an isomorphism, 
see \cite[Theorem 1.1]{CossartPiltant-ResArith3folds}, then 
apply \cite[Theorem 1.4]{Cossart-Jannsen-Saito} (with $B=\emptyset$) to get a projective morphism
$g: \tY_2\to\tY$ such that $g^{-1}(Y)\xrightarrow{\simeq} Y$  is an isomorphism and
$g^{-1}(\tY\setminus Y)_{\rm red}= A\cup E$ with $E$ a simple normal crossing divisor 
and $A$ smooth  intersecting $E$ transversally, then blow-up once more in $A$ to get
$f:\tilde{Z}\to \tY$ as in \ref{res1}. 

In particular for $\tX$ as in Theorem \ref{thm:tamecohRSC} with $\dim X\le 3$ we get
\[R\Gamma((X,\tX)_t, W_n\sO^t)=R\Gamma(\tX_{\rm Zar}, W_n\sO),\]
a version of which was obtained in \cite[Theorem 13.7]{Huebner2021} for $n=1$ and where the left hand side is replaced by 
the cohomology of a discretely ringed adic space.    
\end{example}

 \section{Comparison with the tame site \texorpdfstring{$(X/S)_t$}{(X/S)t} from \texorpdfstring{\cite{HS2020}}{H\"ubner--Schmidt}}\label{sec:12}

Let $X\to S$ be a morphism of schemes.

\def\Xet{X_{\et}}

\begin{para}\label{para:HS}

Recall that the tame site $(X/S)_t$ from \cite[Definition 2.3]{HS2020} has the category of \'etale $X$-schemes 
as underlying category and 
a covering $\{U_i\to U\}_{i\in I}$ is an \'etale covering, such that  for every $(x,v,\epsilon_v)\in \Spa(U,S)$,
there exists $i\in I$ and $(y,w,\epsilon_w)\in \Spa(U_i,S)$ such that 
$(k(y),w)/(k(x),v)$ is tame. 
\end{para}

\begin{prop}\label{prop-def;XStHS}
Let $\tX\to S$ a proper morphism and  $X\hookrightarrow \tX$  be a quasi-compact open immersion of qcqs schemes.
Then there are adjoint functors\[
\begin{tikzcd}
\Sh((X,\tX)_t)\ar[r,"u_*"'] &\Sh((X/S)_t),\ar[l, shift right = 1, "u^*"']
\end{tikzcd}
\]
with $u^*$ exact, such that for $F\in \Sh((X,\tX)_t)$ and $U\in (X/S)_t$ affine, $u_*F(U) = F(U,\tU)$ for any choice of a Nagata compactification $U\inj \tU$ of $U\to \tX$, and for $G\in \Sh((X/S)_t)$, $u^*G$ is the sheafification of the presheaf $(U,\tU)\mapsto G(U)$. If $(U,\tU)\in (X,\tX)_{t}$ with $\tU\to \tX$ separated and 
universally closed, then $u^*G(U,\tU) = G(U)$, in particular $G\cong u_*u^*G$. 
\begin{proof}
 Consider the subsite $(X,\tX)_{t,suc}\subseteq (X,\tX)_{t}$ consisting of $(U,\tU)$ such that $\tU\to \tX$ is separated and universally closed. Consider the functor:\[
\Upsilon \colon (X,\tX)_{t,suc}\to (X/S)_t,\quad (U,\tU)\mapsto U. 
\]
Then \cite[\href{https://stacks.math.columbia.edu/tag/03A0}{Tag 03A0}]{stacks-project} gives an equivalence of topoi\[
\Upsilon^*\colon \Sh((X,\tX)_{t,suc})\cong \Sh((X/S)_t)\colon \Upsilon_*.
\]
Indeed: 
\begin{enumerate}
    \item $\Upsilon$ is cocontinuous. Let $(U,\tU)\in (X,\tX)_{t,suc}$ and let $\{f_i\colon V_i\to U\}$ be a covering. 
    Up to  refining we can assume that all $V_i$ are affine, therefore $V_i\to \tU$ are separated and of finite type. 
    Thus  we can choose a Nagata compactification $\tV_i$ giving a covering $\{(V_i,\tV_i)\to (U,\tU)\}$ in $(X,\tX)_{t,suc}$.
    \item $\Upsilon$ is continuous. If $\{(V_i,\tV_i)\to (U,\tU)\}$ is a tame covering in $(X,\tX)_{t,suc}$, then by the valuative criterion for universal closedness \cite[\href{https://stacks.math.columbia.edu/tag/01KA}{Tag 01KA}]{stacks-project}, $\{V_i\to U\}$ is a covering in $(X/S)_t$, and $\Upsilon$ respects fiber products by definition.
    \item Given maps in $(X,\tX)_{t,suc}$\[
    (f,\tf_1),(f,\tf_2)\colon (U',\tU')\to (U,\tU),
    \]
    consider the closure $\ol{U'}$ of $U'$ in $\tU'$. By Remark \ref{rmk-def;XSt}, $(id,\iota)\colon (U',\ol{U'})\to (U',\tU')$ is a covering in  $(X,\tX)_{t, suc}$ and $U'$ is an open dense in $\ol{U'}$. Since $\tU\to S$ is separated by assumption, the equalizer of $\tf_1 \iota$ and $\tf_2 \iota$ is closed in $\ol{U'}$ by \cite[\href{https://stacks.math.columbia.edu/tag/01KM}{Tag 01KM}]{stacks-project}, and since it contains $U'$ we conclude that $ \tf_1 \iota =  \tf_2 \iota$. 
    \item Given $(U,\tU),(V,\tV)\in (X,\tX)_{t,suc}$ and $c\colon U\to V$ \'etale, we can cover $U$ by affine schemes $U_i$, so that $U_i\to \tU$ and $U_i\to \tV$ are separated \'etale maps between qcqs schemes (hence again of finite type). Let $U_{i,1}$ and $U_{i,2}$ be Nagata compactifications of $U_i\to \tU$ and $U_i\to \tV$, respectively. Since the compositions $U_i\to U_{i,1}\to \tX$ and $U_i\to U_{i,2}\to \tX$ agree, this gives a map $\phi_{12}\colon U_i\to U_{i,1}\times_{\tX} U_{i,2}$.
    The map $\phi_{12}$ is quasi-finite. Indeed, it is separated of finite type since both $U_i\to U_{i,1}$ and $U_i\to U_{i,2}$ are quasi-compact open immersions, and it has isolated fibers. Then by Zariski's Main Theorem \cite[\href{https://stacks.math.columbia.edu/tag/05K0}{Tag 05K0}]{stacks-project} the map $\phi_{12}$ factors as a quasi-compact open immersion  $U_i\to \tU_i$ followed by a finite morphism $\tU_i\to U_{i,1}\times_{\tX} U_{i,2}$, therefore $\tU_i\to \tX$ is separated and universally closed. This gives a covering $\{(f_i,\tf_i)\colon (U_i,\tU_i)\to (U,\tU)\}$ together with a map $(c_i,\tc_i)\colon (U_i,\tU_i)\to (V,\tV)$ such that $c_i = c\circ f_i$.
    \item Given $V\in (X/S)_t$, we can cover $V$ by affine schemes $V_i$, so that $V_i\to \tX$ is a separated \'etale map between qcqs schemes. 
    By taking a Nagata compactification $\tV_i$ of $V_i\to \tX$ 
    we obtain a covering  $\{V_i=\Upsilon(V_i,\tV_i)\to V\}$  in $(X/S)_t$. 
\end{enumerate}
In particular, for every $G\in \Sh((X/S)_t)$ and $F\in \Sh((X,\tX)_{t,suc})$, there are unique sheaves 
$\tG:=\Upsilon_*G\in \Sh((X,\tX)_{t,suc})$ and $\tF:=\Upsilon^*F\in \Sh((X/S)_t)$ such that $\tG(U,\tU)=G(U)$, 
for $(U,\tU)\in (X,\tX)_{t,suc}$, and $\tF(U) = F(U,\tU)$, for $U\in (X/S)_t$ affine and any Nagata compactification $\tU$ of $U\to \tX$. Consider the inclusion $u'\colon (X,\tX)_{t,suc}\hookrightarrow (X,\tX)_t$, which is continuous by definition.
 We obtain  the adjoint functors\[
(u')^*\colon \Sh((X,\tX)_{t,suc})\leftrightarrows \Sh((X,\tX))\colon (u')_*.
\] 
The  functors in the statement are defined as $u_*:=\Upsilon^*(u')_*$ and $u^*:=(u')^*\Upsilon_*$. Since $\Upsilon^*$ is both left and right adjoint of $\Upsilon_*$, we obtain immediately that $u^*$ is left adjoint to $u_*$

We compute them. Let $F\in \Sh((X,\tX)_t)$ and let $U\in (X/S)_t$ be affine, then by definition $u_* F(U) = ((u')_*F)(U,\tU) = F(U,\tU)$ for any choice of a Nagata compactification $\tU$ of $U\to X\to \tX$, and for $G\in \Sh((X/S)_t)$, $u^*G = (u')^*\Upsilon_*G$ is the sheafification of the presheaf $(u')^p\widetilde{G}$ on $(X,\tX)_t$ defined by the rule 
\eq{up-G}{
(u')^p\widetilde{G}\colon (U,\tU)\in (X,\tX)_t\mapsto \colim_{\substack{(U,\tU)\to (V,\tV)\\ \tV\to \tX \textrm{sep. univ. closed}}} G(V).
}
Let $(V, \tV)\in (X,\tX)_{t,suc}$ and let $(g,\tg): (U,\tU)\to (V,\tV)$ be a morphism in $(X,\tX)_t$.
\begin{claim}\label{prop;def-XSTHS-claim}
   There is an affine open covering $\tU=\cup_i \tU_i$ and quasi-compact open immersions $\tU_i\to \ol{U_i}$ 
   such that  $(U_i:=U\cap \tU_i, \ol{U_i})\in (X,\tX)_{t,suc}$ and the composition
   $(U_i,\tU_i)\to (U,\tU)\to (V,\tV)$ factors via a morphism $(U_i,\ol{U_i})\to (V,\tV)$ in $(X,\tX)_{t,suc}$, for all $i$.
\end{claim}
Indeed, the statement is Zariski local on $\tX$ and as $\tU\xrightarrow{\tg} \tV$ is ift, 
we may 
assume that $\tU$ is affine and $\tg :\tU\to \tV$ factors as
$\tU\xrightarrow{\tg_1} \tV' \xrightarrow{\tg_2} \tV$ with $\tV'$ affine, $\tg_1$ integral, and $\tg_2$ of finite type.
Considering a Nagata compactification $\ol{\tW}\to \tV$ of this latter morphism, 
we see that $\tU\to \tV$ factors  as a composition
\[\tU\xrightarrow{h} \tW\xrightarrow{j}\ol{\tW}\xrightarrow{f}\tV,\]
with $h$ an integral morphism, $j$ a quasi-compact open immersion, and $f$ a proper morphism. 
By Lemma \ref{lem:pro-zmt}, we find $\ol{\tU}$, a quasi-compact open immersion $U\to \tU\to \ol{\tU}$ and 
$\ol{\tU} \to \ol{\tW} \to \tV$ a composition of an integral and a proper morphism, so separated and universally closed. Therefore $(U,\ol{\tU})\in (X,\tX)_{t,suc}$ and this proves the claim.

We conclude as follows. Consider the presheaf\[
G'\colon (U,\tU)\in (X,\tX)_t\mapsto G(U).
\]
By construction, there exists a map $(u')^p\tG\to G'$.
By what we observed before, for all $(U,\tU)\to (V,\tV)$ with $\tV\to \tX$ separated and universally closed there exists a covering $(U_i,\tU_i)$ of $(U,\tU)$ such that the following diagram commutes:
\[
\begin{tikzcd}
    G(V)\ar[r]\ar[d] &(u')^p\tG(U,\tU)\ar[r]\ar[d] &G'(U,\tU)\ar[d]\\
    G(U_i)\ar[r]\ar[rr,bend right = 20, equal] &(u')^p\tG(U_i,\tU_i)\ar[r]&G'(U_i,\tU_i).
\end{tikzcd}
\]
Let $a\in (u')^p\tG(U,\tU)$ and let $(U,\tU)\to (V,\tV)$ as in \eqref{up-G} such that there exists $b\in G(V)$ that maps to $a$, and let $\{(U_i,\tU_i)\}$ be a covering such that a diagram as above exists. Assume that $a\mapsto 0$ in $G'(U,\tU)$, then we conclude that $b_{|U_i}=0$, so $a_{|(U_i,\tU_i)}=0$, which implies that the map $(u')^p\tG\to G'$ is injective after sheafification. Let now $b\in G(U) = G'(U,\tU)$. By applying the above diagram with $(V,\tV)=(X,\tX)$, there exists a covering $\{(U_i,\tU_i)\}$ such that $b_{U_i}$ is in the image of $(u')^p\tG(U_i,\tU_i)\to G'(U_i,\tU_i)$. Hence the map $(u')^p\tG\to G'$ is surjective after sheafification.
Moreover, if $\tU\to \tX$ is separated and universally closed, then $G'$ satisfies descent on $(U,\tU)$. Indeed, 
if $\{(V_i,\tV_i)\to (U,\tU)\}$ is a covering in $(X,\tX)_t$ , then $\{V_i\to U\}$
is a tame covering in $(X/S)_t$, as we have bijections $\Spa(U,\tU)\leftrightarrow \Spa(U,\tX)\leftrightarrow \Spa(U,S)$ given by the valuative criterion for separatedness and universal closedness \cite[\href{https://stacks.math.columbia.edu/tag/01KA}{Tag 01KA} and \href{https://stacks.math.columbia.edu/tag/01KY}{Tag 01KY}]{stacks-project}. Therefore in this case 
we have $u^*G(U,\tU)=G'(U,\tU)=G(U)$. 
Finally, for $U\in (X/S)_t$ affine, the map $U\to \tX$ is separated of finite type so it always admits a Nagata compactification $\tU$, and we have $u_*u^*G(U)=u^*G(U,\tU)=G(U)$.
This implies $G\cong u_*u^*G$.
\end{proof}

\end{prop}
\begin{remark}\label{rmk:u-lower-star}
\begin{enumerate}[label=(\arabic*)]
\item
    Notice that for $G\in \Sh((X/S)_t)$ the presheaf $(U,\tU)\mapsto G(U)$ on $(X,\tX)_t$ does not necessarily have descent if $\tU\to \tX$ is not proper and $G$ is not an \'etale sheaf on $X$. Indeed, let $S$ be a scheme with $p\cO_S=0$ and consider $(\A^1_S,\A^1_S)$ in $(\A^1_S,\P^1_S)_t$. Then, writing $\A^1_S=\Spec\cO_S[x]$, the Artin--Schreier cover $C\to \A^1_S$ defined by $t^p-t=x$ gives a covering $(C,C)\to (\A^1_S,\A^1_S)$ in $(\A^1_S,\P^1_S)_t$, but it is not a tame covering in $(\A^1_S/S)_t$.
    \item
    Consider the sheaf $\sO^t\in \Shv((X,\tX)_t)$ from Lemma \ref{lem:Gat}. 
        Then, 
\[u_*\cO^t(U)=\sO^t(U,\ol{U})=\sO(\ol{U})=\cO(T)\qfor U\in \Xet,\]
where $U\inj \ol{U}$ is the integral closure in $U$ of a Nagata compactification of $U\to \tX$, and 
$\ol{U}\to T \to \tX$ is the Stein factorization of the universally closed morphism $\ol{U}\to \tX$.

\item
For $F\in \Sh((X,\tX)_t)$, the map $u^*u_*F\to F$ is not an isomorphism in general. If that was the case, then for $(U,\tU)=\lim (U_{\lambda},\tU_{\lambda})\in \widetilde{(X,\tX)_t}$ tamely acyclic, the above computation gives that\[
u^*u_*\cO^t(U,\tU)=u^pu_*\cO^t(U,\tU) = \colim_{\lambda} \cO(T_{\lambda})
\] 
where $T_{\lambda}\to \tX$ is integral coming from the Stein factorization of $\ol{U_{\lambda}}$ as above. Therefore, this would imply that the map $\cO(\tX)\to \cO(\tU)$ is integral, which does not  hold in general.
      This implies that there is no object $G\in \Shv((X/S)_t)$ such that $u^*G=\sO^t$. 
      Else we would have by Proposition \ref{prop-def;XStHS} that $G=u_*u^*G=u_*\cO^t$, hence 
      $\cO^t=u^*G=u^*u_*\cO^t$.
       \end{enumerate}
\end{remark}
\medbreak

\begin{thm}\label{thm:HS-comparison}
    Let $S$ be an affine scheme, let $\tX\to S$ be a proper morphism and  $X\inj \tX$ be a quasi-compact open immersion.
    Consider the adjoint functors $u^*\dashv u_*$ of Proposition \ref{prop-def;XStHS}. Then for all $G\in \Sh((X/S)_t)$ we have \[
    R\Gamma((X/S)_t,G) \simeq R\Gamma((X,\tX)_t,u^*G).
    \]
\end{thm}
\begin{proof}
    We need to show that $R^qu_*u^*G = 0$ for $q>0$. Indeed, if this holds, we have that the map $G\to Ru_*u^*G$ is an equivalence, therefore\[
    R\Gamma((X/S)_t,G) \simeq R\Gamma((X/S)_t,Ru_*u^*G)\simeq R\Gamma((X,\tX)_t,u^*G).
    \]
    Let $(\ol{x},\ol{v})$ be a tame point of $X$ in the sense of \cite[Definition 2.7]{HS2020}, then\[
    (R^qu_*u^*G)_{(\ol{x},\ol{v})} = \colim_{U\to X} H^q((U,\tU)_t,u^*G),
    \]
    where the colimit on the right hand side is indexed along all \'etale neighborhoods $U\to X$ of $\ol{x}$ which are tame at $(\ol{x},\ol{v})$, and $\tU\to \tX$ a choice of a Nagata compactification of $U\to \tX$.
    We can always suppose that $U$ is affine, therefore the map $U\to S$ is quasi-projective. By Chow's lemma and the control of the birational locus of Raynaud and Gruson (see \cite[Corollary 2.6]{ConradNagata}), we can choose $\tU$ quasi-projective over $S$.
    Hence $\tU$ satisfies the property that every finite set of points is contained in an affine open. Thus  
    by Theorem \ref{thm:cech-comp} \[
    \colim_{U\to X} H^q((U,\tU)_t,u^*G) \cong \colim_{U\to X} \check H^{q}((U,\tU)_t,u^*G),
    \]
    Recall that $u^*G$ is the sheafification of the presheaf $(V,\tV)\mapsto G(V)$. Since \v Cech cohomology of presheaves is invariant under taking sheafification by Theorem \ref{thm:cech-comp}, we have that\[
    \colim_{U\to X} \check H^{q}((U,\tU)_t,u^*G) \cong 
    \colim_{(V,\tV)\to (U,\tU)}\colim_{U\to X} \pi_{-q}G(V^{\times \bullet}),
    \]
    where we can suppose that $(V,\tV)\to (U,\tU)$ ranges over all coverings in $(X,\tX)_{{\rm affine},t}$ by Lemma \ref{lem:reduction-affine}, and moreover by Claim \ref{prop;def-XSTHS-claim} that factors through $(V,\ol{V})\to (U,\tU)$ in $(X,\tX)_{\rm suc}$.
    Since $\pi_{-q}G(V^{\times \bullet})$ does not depend on $\tV$, we can in fact suppose that $\tV=\ol{V}$, therefore we have that\[
    \colim_{U\to X} \check H^{q}((U,\tU)_t,u^*G)\cong \colim_{(V,\ol{V})\to (U,\tU)}\colim_{U\to X} \pi_{-q}G(V^{\times \bullet}),
    \]
    where $(V,\tV)\to (U,\tU)$ ranges over all coverings in $(X,\tX)_{suc,t}$ with $V$ affine. By the valuative criterion (as in the proof of Proposition \ref{prop-def;XStHS}) this implies that $V\to U$ is a covering in $(X/S)_t$, and every covering $V\to U$ in $(X/S)_t$ with $V$ affine induces a covering $(V,\ol{V})\to (U,\tU)$ in $(X,\tX)_{suc,t}$ by the choice of a Nagata compactification, therefore \[
    \colim_{(V,\ol{V})\to (U,\tU)}\colim_{U\to X} \pi_{-q}G(V^{\times \bullet})\cong  \colim_{V\to U}\colim_{U\to X} \pi_{-q}G(V^{\times \bullet}),
    \]
    where on the right $\{V\to U\}$ is the cofiltered set of coverings in $(X/S)_t$ and $\{U\to X\}$ is the cofiltered set of tame neighborhoods of $(\ol{x},\ol{v})$, therefore by \cite[Theorem 7.16]{HS2020} it is\[
    \colim_{U\to X} \check{H}^q((U/S)_t,G) = H^q((\Spec(\cO_{X,(\ol{x},\ol{v})}^t)/S)_t,G) = 0.
    \]
\end{proof}

\begin{remark}\label{rmk:tame-vs-tame}
We observe that for $F\in \Sh((X,\tX)_t)$, the cohomology groups
\[H^i((X,\tX)_t, F)\quad \text{and}\quad 
H^i((X/S)_t, u_*F)=H^i((X,\tX)_t, u^*u_*F))\]
are in general not equal, even if $\tX\to S$ is proper and $F$ is constructed as in \eqref{para:constr-tame-sheaves}. 
Indeed, assume $S=\Spec(k)$ with $k$ an algebraically closed field of characteristic $p>0$ 
and assume  $\dim X\le 3$. Take $F=\sO^t$.
We get from Remark \ref{rmk:u-lower-star} that $u_*\sO^t$ agrees with the constant sheaf $k$ on $(X/S)_t$, so we have by
Theorem \ref{thm:tame-vs-etale}
\[H^i((X/S)_t, u_*\sO^t)= \varinjlim_{\tY} H^i(\tY_{\et}, j_*^{(Y,\tY)}u^*k),\]
where the direct limit is over all smooth compactifications of $X$. By construction, $j_*^{(Y,\tY)}u^*k$ agrees with the constant \'etale sheaf $k$ on $\tY$.
Furthermore we have 
\eq{rmk:tame-vs-tame1}{H^i(\tY_{\et}, k)=H^i(\tY_{\et}, \Z/p)\otimes_{\F_p} k.}
For this we observe, that for any étale sheaf $G$ on $\tY$, the presheaf $U\mapsto G(U)\otimes_{\F_p} k$ is already a sheaf,
since we can identify its sections as an abelian group with a direct sum over a basis of $k/\F_p$,
which  is a sheaf as the site $\tY_{\et}$ 
is noetherian. It also follows that the functor $-\otimes k: \Shv(\tY_{\et},\F_p)\to \Shv(\tY_{\et}, k)$ is exact and maps
$\Gamma$-acyclic sheaves to $\Gamma$-acyclic sheaves. This gives \eqref{rmk:tame-vs-tame1}. 
Furthermore, we remark that $Z\mapsto H^i(Z_{\et},\Z/p)$ is a birational invariant on smooth proper schemes.
Indeed this follows from the fact that proper correspondences act on 
\[H^i(Z_{\et}, \Z/p)= H^i(Z_{\rm Zar}, [\sO_Z\xrightarrow{{\rm Frob}-{\rm id}}\sO_Z]),\]
by \cite[9.11]{BindaRuellingSaito} and the fact that the action of the closure of a graph of a birational 
(maybe not everywhere defined) map $Z\dashrightarrow Z'$ induces an isomorphism $H^i(Z, \sO_Z)\xrightarrow{\simeq} H^i(Z', \sO_{Z'})$,
by \cite[Theorem 1]{CR11}. Thus altogether we find
\eq{rmk:tame-vs-tame2}{H^i((X/S)_t, u_*\sO^t)= H^i(\tX_{\et}, \Z/p)\otimes_{\F_p} k,}
where $\tX$ is some smooth compactification of $X$. 
On the other hand by Remark \ref{rmk:Ot} and the birational invariance of $\sO$-cohomology on smooth proper schemes we have 
\eq{rmk:tame-vs-tame3}{H^i((X,\tX)_t, \sO^t)= H^i(\tX, \sO_{\tX}),}
The groups \eqref{rmk:tame-vs-tame2} and \eqref{rmk:tame-vs-tame3} are equal for all $i$ if and only if the  Frobenius
acts bijectively on $H^i(\tX, \sO_{\tX})$. This is, e.g., not the case if $\tX$ is a supersingular abelian variety.    
\end{remark}
We conclude the section with a computation of the cohomology of $u_*F_{\beta^{\rm rec}_{\log}}$ for $F$ a reciprocity sheaf,
see \eqref{para:tameRSC} for notation.

\begin{lemma}\label{lem:tameRSCA1}
Let $k$, $X\subset \tX$, $D$, and $F$ be as in Lemma \ref{lem:tameRSC} above.
Assume additionally that $\tX$ has dimension $d$ and that \ref{res1} holds (e.g. $d\le 3$).
Denote by $h^0_{\A^1}(F)$ the maximal $\A^1$-invariant subsheaf with transfers  of $F$, which is a Nisnevich sheaf 
(e.g. \cite[4.30]{RulSaito}), and by $h^{0,t}_{\A^1}(F)$ the presheaf on $(X/k)_t$ given by 
$V\mapsto h^0_{\A^1}(F)(V)$. Then 
\[u_*(F_{\beta^{\rm rec}_{\log}})=h^{0,t}_{\A^1}(F),\]
where $u_* : \Shv((X,\tX)_t)\to \Shv((X/k)_t)$ is the functor from Proposition \ref{prop-def;XStHS}.
In particular,  $h^{0,t}_{\A^1}(F)$ is a sheaf on $(X/k)_t$ and if $\tX$ is additionally proper, we get 
\[R\Gamma((X/k)_t, u_*(F_{\beta^{\rm rec}_{\log}}))= R\Gamma((X/k)_t, h^{0,t}_{\A^1}(F))=
R\Gamma((X,\tX)_t, u^*h^{0,t}_{\A^1}(F)).\]
\end{lemma}
\begin{proof}
The second statement follows directly from the first and Theorem \ref{thm:HS-comparison}.
As $h^0_{\A^1}(F)$ is a Zariski sheaf it suffices to show the equality of the first statement when evaluated on 
an affine étale $X$-scheme $V$. 
By \ref{res1} we find a smooth compactification $\ol{V}$ of $V$ with 
simple normal crossing complement $D_V=(\ol{V}\setminus V)_{\red}$. 
By Proposition \ref{prop-def;XStHS} together with Lemma \ref{lem:tameRSC} we have 
\eq{lem:tameRSCA1-1}{u_*(F_{\beta^{\rm rec}_{\log}})(V)= F_{\beta^{\rm rec}_{\log}}(V,\ol{V})= \ul{\omega}^{\CI}F(\ol{V}, D_V).}
By \cite[Theorem 4.2]{shujilog}  for each map $\Spec \sO_v^h\to \ol{V}$, 
with $\sO^h_v$ a henselian discrete valuation ring of geometric type over $k$,  
we get an  induced morphism
$\underline{\omega}^{\CI}F(\ol{V},D_V)\to \underline{\omega}^{\CI}F(\sO_v^h, \fm_v^{-1})$.
Therefore, the right hand side of \eqref{lem:tameRSCA1-1} is precisely the subset of $F(V)$ consisting of  
the sections with all the motivic conductors $\le 1$. Thus by \cite[Corollary 4.32]{RulSaito}, 
\eq{lem:tameRSCA1-2}{\ul{\omega}^{\CI}F(\ol{V}, D_V)= h^0_{\A^1}(F)(V),}
which yields the statement.
\end{proof}

\begin{remark}\label{rmk:h0-tame-vs-et}
Let $X$, $\tX$  and $F$ be as above. 
If $h^0_{\A^1}(F)$ happens to be an étale sheaf on $X$, then  the presheaf
$(V,\tV)\mapsto h^0_{\A^1}(F)(V)$ is also a sheaf on $(X,\tX)_t$ and  is equal to $u^*h^{0,t}_{\A^1}(F)$, 
by Proposition \ref{prop-def;XStHS}.

Note however that $h^0_{\A^1}(F)$ will  in positive characteristic $p$ in general not be an étale sheaf even if $F$ is.
For example consider $F=\Omega^q$ and assume \ref{res1}, with $d=\dim X$.
By \cite[(4) on p. 5]{RulSaito}  and \eqref{lem:tameRSCA1-2} the Nisnevich sheaf $h^0_{\A^1}(F)$ on $X$ is given by
\[V\mapsto H^0(\ol{V}, \Omega^q_{\ol{V}/k}(\log D)),\]
where $(\ol{V}, D)$ is a normal crossing compactification of $V$. 
It is well known that this is not an \'etale sheaf. Indeed,  
consider any form $\alpha\in \Omega^1_{k(x,y)}$, which has not only log-poles, e.g. $\alpha=dx/y$. 
By \cite[I, Theorem]{Mumford61} and resolution of singularities in dimension 2, 
there is a dense open $V\subset \P^2$ and a projective morphism $f:X\to \P^2$ from a smooth  $k$-scheme $X$, 
such that  $\alpha \in H^0(V, \Omega^1_V)$, the induced map $f^{-1}(V)\to V$ is finite \'etale, 
$D=f^{-1}(\P^2\setminus V)_{\red}$ is an SNCD, and $f^*\alpha\in H^0(X, \Omega^1_X)\subset H^0(X, \Omega^1_X(\log D))$. 
In this case clearly the pullbacks of $f^*\alpha$  to any normal crossing  model of 
$f^{-1}(U)\times_U f^{-1}(U)$ along the maps induced by the two projections agree, but it does not descent to a logarithmic form on a normal crossing model of $U$ by our choice of $\alpha$.
\end{remark}

\bibliographystyle{alpha}
\bibliography{bib-1}

\end{document}